\pdfoutput=1
\RequirePackage{ifpdf}
\ifpdf 
\documentclass[pdftex]{sigma}
\else
\documentclass{sigma}
\fi

\newtheorem{thm}{Theorem}[section]
\newtheorem{prop}[thm]{Proposition}
\newtheorem{cor}[thm]{Corollary}
\newtheorem{lem}[thm]{Lemma}
{ \theoremstyle{definition}
\newtheorem{defn}[thm]{Definition}
\newtheorem{rem}[thm]{Remark}
\newtheorem{ex}[thm]{Example}

}

\numberwithin{equation}{section}

\def\bG{{\mathbb G}}

\def\C{{\mathbb C}}
\def\F{{\mathbb F}}
\renewcommand{\H}{{\mathbb H}}
\def\N{{\mathbb N}}
\renewcommand{\P}{{\mathbb P}}
\def\Q{{\mathbb Q}}
\def\R{{\mathbb R}}
\def\Z{{\mathbb Z}}
\def\K{{\mathbb K}}

\def\cA{{\mathcal A}}
\def\cB{{\mathcal B}}
\def\cC{{\mathcal C}}
\def\cD{{\mathcal D}}
\def\cE{{\mathcal E}}
\def\cF{{\mathcal F}}
\def\cG{{\mathcal G}}
\def\cH{{\mathcal H}}
\def\cI{{\mathcal I}}
\def\cJ{{\mathcal J}}

\def\cM{{\mathcal M}}
\def\cN{{\mathcal N}}

\def\cP{{\mathcal P}}

\def\cR{{\mathcal R}}
\def\cS{{\mathcal S}}

\def\cU{{\mathcal U}}

\def\cW{{\mathcal W}}

\def\Aut{{\rm Aut}}

\def\Hom{{\rm Hom}}

\def\Spec{{\rm Spec}}

\def\Tr{{\rm Tr}}

\begin{document}
\allowdisplaybreaks

\newcommand{\arXivNumber}{1907.13545}

\renewcommand{\thefootnote}{}

\renewcommand{\PaperNumber}{038}

\FirstPageHeading

\ShortArticleName{Quantum Statistical Mechanics of the Absolute Galois Group}

\ArticleName{Quantum Statistical Mechanics\\ of the Absolute Galois Group\footnote{This paper is a~contribution to the Special Issue on Integrability, Geometry, Moduli in honor of Motohico Mulase for his 65th birthday. The full collection is available at \href{https://www.emis.de/journals/SIGMA/Mulase.html}{https://www.emis.de/journals/SIGMA/Mulase.html}}}

\Author{Yuri I.~MANIN~$^{\dag^1}$ and Matilde MARCOLLI~$^{\dag^2\dag^3\dag^4}$}

\AuthorNameForHeading{Yu.I.~Manin and M.~Marcolli}

\Address{$^{\dag^1}$~Max Planck Institute for Mathematics, Bonn, Germany}
\EmailDD{\href{mailto:manin@mpim-bonn.mpg.de}{manin@mpim-bonn.mpg.de}}

\Address{$^{\dag^2}$~California Institute of Technology, Pasadena, USA}
\EmailDD{\href{mailto:matilde@caltech.edu}{matilde@caltech.edu}}
\Address{$^{\dag^3}$~University of Toronto, Toronto, Canada}
\Address{$^{\dag^4}$~Perimeter Institute for Theoretical Physics, Waterloo, Canada}

\ArticleDates{Received August 01, 2019, in final form April 15, 2020; Published online May 05, 2020}

\Abstract{We present possible extensions of the quantum statistical mechanical formulation of class field theory to the non-abelian case, based on the action of the absolute Galois group on Grothendieck's dessins d'enfant, the embedding in the Grothendieck--Teichm\"uller group, and the Drinfeld--Ihara involution.}


\Classification{11G32; 11M55; 11R32; 82B10; 82C10; 58B34}

\renewcommand{\thefootnote}{\arabic{footnote}}
\setcounter{footnote}{0}

\section{Introduction}\label{IntroSec}

\subsection{Abelian Galois theory via quantum statistical mechanics}

The Bost--Connes system \cite{BC} provides the original model for
the recasting of explicit class field theory problems in the setting of classical or
quantum statistical mechanics.

Basically, it starts with classical Mellin transforms
of Dirichlet series of various arithmetic origins, and represents them as
calculations of statistical averages of certain observables
as functions of inverse temperature $1/kT$ (classical systems)
or imaginary time $it$ (quantum systems).
Here is a brief description of such systems relevant in the contexts of abelian extensions
of number fields
(see \cite{BC} and \cite[Chapter~3]{CoMa-book}).

One constructs the arithmetic noncommutative algebra $\cA_\Q=\Q[\Q/\Z]\rtimes \N$ given by
a semigroup crossed product. One then passes to the associated $C^*$-algebra of observables with
the time evolution $(\cA,\sigma_t)$. The algebra is $\cA=C^*(\Q/\Z)\rtimes \N = C\big(\hat \Z\big)\rtimes \N$. The time evolution acts trivially on the abelian subalgebra, $\sigma_t|_{C^*(\Q/\Z)}={\rm Id}$, and nontrivially on the semigroup~$\N$: $\sigma_t (\mu_n)=n^{{\rm i}t} \mu_n$, where $\mu_n$ are the isometries in $\cA$ corresponding to $n\in \N$. The group $G=\hat\Z^*$ acts as symmetries of the quantum statistical mechanical system $(\cA,\sigma_t)$. There are covariant representations $\pi\colon \cA \to \cB(\cH)$ on the Hilbert space $\cH=\ell^2(\N)$, with ${\rm e}^{{\rm i}t H} \pi(a) {\rm e}^{-{\rm i}t H} = \pi(\sigma_t(a))$, for $a\in \cA$ and $t\in \R$, with Hamiltonian~$H$, for which the partition function $Z(\beta)=\Tr\big({\rm e}^{-\beta H}\big)=\zeta(\beta)$ is the Riemann zeta function. The extremal KMS equilibrium states $\varphi\in \cE_\infty$ at zero temperature, when evaluated on elements of $\cA_\Q$, take values in the maximal abelian extension $\Q^{ab}$ (generated by roots of unity) and intertwine the symmetries and the Galois action, ${\rm \theta}(\gamma) \varphi(a) = \varphi(\gamma a)$, for $\gamma\in \hat\Z^*={\rm GL}_1\big(\hat\Z\big)$ and $\varphi\in \cE_\infty$, with $\theta\colon \hat\Z^* \to {\rm Gal}\big(\Q^{ab}/\Q\big)$ the class field theory isomorphism.

An interested reader will find much more details and basic references in \cite[Chapter~3, Section~2]{CoMa-book}.

This construction of the Bost--Connes system was then generalized in \cite{CoMaRa, HaPa,LLN,Yalk}, to the case of abelian extensions of arbitrary number fields. Given a number field~$\K$, there is an associated quantum statistical mechanical system ${\rm BC}_\K$ with the properties that its partition function is the Dedekind zeta function of~$\K$, and its symmetry group is the Galois group of the maximal abelian extension ${\rm Gal}\big(\K^{ab}/\K\big)$. Moreover, there is an arithmetic subalgebra $\cA_\K$ with the property that evaluations of zero-temperature KMS states on elements of $\cA_\K$ take values in $\K^{ab}$, and
intertwine the action of ${\rm Gal}(\K^{ab}/\K)$ by symmetries of the quantum statistical
mechanical system with the Galois action on $\K^{ab}$.

Later, it was shown in \cite{CdSLMS, CLMS} that the quantum statistical mechanical system ${\rm BC}_\K$ completely determines the number field~$\K$. Other generalizations of the Bost--Connes system were developed for the abelian varieties related to 2-lattices and the Shimura variety of ${\rm GL}_2$
(see~\cite{CoMa}), and for more general Shimura varieties in~\cite{HaPa} and abelian varieties in~\cite{Yalk2}. See also \cite[Chapter~3]{CoMa-book} for an overview of these arithmetic quantum statistical mechanical models.

\subsection{Non-abelian Galois theory via quantum statistical mechanics}

A natural further generalization of this program is a development of quantum statistical mechanical models for the action of the non-abelian absolute Galois group ${\rm Gal}\big(\bar\Q/\Q\big)$ instead of its abelianization ${\rm Gal}\big(\bar\Q/\Q\big)^{ab}={\rm Gal}\big(\Q^{ab}/\Q\big)$ as in the original Bost--Connes system. This development is the main focus of the present paper. We will
approach the problem from the point of view of Belyi's theorem and Grothendieck's theory of
{\em dessins d'enfant}, see~\cite{Groth} and further articles in the same collection. See also the
recent work~\cite{Boa,Fresse}.

More concretely, whereas ${\rm Gal}\big(\bar\Q/\Q\big)^{ab}$ can be presented in terms of the profinite
completion of the fundamental group of $\P^1\setminus \{0,\infty\}$, according to Grothedieck's prophetic vision, ${\rm Gal}\big(\bar\Q/\Q\big)$ might be similarly presented in terms the profinite completion of the fundamental group (or rather groupoid) of $\P^1\setminus \{0, 1,\infty\}$. A decisive step in this direction was made by G.~Belyi~\cite{Be}. The semigroup $\N$ of the classical Bost--Connes system is replaced here by the semigroup of Belyi-extending maps of~\cite{Wood}.

The relation to the profinite fundamental groups mentioned above can be
seen more explicitly in the following way. The abelianization ${\rm Gal}\big(\bar\Q/\Q\big)^{ab}$ is isomorphic by the class field theory isomorphism to $\hat\Z^*$ and acts by
automorphisms on the profinite completion $\hat \Z$ of the fundamental group
of $\P^1\setminus \{0,\infty\}$. On the other hand, the absolute Galois
group ${\rm Gal}\big(\bar\Q/\Q\big)$ has a faithful action by automorphisms on
the profinite completion $\hat F_2$ of the fundamental group of $\P^1\setminus \{0, 1,\infty\}$, as in of \cite[Proposition~1.6]{Ihara} and Theorem~1 in the Appendix to the same paper. This action is given
on a set of generators $x$, $y$ of $\hat F_2$ by $x\mapsto x^{\chi(\gamma)}$ and $y\mapsto f_\gamma^{-1}(x,y) y^{\chi(\gamma)} f_\gamma(x,y)$, with $\chi\colon {\rm Gal}\big(\bar\Q/\Q\big) \to \hat\Z^*$ the cyclotomic character and $f_\gamma(x,y)$ are elements in the commutator subgroup of $\hat F_2$ \cite[Proposition~1.5]{Ihara}. These two actions are clearly compatible through the cyclotomic character, since the action $\alpha\in \hat\Z^*$ on the profinite fundamental group of $\P^1\setminus \{0,\infty\}$ is just given by $x\mapsto x^\alpha$.

Belyi's theorem, Grothendieck's program, and a construction of the respective versions of Bost--Connes system will be described in the Section~\ref{DessinsSec} of this paper. This section, however, is mainly concerned with the combinatorics and enumeration of dessins d'enfant, rather than Galois action upon the fundamental groupoid of $\P^1\setminus \{0, 1,\infty\}$. However, see Section~\ref{HopfGaloisSec} and the following discussions, making explicit the fact that not all our constructions are purely combinatorial ones that avoid an explicit use of Galois action upon dessins.

\looseness=1 Our main motivation here is the fact that various Bost--Connes systems
 are practically equivalent to the relevant enumerations of dessins d'enfant, and their
 behaviour formally is determined by asymptotic properties of the relevant enumerations and
 their generating functions.

The broad idea behind our generalization of the Bost--Connes system from the
abelianized to the absolute Galois group can be summarized as obtained by replacing
the Hopf algebra~$\Q[\Q/\Z]$ on roots of unity by a suitable Hopf algebra on dessins,
and the action of the power maps separating different orbits of roots of unity by the
action of a semigroup of Belyi-extending maps. However, expecting that the resulting
quantum statistical mechanical system obtained in this way will suffice to
separate the different Galois orbits of dessins is probably too strong an expectation,
at least in its present form. We will show that one can see some of the new Galois
invariants of dessins introduced in \cite{Wood} occurring in the low temperature KMS
states of the system, but some technical restrictions on the choice of the Belyi-extending
maps, which will be specified in Section~\ref{DessinsSec}, limit the effectiveness of
these invariants at separating Galois orbits. It is possible that
using constructions from \cite{Boa,Fresse}, one may be able to obtain
a more refined version of the quantum statistical mechanical system presented here
that bypasses this limitation.

The Hopf algebra formalism of dessins we discuss in Section~\ref{DessinsSec} makes it possible,
in principle, to obtain new Galois invariants of dessins from known ones, by applying
the formalism of Birkhoff factorization, so that the new invariants depend not only
on a dessin and its sub-dessins, but also on associated quotient dessins. We illustrate
this principle in Proposition~\ref{newGalois}
in the case of Tutte and Bollob\'as--Riordan--Tutte polynomials.

Section~\ref{IharaSec} is dedicated to the absolute Galois action and its (partial) descriptions.
Such descriptions are centred around various versions of the {\it Grothendieck--Teichm\"uller group}.
We start with its formal definition in the Section~\ref{IharaBCsec}, where this group
appears together with its action upon the Bost--Connes algebra and upon the relevant
quantum statistical mechanical system. In this environment, the Bost--Connes statistical mechanics
becomes much more sensitive to the arithmetics of the absolute Galois group, through its
relation to the Grothendieck--Teichm\"uller group, rather than through its action on dessins
as in the choices surveyed in Section~\ref{DessinsSec}.

The last Section~\ref{section3.10} presents, following \cite{ComMan2}, one of the avatars ${\bf mGT}$ of the
Grothendieck--Teichm\"uller group as the symmetry group of the genus zero modular operad introduced in~\cite{KoMa} and much studied afterwards.

Section~\ref{section4} discusses the relation between the algebras of our quantum statistical mechanical systems and Drinfeld's quasi-triangular quasi-Hopf algebras~\cite{Dri3, Dr}. In particular, we show that both the Bost--Connes algebra and our Hopf algebra of dessins d'enfant have an associated direct system of quasi-triangular quasi-Hopf algebras obtained using Drinfeld's twisted quantum double construction~\cite{Dri2}.

\section{Dessins and dynamics}\label{DessinsSec}

Below we will use (with certain laxity) the language of graphs as it was described in~\cite{BoMa}.

A finite graph $\tau$ is a quadruple of data $(V_\tau, F_\tau, \delta_\tau, j_\tau)$ of finite sets $V_{\tau}$ of {\em vertices} and $F_{\tau}$ of {\em flags}, a map $\partial_{\tau}\colon F_{\tau} \to V_{\tau}$, and an involution $j_{\tau}\colon F_{\tau}\to F_{\tau}$: $j_{\tau}^2={\rm id}$. For each vertex $v\in V_{\tau}$, the graph with flags $\partial_{\tau}^{-1}(v)$ and trivial involution~$j$ is called {\em the corolla} of~$v$ in~$\tau$. If a graph has just one vertex, it is called a corolla itself.

A {\em geometric realisation} of graph $\tau$ is a topological space of the following structure.
If $\tau$ is a non-empty corolla with vertex~$v$, then its geometric realisation is the disjoint union
of segments whose components bijectively correspond to elements of $F_{\tau}$
modulo identification of all points $0$, which becomes the geometric
realisation of $v$. Generally, consider the disjoint union of geometric realisations of corollas
of all vertices of $\tau$, and then identify the endpoints of two different flags,
if these two flags are connected by the involution~$j_{\tau}$. Sometimes one considers the endpoint removed in the case of geometric realisations of all flags that are stationary points of the involution.

The pairs of flags, connected by involution, resp.\ their geometric realisations, are called {\em edges} of $\tau$, resp.\ their geometric realisations.

In order not to mix free flags with halves of edges, one may call a free flag {\em a leaf}, or {\em a tail} as in~\cite{KoMa}.

The description of graphs in terms of flags with an involution is common especially in contexts originating from physics where corollas represent possible interactions and one expects graph to have both internal edges, namely pairs of different flags glued by the involution, and external edges (leaves or tails), namely flags that are fixed points of the involution.

Among various types of labelings or markings of the sets $V_\tau$, $F_\tau$ forming a graph,
an important role is played by {\it orientations of flags.} If two flags form an edge, their orientations
must agree. We often use oriented trees, in which one free flag is chosen as a {\it root},
and all other leaves are oriented in such a way, that from each leaf there exists a unique
oriented path to the root.

Below we will often pass in inverse direction: from the geometric realization of a graph
to its set-theoretic description. It should not present any difficulties for the reader.

\subsection{Belyi maps and dessins d'enfant}\label{section2.1}

Let $\Sigma$ be a smooth compact Riemann surface. We recall the following definition
(see \cite{Be,Groth}).

\begin{defn}\label{Be-dessins} \quad
\begin{enumerate}\itemsep=0pt
\item[$(i)$] A Belyi map $f\colon \Sigma \rightarrow \P^1(\C)$ is a holomorphic map that is unramified outside of the set $\{0, 1, \infty\}$.

\item[$(ii)$] Dessin d'enfant $D=D_f$ of $f$ is the preimage $f^{-1}([0,1])$ embedded in $\Sigma$ and
considered as a marked graph in the following way.
\end{enumerate}

Its set of vertices $V(D)$ is the union $f^{-1}(0)\cup f^{-1}(1)$; each point in the preimage $V_1(D)=f^{-1}(1)$, resp.~$V_0(D)=f^{-1}(0)$, is marked white, resp.~black. Such graphs are called bipartite ones.

Its set of (open) edges $E(D)$ consists of connected components of $f^{-1}(0,1)$ of the preimage of the unit interval $\cI$.
\end{defn}

Each $\Sigma$ is the Riemann surface of a smooth complex algebraic curve.
The role of Belyi maps in this context is determined by Belyi's converse theorem~\cite{Be}:
each smooth complex algebraic curve defined over subfield of algebraic numbers admits
a~Belyi map.

Grothendieck's intuition behind introduction of dessins of such maps was the hope,
that action of the absolute Galois group upon the category of such algebraic curves
would admit an explicit description after the transfer of this action to dessins.

More precisely, in terms of branched coverings and Belyi maps $f\colon \Sigma \to \P^1$, we can view the Galois action of $G={\rm Gal}\big(\bar\Q/\Q\big)$ in the following way. In the connected case, consider the finite extension $\bar\Q(\Sigma)$ of the function field~$\bar\Q(t)$. The field $\bar\Q(\Sigma)$ is given by $\bar\Q(t)[z]/(P)$ where $P\in \bar\Q(t)[z]$ is an irreducible polynomial. An element $\gamma \in G$ maps $P\mapsto P_\gamma$ where~$P_\gamma$ is the polynomial obtained by action of~$\gamma$ on the coefficients of~$P$ (extending the action of~$\gamma$ to~$\Q(t)$ by the identity on~$t$). Thus the action of $G$ maps the extension $\bar\Q(t)[z]/(P)$ to $\bar\Q(\Sigma_\gamma):=\bar\Q(t)[z]/(P_\gamma)$. The action extends to the case where $\Sigma$ is not necessarily connected, by decomposing the \'etale algebra of the covering into a direct sum of field extensions as above. Correspondingly, the element $\gamma \in G$ maps the dessin~$D$ determined by the branched covering $f\colon \Sigma \to \P^1$, or equivalently by the extension $\bar\Q(\Sigma)$ to the dessin~$\gamma D$ determined by $\bar\Q(\Sigma_\gamma)$.

Therefore we must first trace how the combinatorics of dessins encodes the geometry of their Belyi maps.

\subsection{Geometry of Belyi maps vs combinatorics of dessins}
\begin{enumerate}\itemsep=0pt
\item[(a)] The number of edges $d=\# E(D)$ equals the degree of $f\colon \Sigma \to \P^1(\C)$
\item[(b)] The numbers of black and white
vertices $m=\# V_0(D)$ and $n=\# V_1(D)$ correspond to the orders of ramification of $f$ at $0$ and $1$.
\item[(c)] The numbers $\mu_1,\ldots, \mu_m$ and $\nu_1,\ldots,\nu_n$ of edges in the corollas around the
black, resp.\ white vertices, give the ramification profiles of $f$ at $0$ and $1$. According to the
degree sum formula for bipartite graphs, $\sum_i \mu_i = d =\sum_j\nu_j$.
\item[(d)] In order to keep track of additional data of topology of embedding $D=D_f$ into $\Sigma$ we must endow $D$
with cyclic ordering of flags at each corolla coming from the canonical orientation of $\Sigma$.
In other words, $D$ must be considered as a bipartite {\em ribbon} graphs.
\item[(e)] The graph $D$ is connected iff the covering is connected, or equivalently iff the \'etale algebra of the covering is a field. The
degree $d$ is the dimension of the \'etale algebra of the covering as a vector space over $\bar\Q(t)$,
in the connected case the degree of the field extension. The group of symmetries of a connected
dessin is the automorphism group of the extension, which is a finite group of order at most equal to
the degree $d$.
\item[(f)] Finally, for each bipartite ribbon graph $D$ one can construct a Belyi map
$f$ such that $D=D_f$.
\end{enumerate}

In order to see that, one can first construct a map of topological surfaces $\Sigma^{\prime}\to \P^1(\C )$
with desired properties, by working locally in a suitable covering of $\Sigma^{\prime}$,
and then endow $\Sigma^{\prime}$ with the complex structure lifted from~$\P^1(\C )$.

\subsection{Regular and clean dessins} \label{section2.3}
For studying action of the Galois group, two subclasses of dessins d'enfant are especially inte\-res\-ting.
\begin{enumerate}\itemsep=0pt
\item[(a)] {\em Regular dessins.} A dessin is called regular one, if it is connected and if its symmetry
group is as large as possible: its cardinality equals to the degree $d$.
\end{enumerate}

Regular dessins correspond to Galois field extensions. Geometrically, they correspond to regular branched
coverings where the deck group acts transitively on any fiber. Every connected dessin admits a
regular closure, which corresponds to the smallest extension that is Galois.
\begin{enumerate}\itemsep=0pt
\item[(b)] {\em Clean dessins}. A dessin is called clean, if all its white vertices have valence two.
\end{enumerate}

The clean dessins can be re-encoded to ribbon graphs that are not bipartite, by colouring all its vertices
black and then inserting valence two white vertices in the middle of each edge. This was the class
of dessins originally considered by Grothendieck. In terms of branched coverings of~$\P^1(\C)$ and
Belyi maps, the clean dessins correspond to coverings whose ramification profile over the point~$1$
is only of type $(2, 2, \ldots, 2, 1, \ldots, 1)$, that is, they have either simple ramification or no ramification.

Any dessin, not necessarily connected, whose all connected components are clean,
will be called {\em locally clean}.

\subsection{Hopf algebras of dessins}
\begin{enumerate}\itemsep=0pt
\item[(a)] {\em Locally clean dessins.} We consider first the class $\cC\cD$ of locally clean dessins re-encoded as in Section~\ref{section2.3}(b). Isomorphism classes of such graphs can be used in order to construct a Hopf algebra, which will serve as a generalization to ribbon graphs~\cite{Malyshev} of the Connes--Kreimer Hopf algebra of renormalization in perturbative quantum field theory~\cite{CK}.
\end{enumerate}

We start by constructing a Hopf algebra $\cH_{\cC\cD}$ of locally clean dessins.
As an algebra $\cH_{\cC\cD}$ is the free commutative $\Q$-algebra freely
generated by isomorphism classes of clean dessins. Locally clean dessins
are identified with the monomials given by the products of their connected
components.
More generally, if $D$ is a clean design considered as ribbon graph on the black vertices
and $\delta \subset D$ is its proper subgraph, we can construct the quotient $D/\delta$.

Combining these two constructions, we can construct the map
\begin{gather*} \Delta \colon \ \cH_{\cC\cD} \to \cH_{\cC\cD}\otimes \cH_{\cC\cD}, \qquad \Delta (D) :=
D \otimes 1 + 1 \otimes D +\sum_{\delta \subset D} \delta \otimes D/\delta.
\end{gather*}

An equivalent description of the coproduct is given in~\cite{Malyshev} in terms of surfaces.
Let $\Sigma$ denote the Riemann surface determined by the dessin~$D$. Consider open not
necessarily connected subsurfaces (Riemann surfaces with boundary) $\sigma \subset \Sigma$.
To each such subsurface $\sigma$ one can associate a quotient $\Sigma/\sigma$ and a closure $\bar \sigma$.
The quotient is the closed surface obtained by replacing $\sigma$ in $\Sigma$ with a sphere with
the same number of holes glued back into $\Sigma$ in place of $\sigma$. The closure $\bar \sigma$ is
obtained by adding boundary edges so that every sphere with holes in $\bar \sigma$ has the same
number of boundary edges as it would have in $\Sigma$. The coproduct is then equivalently
written in terms of surfaces as \cite{Malyshev}
\begin{gather*} \Delta(\Sigma) = \Sigma \otimes 1 + 1 \otimes \Sigma +\sum_{\sigma \subset \Sigma} \bar \sigma
\otimes \Sigma/\sigma. \end{gather*}
This coproduct is shown in \cite{Malyshev} to be coassociative.

The number of edges (which is equal to the degree of the Belyi map) is
additive, $\# E(D)=\# E(\delta) + \# E(D/\delta)$. One can consider the Hopf algebra as graded by
this degree. The antipode on a connected graded Hopf algebra is defined inductively by the formula
$S(X)=-X +\sum S(X') X''$ for $\Delta(X)=X\otimes 1+1\otimes X +\sum X'\otimes X''$ with terms
$X'$ and $X''$ of lower degree.

\begin{enumerate}\itemsep=0pt
\item[(b)] {\em General dessins.} This construction can be extended from clean dessins to all dessins in the following way.
\end{enumerate}

\begin{defn}\label{subD} Let $D$ be a connected dessin. A $($possibly non-connected$)$ sub-dessin $\delta$
consists of a $($possibly non-connected$)$ subgraph of $D$ which is bipartite.

We consider $\delta$ as endowed with internal and external edges: the internal edges
are the edges between vertices of the subgraph and the external edges $($flags$)$ are half-edges
for each edge in~$D$ with one end on the subgraph~$\delta$ and the other on some
vertex in $D\smallsetminus \delta$. We endow~$\delta$ with a~ribbon structure where
the corolla of each vertex $($including both internal and external edges$)$ has a~cyclic
orientation induced from the ribbon structure of~$D$.
\end{defn}

In the following, with a slight abuse of notation, we will use the notation $\delta \subset D$
for the sub-dessin with or without the external edges included, depending on context. In
each case the graph $\delta$ is bipartite and with a ribbon structure induced by $D$.
We describe how to obtain a quotient dessin in this more general case, and we then
discuss the corresponding possible choices of grading of the resulting Hopf algebra.

\begin{lem}\label{subquotD}Let $D$ be a connected dessin and $\delta \subset D$ a sub-dessin as in Definition~{\rm \ref{subD}}. There is a quotient bipartite graph $D/\delta$, obtained by shrinking each component of $\delta$ to a~bipartite graph with one white and one black vertex and a single edge, for which any choice of a cyclic ordering of the boundary components of a tubular neighborhood of $\delta$ in $\Sigma$
determines a ribbon structure, making~$D/\delta$ a dessin.
\end{lem}

\begin{proof}We need to check that the graph $D/\delta$ is indeed still a dessin, namely that
it is bipartite and has a ribbon structure, and that these are compatible with those of
the initial dessin $D$ and of the subdessin~$\delta$.

The dessin structure on the
quotient $D/\delta$ is obtained by the following procedure.
First label all the external half edges of $\delta$ with black/white labels according to whether
they are attached to a black or white vertex of $\delta$.

Forgetting temporarily the bipartite structure, consider then the quotient graph, which we denote by $D//\delta$,
obtained by shrinking each connected component of the subgraph $\delta$ to a single vertex,
but with the edges of $D//\delta$ out of these quotient vertices retaining the black/white labels coming
from the bipartite structure.

Consider a small tubular neighborhood of the subgraph $\delta$ on the surface $\Sigma$
where $D$ is embedded and consider its boundary components: endow the corollas of the
quotient vertices in $D//\delta$ with a cyclic structure obtained by listing the half edges in the order
in which they are met while circling around each of the boundary components (in an assigned
cyclic order) in the direction induced by the orientation of $\Sigma$.

Because of the black/white labels of the external edges of $\delta$, the cyclic ordering of half-edges
around each quotient vertex in $D//\delta$ can be seen as a shuffle of two cyclically ordered sets
of black and white labelled half-edges: separate out each quotient vertex into a black and a white vertex
connected by one edge, with the cyclically ordered black half-edges attached to the black vertex
and the cyclically ordered white half-edges attached to the white vertex.
The resulting graph $D/\delta$ obtained in this way is bipartite and has a ribbon structure,
hence it defines a quotient dessin. The ribbon structure defined in this way depends on the
assignment of a cyclic ordering to the boundary components of a tubular neighborhood
of $\delta$ in $\Sigma$.
\end{proof}

We can again reformulate the construction of the quotient dessins $D/\delta$ in terms of
the surfaces~$\sigma$ and~$\Sigma/\sigma$ as above~\cite{Malyshev}.
Although the formulation we gave above in terms of graphs is more explicit, an advantage of
the reformulation in terms of surfaces we discuss below is that it explains
the definition of the ribbon structure on~$D/\delta$ used in Lemma~\ref{subquotD}
in a more natural way as follows.

\begin{cor}\label{quotDS}The ribbon structure on $D/\delta$ can be equivalently described by
considering $D/\delta$ embedded in the surface $\Sigma/\sigma$, for $\sigma$
the sub-surface with boundary $\sigma \subset \Sigma$ containing $\delta$,
with the external edges of $\delta$ cutting through the components
of $\partial \sigma$. Each cyclic orientation of the components of $\partial \sigma$
determines a ribbon structure on $D/\delta$.
\end{cor}

\begin{proof} An open subsurface $\sigma$ of $\Sigma$ containing $\delta$,
with the property that the external edges of $\delta$ pass through
the boundary curves $\partial \sigma$ is equivalent, up to homotopy, to
a tubular neighborhood of $\delta$ in $\Sigma$ with its boundary curves.
We can assume for simplicity that $\delta$, hence $\sigma$, are connected.
Thus, a choice of a cyclic ordering of the boundary curves of such a
tubular neighborhood is equivalent to the choice of a cyclic ordering
of the components of $\partial \sigma$. The quotient bipartite graph $D/\delta$
is contained in the surface $\Sigma/\sigma$ obtained by gluing a sphere
with the same number of punctures to the boundary
$\partial (\Sigma\smallsetminus \sigma)=\partial \sigma$, with the bipartite graph
with one black and one white vertex and the same external edges
of $\delta$ contained in this sphere, with the same external edges
cutting through the same boundary components as in the case
of $\delta$ in $\sigma$. If $\delta$ has multiple connected components,
and $\sigma$ has correspondingly multiple components, then the same
can be argued componentwise, when each component of $\sigma$ is
replaced by a sphere with the same boundary in $\Sigma/\sigma$
with each of these spheres containing a
 bipartite graph with one black and one white vertex and the
 number of external edges of the corresponding component
 of $\delta$. We see in this way that the procedure for the
construction of the ribbon structure on $D/\gamma$
described in Lemma~\ref{subquotD} is the same as the one
described here.
\end{proof}


As the result, we obtain another commutative connected Hopf algebra $\cH_{\cD}$ with involution,
with a grading that is again expressible in terms of the number of edges (the degree of the
respective Belyi map), as follows.

\begin{lem}\label{HDgrading}
Any of the functions defined on connected dessins as follows defines a good grading
of the Hopf algebra $\cH_{\cD}$,
\begin{gather}\label{Dgrading}
\deg(D) = \begin{cases} b_1(D) & \text{or} \\
\# E(D) -1 & \text{or} \\
\# V(D) -2.
\end{cases}
\end{gather}
\end{lem}

\begin{proof} The degree defined with any of the options in~\eqref{Dgrading} for connected dessins
extends additively to monomials in $\cH_D$ (multi-connected dessins) so that one
obtains, respectively, $b_1(D)$, $\# E(D)-b_0(D)$, and
$\# V(D)-2 b_0(D)$.
With the construction of the quotient dessin $D/\delta$ of a~subdessin $\delta \subset D$
as in Lemma~\ref{subquotD}, we have $\# V(D) = \#V(\delta) +\# V(D/\delta)- 2 b_0(\delta)$
and $\# E(D) =\# E(\delta)+ \# E(D/\delta) - b_0(\delta)$, hence $b_1(D)=1-\# V(D) +\# E(D)= 1-
\#V(\delta) - \# V(D/\delta)+ 2 b_0(\delta) +\# E(\delta)+ \# E(D/\delta) - b_0(\delta)=
b_0(\delta)+ \# E(\delta) -\#V(\delta) +1 +\# E(D/\delta)-\# V(D/\delta)=b_1(\delta)+b_1(D/\delta)$.
Thus, in all three cases of~\eqref{Dgrading} we obtain
$\deg(D)=\deg(\delta)+\deg(D/\delta)$.
\end{proof}

The gradings of Lemma~\ref{HDgrading} are the analogs for dessins of the
choices of gradings for the original Connes--Kreimer Hopf algebra of graphs
by either loops $b_1(\Gamma)$, or internal edges $\#E(\Gamma)$ or by
$\# V(\Gamma)-1$ (for connected graphs, extending to $\# V(\Gamma)-b_0(\Gamma)$
for monomials). Here all the choices of grading in Lemma~\ref{HDgrading} have a direct
interpretation in terms of Belyi maps, since~$\# E(D)$ is the degree of the Belyi
maps and~$\# V(D)$ is the sum of the orders of ramification at~$0$ and~$1$.

The coassociativity of the coproduct~\eqref{HDcoprod} is the only property that needs to
be verified to ensure that indeed we obtain a Hopf algebra, as all the other properties are
clearly verified. The argument for the coassociativity is similar to the case of the Conner--Kreimer
Hopf algebra, with some modifications due to the different definition of the quotient graph
that we used in Lemma~\ref{subquotD}. We present the explicit argument for the reader's
convenience.

\begin{prop}\label{HopfDessins}As an algebra $\cH_\cD$ is the commutative polynomial $\Q$-algebra in the connected dessins.
The coproduct on $\cH_\cD$ is given by
\begin{gather}\label{HDcoprod}
\Delta(D)=D\otimes 1 + 1\otimes D + \sum_{\delta \subset D} \delta \otimes D/\delta,
\end{gather}
where now the sum is over the sub-dessins defined as in Definition~{\rm \ref{subD}} and also
over all the possible ribbon structures on the corresponding quotients, defined as in
Lemma~{\rm \ref{subquotD}}. Then $\cH_\cD$ is a graded connected Hopf algebra, with the grading by $\#E(D)-b_0(D)$, where $\#E(D)$ is the number of edges $($the degree of the Belyi map$)$.
\end{prop}

\begin{proof} We just need to show that $\Delta$ is coassociative. The antipode is then constructed inductively as before. It suffices to check the identity $(\Delta \otimes {\rm id})\Delta = ({\rm id} \otimes \Delta)\Delta$ on a single connected dessin $D$. We have $(\Delta \otimes {\rm id})\Delta(D)=\sum\limits_{\delta\subseteq D} \Delta(\delta)\otimes D/\delta$. If $\delta$ is multiconnected, $\Delta(\delta)$ is the product of $\Delta(\delta_j)$
over the connected components $\delta_j$, since $\Delta$ is an algebra homomorphism. We have
$\Delta(\delta)=\sum\limits_{\delta'\subseteq \delta}
\delta'\otimes \delta/\delta'$, where in the case where $\delta$ has multiple connected components
the subdessin~$\delta'$ of~$\delta$ consists of a collection
of a subdessin of each component. Similarly, in the quotient dessin~$\delta/\delta'$ we can perform the
quotient operation of Lemma~\ref{subquotD} separately on the components of~$\delta$, so that the
resulting~$\delta/\delta'$ is a product of the quotients over each component. The bipartite and ribbon
structure of~$\delta'$ are the same with respect to~$\delta$ as they are with respect to $D$, since the
set of external edges of~$\delta'$ in~$\delta$ are either external lines in~$\delta$ or internal edges of~$\delta$ that define external lines in~$\delta'$ and in each case these would also occur as external
lines of~$\delta'$ in $D$. Thus we have $(\Delta \otimes {\rm id})\Delta(D)=\sum\limits_{\delta'\subseteq \delta\subseteq D} \delta'\otimes \delta/\delta' \otimes D/\delta$. We also have $({\rm id} \otimes \Delta)\Delta(D)=\sum\limits_{\delta'\subseteq D} \delta' \otimes \Delta(D/\delta')$. As subgraphs and quotient graphs we have $\Delta(D/\delta')=\sum\limits_{\delta'\subseteq \delta\subseteq D} \delta/\delta' \otimes D/\delta$ (see \cite[Theorem~1.27]{CoMa-book}). Thus, we need to check that this identification is also compatible with the bipartite and the ribbon structures. The identification above depends on identifying subgraphs $\tilde\delta\subseteq D/\delta'$
with subgraphs $\delta\subseteq D$ with $\delta'\subseteq \delta$. Given such a $\delta$ with
$\delta'\subseteq \delta\subseteq D$ one obtains $\tilde\delta$ from $\delta$ with the quotient procedure
described in Lemma~\ref{subquotD}, which determines the bipartite and ribbon structure on $\tilde\delta$
and every $\tilde\delta\subseteq D/\delta'$ is obtained in this way from a unique $\delta\subseteq D$ with
$\delta'\subseteq \delta$, with~$\tilde\delta=\delta/\delta'$.
\end{proof}

Each of the Hopf algebras $\cH = \cH_{\cC\cD}$ or $\cH_{\cD}$ determines the respective
affine group scheme $\cG = \cG_{\cC\cD}$ or $\cG_{\cD}$ whose $\bar{\Q}$-points
are morphisms of $\Q$-algebras $\cH \to \bar{\Q}$
and multiplication is defined by $\phi_1 * \phi_2 (X)= (\phi_1\otimes \phi_2) \Delta (X)$,
with inverse $\phi^{-1}=\phi\circ S$ given by composition with the antipode.

\subsection{Hopf algebra and Galois action}\label{HopfGaloisSec}

We want to make the construction described above of the Hopf algebra $\cH_{\cD}$ of dessins
compatible with the Galois action. This would mean requiring that, for all $\gamma \in {\rm Gal}(\bar \Q/\Q)$ we have $\Delta(\gamma D)=\sum \gamma \delta \otimes \gamma D/\delta$. The condition is satisfied
if we have an identification of the class of images $\gamma \delta$ of subdessins $\delta \subset D$
with the class of subdessins of $\gamma D$ and an identification of the images of the quotient
dessins $\gamma D/\delta$ with the quotients $\gamma D/\gamma\delta$.
When we consider in the construction of the coproduct~$\Delta(D)$
of~$\cH_D$ all the possible choices of subdessins $\delta$ and quotient dessins~$D/\delta$, defined
as in Lemma~\ref{HopfDessins}, we are also including pairs $\{ \delta, D/\delta \}$ that are not necessarily
compatible with the Belyi map and the Galois action on the Belyi maps, for which the required identifications need not hold. We can see this by considering
the following example, based on of \cite[Example~2.2.15]{LaZvon}.

\begin{lem}\label{nofix}The dessins $A$ and $B$ shown in the figure are in the same Galois orbit. The subdessins $\delta_A$ and $\delta_B$ circled in the figure are also in the same Galois orbit $($which in this case consists of a single element$)$. However, the quotients $A/\delta_A$ and $B/\delta_B$ do
not belong to the same Galois orbit.
\end{lem}

\begin{figure}[h!]\centering
\includegraphics[scale=0.2]{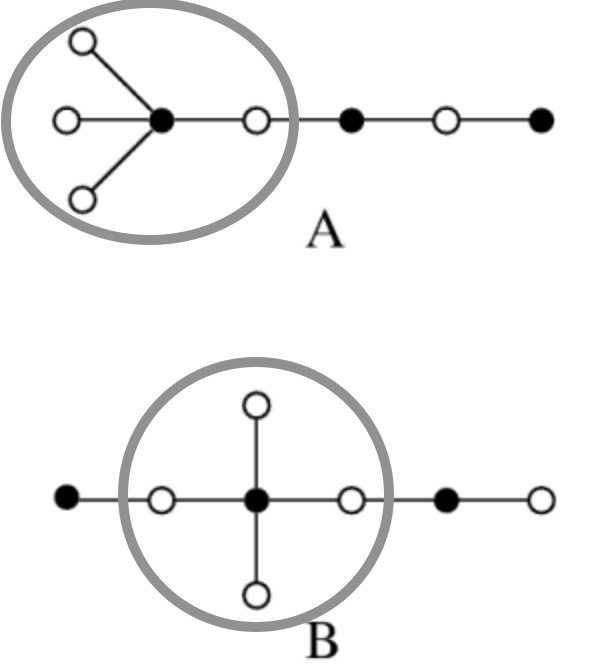}
\end{figure}

\begin{proof}The circled subgraph $\delta_B$ of $B$ is the only subdessin in the same Galois orbit
of the subgraph $\delta_A$ of $A$ since the action of the Galois group is trivial on trees with
less than five edges. The quotient dessins $A/\delta_A$ and $B/\delta_B$, however,
are not in the same Galois orbit since they have different vertex degrees (a Galois invariant).
\end{proof}

This problem can be corrected by only considering, in the
construction of the coproduct~$\Delta(D)$ of the Hopf algebra,
certain pairs $\{ \delta, D/\delta \}$ of subdessins and quotient
dessins that behave well with respect to the Galois orbit.

\begin{defn}\label{Sigmapairs}
A balanced pair $\{ \delta,D/\delta \}$ for a dessin $D$ is a pair of subdessin
and quotient dessin with the property that for $\gamma\in {\rm Gal}(\bar\Q/\Q)$
the pair $\{ \gamma \delta, \gamma D/\delta \}$ is a pair of a subdessin and
a quotient dessin for $\gamma D$. A subdessin $\delta \subset D$ is strongly
balanced if for all subdessins $\delta'\subseteq \delta$ the pair $\{ \delta', D/\delta'\}$
is balanced. Let $\cB(D)$ denote the set of strongly balanced subsessins.
\end{defn}

\begin{ex}\label{strongbal}The subdessins circled in the figure are an example of strongly balanced subdessins. Both have quotient dessin the bipartite line with three white and three black vertices.
All their subdessins consist of either a single edge with a white and a black vertex, or
two edges with a common black vertex and two white vertices. In the first case the
quotient dessin is just~$A$ or~$B$ itself and in the second case the quotients are given
by two tree dessins on six edges that are in the same Galois orbit.
\end{ex}

\begin{figure}[h!]\centering
\includegraphics[scale=0.2]{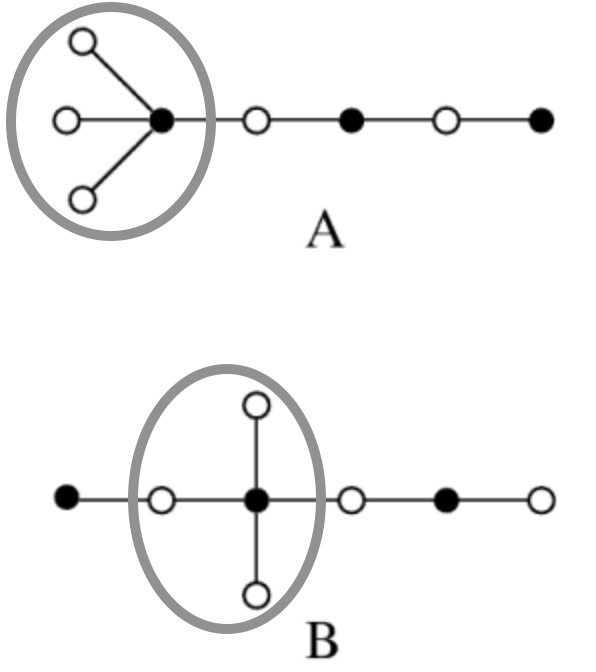}
\end{figure}

Note that the balanced property depends on the entire Galois orbit of $D$ not just on $D$ itself,
hence the set $\cB(D)$ of strongly balanced pairs is associated to all dessins~$D$ in the same orbit.
The coproduct of the Hopf algebra $\cH_\cD$ is then modified by summing in~\eqref{HDcoprod}
only over the strongly balanced subdessins $\delta$, with each appearing once in the coproduct,
\begin{gather}\label{HDcoprodB}
\Delta(D)=D\otimes 1 + 1\otimes D + \sum_{\delta\in \cB(D)} \delta \otimes D/\delta .
\end{gather}
In the following we just write for simplicity of notation $\Delta(D)=\sum\limits_{\cB(D)} \delta \otimes D/\delta$ with the primitive part $D\otimes 1 + 1\otimes D$ implicitly included in the sum.

\begin{lem}\label{coprodBcoass}The restriction \eqref{HDcoprodB} of the coproduct to the strongly balanced subdessins maintains the coassociativity property.
\end{lem}

\begin{proof} We have $(\Delta \otimes {\rm id})\Delta (D) =\sum\limits_{\delta\in \cB(D)} \Delta(\delta)\otimes D/\delta$
and $\Delta(\delta)=\sum\limits_{\delta' \in \cB(\delta)} \delta' \otimes \delta/\delta'$. Because of the strongly
balanced conditions $\cB(D)$ and $\cB(\delta)$, for $\delta'\subset \delta$ we have both $\gamma D/\delta'=
\gamma D/ \gamma \delta'$ and $\gamma \delta/\delta' =\gamma\delta/ \gamma \delta'$. Thus, the pair
$(\delta/\delta', D/\delta')$ with $D/\delta'=(D/\delta)/(\delta/\delta')$ is a balanced pair for~$D/\delta$, and in fact strongly balanced, since any subdessin of $\delta/\delta'$ can be identified with a~$\delta/\delta''$ for some subdessin $\delta''\subset \delta$. Thus, the rest of the argument of Proposition~\ref{HopfDessins} continues to hold.
\end{proof}

We still denote by $\cG$ the affine group scheme dual to the Hopf algebra $\cH_\cD$ endowed with the coproduct~\eqref{HDcoprodB}.

\begin{prop}\label{GHopfact} The action of the absolute Galois group $G={\rm Gal}\big(\bar\Q/\Q\big)$ passes through the automorphism group scheme ${\rm Aut}(\cG)$ of the affine group scheme $\cG$.
\end{prop}

\begin{proof} It suffices to show that the $G$-action on dessins induces an action on $\cH$ by
bialgebra homomorphisms. The compatibility with the multiplication is clear since on
non-connected dessins the action of $G$ is defined componentwise, so $\gamma (D_1 \cdot D_2)=
\gamma D_1 \cdot \gamma D_2$. The compatibility with comultiplication $\Delta(\gamma D)=
\sum\limits_{\delta\subseteq D} \gamma \delta \otimes \gamma D/\delta$ requires the identification of
the class of images~$\gamma \delta$ of subdessins $\delta \subset D$ with the class of subdessins
of $\gamma D$ and the identification of the images of the quotient dessins $\gamma D/\delta$ with
the quotients $\gamma D/\gamma\delta$. This is ensured by the balanced condition on the
pairs $\{ \delta, D/\delta \}$ introduced in the coproduct~\eqref{HDcoprodB}.
Thus, elements $\gamma \in G$ act as bialgebra homomorphisms on~$\cH$. The compatibility
$S \circ \gamma =\gamma \circ S$ with the antipode $S$ then follows for a general bialgebra homomorphism,
hence they are automorphisms of the Hopf algebra $\cH$ and dually of the affine group scheme $\cG$.
\end{proof}

\subsection{Rota--Baxter algebras and Birkhoff factorization} Let $\cR$ be a commutative algebra.
A structure of Rota--Baxter algebra of weight $-1$ on it is given by a linear operator $T\colon \cR \to \cR$ satisfying the Rota--Baxter relation of weight $-1$:
\begin{gather}\label{RBrel}
T(xy) + T(x) T(y) = T(x T(y)) + T(T(x) y).
\end{gather}
The ring of Laurent series with $T$ given by the projection onto the polar part is the prototype
example of a Rota--Baxter algebra of weight $-1$. The Rota--Baxter relation~\eqref{RBrel}
ensures that the range $\cR_+=(1-T)\cR$ is a subalgebra of~$\cR$, not just a linear subspace,
and $\cR_-=T \cR$ as well.

Consider now $\cR$-valued characters of, say, $\cH_\cD$.

The Birkhoff factorization $\phi=(\phi_- \circ S) \star \phi_+$ of a
character $\phi\in {\rm Hom}_{{\rm Alg}}(\cH_\cD,\cR)$ is defined inductively by the formula
\begin{gather*}
\phi_- (x) = -T (\phi(x) + \sum \phi_-(x') \phi(x'')),\nonumber \\
\phi_+(x) = (1-T) (\phi(x) + \sum \phi_-(x') \phi(x'')),
\end{gather*}
where $\Delta(x)=\sum x' \otimes x''$ in Sweedler notation.
The factors $\phi_\pm$ are algebra homomorphisms
$\phi_\pm \in {\rm Hom}_{{\rm Alg}}(\cH_\cD,\cR_\pm)$.

In the case where the Rota--Baxter algebra is an algebra of Laurent power
series with the operator $T$ given by projection onto the polar part, the
evaluation at $0$ of the positive part $\phi_+$ of the Birkhoff factorization
has the effect of killing the polar part, hence it achieves renormalization,
see~\cite{CK}. Such a systematic deletion of ``divergences'' (at least, formal ones) is used
in the Connes--Kreimer algebraic theory of renormalization in quantum field theory~\cite{CK}. For a general introduction to Rota--Baxter algebras and their properties see~\cite{Guo}. In our context,
formal sums over dessins d'enfant with weights defining partition functions
are only formal series because ``there are too many'' such dessins, so that
divergences should be deleted by renormalization. This will be achieved through
a deformation of our quantum statistical mechanical system, which we introduce
in Section~\ref{RegSysSec}, rather than through a Birkhoff factorization procedure.
Birkhoff factorizations, however, have other interesting applications besides the elimination
of divergences in quantum field theory. As shown in~\cite{MarTed} in the
context of computation, Birkhoff factorization can be used to enrich invariants of a graph by
compatibly combining invariants of subgraphs and quotient graphs.
The reason why we are considering them here is similar to the applications
considered in~\cite{MarTed}, namely as a method for constructing new Galois
invariants of dessins, as we show in the following.

\begin{lem}\label{GalBirkhoff}The action $\phi\mapsto \gamma\phi$ of $G={\rm Gal}\big(\bar\Q/\Q\big)$ on the group-scheme $\cG_\cD$ is compatible with the Birkhoff factorization, in the sense that $(\gamma \phi)_\pm=\gamma \cdot \phi_\pm$.
\end{lem}

\begin{proof} The action of $G$ on $\cG_\cD$ is determined by the action of $G$ on $\cH_\cD$ by
Hopf algebra homomorphisms. We have $\Delta(\gamma x)=\sum \gamma x' \otimes \gamma x''$.
Thus we obtain $\phi_\pm (\gamma x)$ as either $-T$ or $1-T$ applied to $\phi (\gamma x) + \sum \phi_- (\gamma x') \phi(\gamma x'')$. This consistently defines $\gamma \phi_\pm (x)$ as $\phi_\pm(\gamma x)$ compatibly with the action of $G$ on ${\rm Hom}(\cH_\cD,\cR_\pm)$.

\begin{lem}\label{idealHopf}The action of the Galois group $G={\rm Gal}\big(\bar\Q/\Q\big)$ by Hopf algebra homomorphisms of $\cH_\cD$ determines a Hopf ideal and a quotient Hopf algebra $\cH_{\cD,G}$.
\end{lem}

\begin{proof} Consider the ideal $\cI_{\cD,G}$ of the algebra $\cH_\cD$ generated by elements of the form
$D -\gamma D$ for $D\in \cD$ and $\gamma \in G$. To see that $\cI_{\cD,G}$ is a Hopf ideal notice
that we have $\Delta(D) - \Delta(\gamma D) =\sum \delta'\otimes \delta'' - \sum \gamma \delta' \otimes
\gamma \delta'' = \sum (\delta' - \gamma \delta')\otimes \delta'' + \sum \gamma \delta' \otimes (\delta'' -\gamma \delta'')$ hence $\Delta(\cI_{\cD,G})\subset \cI_{\cD,G}\otimes \cH_\cD + \cH_\cD \otimes \cI_{\cD,G}$. Since $\cI_{\cD,G}$ is a Hopf ideal, it defines a quotient Hopf algebra $\cH_{\cD,G}=\cH_\cD / \cI_{\cD,G}$.
\end{proof}

Characters in ${\rm Hom}_{{\rm Alg}}(\cH_{\cD,G},\cR)$ with $\cR$ a commutative algebra over $\Q$
are $\cR$-valued Galois invariants of dessins that satisfy the multiplicative property
$\phi(D\cdot D')=\phi(D) \phi(D')$ over connected components. The following is then an immediate
consequence of the previous statements.

\begin{prop}\label{BirkhoffHDG}Given a Galois invariant of dessins $\phi\in {\rm Hom}_{{\rm Alg}}(\cH_{\cD,G},\cR)$, where $(\cR,T)$ is a Rota--Baxter algebra of weight $-1$, the Birkhoff factorization $\phi=(\phi_-\circ S)\star \phi_+$ determines new Galois invariants $\phi_\pm \in {\rm Hom}_{{\rm Alg}}(\cH_{\cD,G},\cR_\pm)$.
\end{prop}

One should regard the Galois invariants $\phi_\pm$ constructed in this way as a refinement
of the Galois invariant $\phi$ in the sense that they do not depend only on the value of $\phi$
on $D$, but also inductively on the values on all the subdessins $\delta$ and the quotient
dessins $D/\delta$, hence for example they can potentially distinguish between non-Galois
conjugate dessins with the same value of $\phi$, but for which the invariant $\phi$ differs
on subdessins or quotient dessins.

\subsection{New Galois invariants of dessins from Birkhoff factorization}\label{TutteSec}

We discuss here a simple explicit example, to show how the formalism
 of Rota--Baxter algebras and Bikrhoff factorizations recalled in the
 previous subsection can be used as a method to construct new
 Galois invariants of dessins.

We start by considering a known invariant of dessins with values in a
 polynomial algebra. This is given by the $2$-variable Tutte polynomial,
 as well as its $3$-variable generalization, the Bollob\'as--Riordan--Tutte polynomial
 of~\cite{BoRio} defined for ribbon graphs. Both can be applied to dessins $D$
 and we will show in Lemma~\ref{TutteGal} that both are Galois invariants of dessins.
 We also consider some possible specializations of these polynomials,
 related to the Jones polynomial, as discussed in~\cite{Das}.
 As already discussed in \cite{AluMa1,AluMa2} in the context of the
 Connes--Kreimer Hopf algebra of Feynman graphs, such invariants derived from
 the Tutte polynomial can be regarded as ``algebraic Feynman rules'', which
 simply means characters of the Hopf algebra~$\cH$: homomorphisms of
 commutative algebras from $\cH$ to a commutative algebra $\cR$ of polynomials
 (Tutte and Bollob\'as--Riordan case) or of Laurent polynomials (Jones case).

We then consider
 a Rota--Baxter structure of weight $-1$ on the target algebra of polynomials
 or of Laurent polynomials
 and use the Birkhoff factorization formula to obtain new Galois invariants
 of dessins.

 The main point of this construction in our environment is not
 just subtraction of divergences: the intended effect is to
 introduce new invariants that are sensitive not only to~$D$ itself but
 that encode in a consistent way (where consistency is determined by the Rota--Baxter operator)
 the information carried by all the sub-dessins $\delta$ and the
 quotient dessins $D/\delta$ as well.

 The Tutte polynomial of a graph is determined uniquely by a deletion-contraction
 relation and the value on graphs consisting of a set of vertices with
 no edges, to which the computation reduces by repeated application
 of the deletion-contraction relation. It can also be defined by a~closed
 ``sum over states'' formula
 \begin{gather}\label{Tutte}
 P_D(x,y) = \sum_{\delta \subset D} (x-1)^{b_0(\delta)-b_0(D)} (y-1)^{b_1(\delta)},
 \end{gather}
 where here the sum is over all subgraphs $\delta\subset D$ with the same
 set of vertices $V(\delta)=V(D)$ but with fewer edges $E(\delta)\subset E(D)$.

 The Bollob\'as--Riordan--Tutte polynomial is defined for an oriented
 ribbon graph. It can also be characterised in terms of a deletion-contraction
 relation or in terms of a state-sum formula, which generalizes \eqref{Tutte} in the form
 \begin{gather}\label{BRTutte}
 BR_D(x,y,z) =\sum_{\delta \subset D} (x-1)^{b_0(\delta)-b_0(D)} (y-1)^{b_1(\delta)} z^{b_0(\delta)+b_1(\delta)-f(\delta)} ,
 \end{gather}
 with $f(\delta)$ the number of faces of the surface embedding of $\delta$.
 Note that here, for consistency with \eqref{Tutte}, we used $y-1$ instead of $y$ in
 the analogous formula of~\cite{BoRio}: switching to the $y-1$ variable is also advocated
 in the last section of~\cite{BoRio} for symmetry reasons.

As shown in \cite[Theorem~5.4]{Das}, the specialization of the Bollob\'as--Riordan--Tutte polynomial
 \begin{gather*}
 t^{2-2 \#V(D)+\#E(D)} BR_D\big({-}t^4,1-t^{-2}\big(t^2+t^{-2}\big),\big(t^2+t^{-2}\big)^2\big)
 \end{gather*}
gives the Kauffman bracket (for a link projection whose associated connected ribbon graphs is~$D$). It is similarly known that the Jones polynomial is obtained from the $2$-variable Tutte polynomial by specializing to $P_D(-t,-1/t)$. The specialization to the diagonal $P_D(t,t)$ of the Tutte polynomial, on the other hand, gives the Martin polynomial~\cite{Martin}. Such $1$-variable specializations, with values in polynomials or Laurent polynomials, are also invariants of dessins~$D$.

 \begin{lem}\label{TutteGal} The Tutte polynomial $P_D(x,y)$ is a Galois invariant of dessins~$D$,
 namely $P_{\gamma D}(x,y)=P_D(x,y)$ for all $\gamma \in G={\rm Gal}\big(\bar\Q/\Q\big)$.
 Similarly, the Bollob\'as--Riordan--Tutte polynomial $BR_D(x,y,z)$ is a Galois invariant of dessins~$D$.
 \end{lem}

\begin{proof} Note that if $D$ is a dessin, then the quantities $b_0(D)$, $\#E(D)=d$
 and $\#V(D)=m+n$ are Galois invariants and so is $b_1(D)$ by the Euler
 characteristic formula. For $\gamma \in G={\rm Gal}\big(\bar\Q/\Q\big)$ we then have
 \begin{gather*} P_{\gamma D}(x,y) = \sum_{\delta' \subset \gamma D} (x-1)^{b_0(\delta')-b_0(D)} (y-1)^{b_1(\delta')}. \end{gather*}
As in Proposition~\ref{GHopfact}, we can argue that the set of subgraphs $\delta'\subset \gamma D$ with $V(\delta')=V(\gamma D)$ and $E(\delta')\subset E(\gamma D)$ can be identified with the set of subgraphs $\gamma \delta$ where $\delta$ ranges over the
 subgraphs of $D$ with $V(\delta)=V(D)$ and $E(\delta)\subset E(D)$. Thus, we get
 \begin{gather*} P_{\gamma D}(x,y) = \sum_{\gamma \delta \subset \gamma D} (x-1)^{b_0(\gamma \delta)-b_0(D)} (y-1)^{b_1(\gamma \delta)} = P_D (x,y) ,\end{gather*}
 since $b_0(\gamma \delta)=b_0(\delta)$ and $b_1(\gamma \delta)=b_1(\delta)$. The case of the
 Bollob\'as--Riordan--Tutte polynomial $BR_D(x,y,z)$ is similar: we only need to verify that the number of
 faces $f(D)$ of the embedding of $D$ in $\Sigma$ is also a Galois invariant. This is the case (see also
 Lemma~\ref{fiberinvs} below)
 since by the Riemann--Hurwitz formula $\chi(\Sigma)=-d +m +n +r$, where the degree $d$ and
 the ramification indices $m$, $n$, $r$ at $0$, $1$, $\infty$ are all Galois invariants. Writing $\chi(\Sigma)=\# V(D)-\# E(D)+ f(D)$  then shows $f(D)$ is also a Galois invariant.
 \end{proof}

 Note that, as shown explicitly by the state-sum formulae~\eqref{Tutte} and~\eqref{BRTutte},
 these invariants are only sensitive to simple Galois invariants such as $b_0$, $b_1$, for
 subdessins $\delta\subset D$. However, when we apply to this invariant the Birkhoff factorization procedure we obtain more refined invariants that also depend on the quotient dessins $D/\delta$.

We now discuss Rota--Baxter structures. In the case of specializations to
 Kauffman brackets and Jones polynomials, which take values in the algebra~$\cR$
 of one-variable Laurent polynomials, we can use the Rota--Baxter operator~$T$ of
 weight $-1$ given by projection onto the polar part. In the case of Tutte and
 Bollob\'as--Riordan--Tutte polynomials with values in a multivariable polynomial
 algebra, we can consider a Rota--Baxter operator built out of the following
 Rota--Baxter structure of weight $-1$ on the algebra of polynomials, applied
 in each variable. As shown in~\cite{Guo2}, the algebra $\cR=\Q[t]$ with the vector space basis given by
 the polynomials
 \begin{gather}\label{pinpol}
 \pi_n(t) =\frac{t (t+1)\cdots (t+n-1)}{n!},
 \end{gather}
 for $n\geq 1$ and $\pi_0(t)=1$, is a Rota--Baxter algebra of weight $-1$ with
 the Rota--Baxter opera\-tor~$T$ given by the linear operator defined on this basis by
 \begin{gather*} T(\pi_n)=\pi_{n+1}. \end{gather*}
 In the case of multivariable polynomials, Rota--Baxter structures of
 weight $-1$ can be constructed using the tensor product in the category of
 commutative Rota--Baxter algebras described in~\cite{ChuGuo}.

\begin{prop}\label{newGalois} Let $\phi\in\Hom_{{\rm Alg}_\Q}( \cH_\cD, \cR)$ be either the Tutte or
 Bollob\'as--Riordan--Tutte polynomial or one of its specializations to
 Kauffman bracket or Jones polynomial or Martin polynomial. Let $\cR$ be endowed with
 the Rota--Baxter operator~$T$ of weight~$-1$ as above.
 Then the Birkhoff factorization $\phi=(\phi_- \circ S) \star \phi_+$
 determines new Galois invariants $\phi_\pm\in \Hom_{{\rm Alg}_\Q}(\cH_{\cD,G}, \cR_\pm)$
 which depend on both the strongly balanced subdessins $\delta\subset D$ and their
 quotient dessins $D/\delta$.
 \end{prop}

 \begin{proof} We discuss explicitly only the one-variable cases of the Martin polynomial
 specialization $PD(t,t)$ and the Jones polynomial specialization $P_D(-t,-1/t)$
 of the Tutte polynomial. The other cases are not conceptually different but computationally
 more complicated.
 For the Jones specialization $P_D(-t,-1/t)$ is a Laurent polynomial and we
 use the Rota--Baxter operator $T$ of projection onto the polar part. We obtain
 two new invariants $P_{D,\pm}$ with $P_{D,+}(t)$ a polynomial in $t$ and $P_{D,-}(u)$
 a polynomial in $u=1/t$ given by
 \begin{gather*} P_{D,-}(1/t) := -T (P_D(-t,-1/t)+\sum_{\delta\in \cB(D)} P_{\delta,-}(1/t) P_{D/\delta}(-t,-1/t) ), \\
 P_{D,+}(t) :=(1-T)(P_D(-t,-1/t)+\sum_{\delta\in \cB(D)} P_{\delta,-}(1/t) P_{D/\delta}(-t,-1/t)). \end{gather*}
 In the case of the Martin specialization $P_D(t,t)$ of the Tutte polynomial, let $\pi_n$ be
 the linear basis of $\Q[t]$ as in~\eqref{pinpol} and let
 \begin{gather*} P_D(t,t)=\sum_{n\geq 0} a_n(D) \pi_n(t) \end{gather*}
 be the expansion in this basis of the Martin polynomial. Let $u_{m,n,k} \in \Q$ be the combinatorial
 coefficients satisfying $\pi_n \pi_m =\sum_k u(m,n,k) \pi_k$. The Rota--Baxter operator in this case is the linear map $T (\pi_n)=\pi_{n+1}$ so that $T(P_D(t,t))=\sum\limits_{n\geq 1} a_{n-1}(D) \pi_n$. In this case, the new invariants~$P_{D,\pm}(t)$ obtained from the Birkhoff factorization are polynomial
 invariants of the form
 \begin{gather*} P_{D,\pm}(t) :=\sum_{k\geq 0} a_k^\pm(D) \pi_k(t),\\
a_k^-(D):=a_{k-1}(D) +\sum_\delta \sum_{m,n} a_m^-(\delta) a_n(D/\delta) u_{m,n,k-1},\\
 a_k^+(D):=(a_k(D)-a_{k-1}(D)) +\sum_\delta \sum_{m,n} a_m^-(\delta) a_n(D/\delta) (u_{m,n,k}-u_{m,n,k-1}). \end{gather*}
This shows explicitly how passing from the coefficients $a_n(D)$ to the coefficients $a_n^\pm(D)$
incorporates the information on the Galois invariants of all the quotient dessins $D/\delta$ as
well as sub-dessins $\delta$ in a combination dictated by the specific form of the Rota--Baxter operator.
 \end{proof}

\subsection{The semigroup of Belyi-extending maps}\label{BelyiExtSec}

Belyi-extending maps were introduced in \cite{Wood} as a way to obtain new Galois invariants for
the action of the absolute Galois group on dessins d'enfant. They consist of maps $h\colon \P^1\to \P^1$,
defined over~$\Q$, ramified only at $\{ 0,1,\infty \}$, and mapping $\{ 0,1,\infty \}$ to itself. In particular, they have the property that, if $f\colon \Sigma \to \P^1$ is a Belyi map, then the composition $h\circ f$ is still a~Belyi map. This defines an action of the semigroup of Belyi-extending maps on dessins, which commutes with the action of $G={\rm Gal}(\bar\Q/\Q)$.

In the following, we will denote by $\cE$ the semigroup of Belyi-extending maps with the operation of
composition. We define the product in $\cE$ as $\eta_1\cdot \eta_2=\eta_2\circ \eta_1$, for the
convenience of writing semigroup homomorphisms rather than anti-homomorphism later in this
section, since the action of $\cE$ on Belyi maps will be by composition on the right.

Notice that the semigroup $\N$ of self maps of $\bG_m$ (extended to self maps of $\P^1$
ramified at~$0$ and~$\infty$), used in~\cite{CCM} for the construction of the Bost--Connes endomotive
can be seen as a~subsemigroup of the semigroup $\cE$ of Belyi-extending maps.

We can in principle consider two different versions of the semigroup $\cE$, one
as in \cite{Wood}, where we consider all Belyi-extending maps (mapping the set
$\{0,1,\infty \}$ into itself), and one where we consider only those Belyi-extending maps
that map it to itself. The first choice has the advantage of giving rise to a larger set of
new Galois invariants of dessins. In particular, the
only known example of a Belyi-extending map separating Galois-orbits, as shown in~\cite{Wood},
does not satisfy the more restrictive condition. In the more restrictive class
one can always assume, up to a change of coordinates on~$\P^1$, that a Belyi-extending map
sends the points $0$, $1$, $\infty$ to themselves. In each isomorphism class there is a unique
representative with this property. This assumption is convenient when one
considers dynamical properties under compositions, see for instance~\cite{ABEGKM}.
The main construction we discuss works in both settings, but some statements will
depend on choosing the more restrictive class of Belyi-extending maps, as we will see
in the following subsections. We will state explicitly when this choice is needed.

Note that, in principle, one could also consider a larger semigroup $\cE_{\bar\Q}$
consisting of all Belyi maps of genus zero mapping $\{0,1,\infty \}$ into itself. The
construction of the crossed product system we outline below would also work in
this case, but this choice would have the inconvenient property that the absolute
Galois group ${\rm Gal}\big(\bar \Q/\Q\big)$ action would involve both the part
$\cH_\cD$ of the algebra and the semigroup $\cE_{\bar\Q}$,
unlike the Bost--Connes case. However, we will include the case of
$\cE_{\bar\Q}$ and subsemigroups in our discussion, because it will
be useful in presenting in an explicit way a method of removal of certain catastrophic
everywhere-divergence problems that can occur in the partition function
of the quantum statistical mechanical system. We will discuss this in
Section~\ref{RegSysSec} below.

\begin{defn}\label{EetaBeta}Given an element $\eta \in \cE$ we denote by
\begin{gather*}
\cE_\eta=\{ h\in \cE \,|\, \exists\, h'\in \cE, \, h=h' \circ \eta \},
\end{gather*}
the range of precomposition by $\eta$. We also denote by $\cB_\eta$
the subset of Belyi maps that are in the range of composition with a given Belyi-extending map
$\eta\in \cE$, that is,
\begin{gather*}
\cB_\eta=\big\{ f \colon \Sigma \to \P^1 \text{ Belyi }\,|\, \exists\, f'\colon \Sigma \to \P^1, \text{ Belyi}\colon f=\eta\circ f' \big\}.
\end{gather*}
\end{defn}

We consider the Hilbert space $\ell^2(\cE)$ with the standard orthonormal basis $\{ \epsilon_\eta \}$
for $\eta\in \cE$ ranging over the Belyi extending maps. The semigroup $\cE$ acts on $\ell^2(\cE)$
by precomposition, $\mu_\eta \epsilon_{\eta'}=\epsilon_{\eta'\circ \eta}$. With the product in $\cE$
defined as above, this gives a semigroup homomorphism $\mu\colon \cE \to \cB(\ell^2(\cE))$.
We denote by $\Pi_\eta$ the orthogonal projection of $\ell^2(\cE)$ onto the subspace generated
by the elements of $\cE_\eta$.

\subsection{Developing maps}

We review here some general facts about orbifold uniformization and developing maps and apply them to
the studies of multivalued inverses of Belyi maps.

In general (see~\cite{Yos}), an orbifold datum $(X,D)$ consists of a complex manifold $X$, a
divisor $D=\sum_i m_i Y_i$ on $X$ with integer multiplicities $m_i>2$, and hypersurfaces $Y_i$
such that near every point of $X$ there is a neighborhood $\cU$ with a branched
cover that ramifies along the loci $\cU\cap Y_i$ with branching indices $m_i$.

A complex manifold $\tilde X$ with a branched covering map $f\colon \tilde X \to X$ that
ramifies exactly along $Y=\cup_i Y_i$ with indices $m_i$ is called an {\em uniformization of
the orbifold datum~$(X,D)$.}

One might consider the multivalued inverse $\phi\colon X\to \tilde X$ of the
uniformization map $f\colon \tilde X \to X$. This is usually referred to
as the {\em developing map} in the cases where $\tilde X$ is simply connected, though
we will be using the term ``developing map'' more generally here for multivalued inverse maps.
Some concrete developing maps are related to certain classes of differential equations, such
as hypergeometric, Lam\'e, Heun, etc.
In particular, the Belyi maps $f\colon \Sigma \to \P^1$ can be
seen as special cases of uniformization of orbifold data $\big(\P^1,D\big)$ with $D$
supported on $\{0,1,\infty\}$, and we can consider the associated multivalued
developing map $\phi=\phi_f$.

In the case where $\tilde X=\H$ is the upper half-plane and the uniformization map
is the quotient map that realizes an orbifold datum $\big(\P^1,D\big)$ as the quotient of
$\H$ by a Fuchsian group, if we denote by $z$ the coordinate on $\H$ and by $t$ the (affine) coordinate on $\P^1(\C)$, a developing
map $z=\phi(t)$ (period map) is obtained as the ratio $\phi(t)=u_1(t)/u_2(t)$ of two independent non-trivial
solutions of the differential equation (orbifold uniformization equation, see \cite[Proposition~4.2]{Yos})
\begin{gather}\label{orbunif}
 \frac{{\rm d}^2 u}{{\rm d}t^2} + \frac{1}{2} \{ z, t \} u =0,
\end{gather}
where $\{ z, t\}=\{ \phi(t), t \}=\cS(\phi)(t)$ is the Schwarzian derivative
\begin{gather*}
 \cS(\phi)(t) := \left( \frac{\phi^{\prime \prime}}{\phi^\prime} \right)^\prime - \frac{1}{2}
\left( \frac{\phi^{\prime \prime}}{\phi^\prime} \right)^2 = \frac{\phi^{\prime\prime\prime}}{\phi^\prime} -\frac{3}{2}
\left( \frac{\phi^{\prime \prime}}{\phi^\prime} \right)^2.
\end{gather*}

More generally, given a Belyi map $f\colon \Sigma \to \P^1$, we can precompose it with a Fuchsian uniformization $\pi\colon \H \to \H/\Gamma=\Sigma$ and consider the multivalued inverse of the Belyi map $f$ in terms of a developing map for $f\circ \pi$.

Riemann surfaces $\Sigma$ that are obtained from algebraic curves defined over number fields admit
uniformizations $\Sigma=\H/\Gamma$ by Fuchsian groups that are finite index subrgroups of a Fuchsian
triangle group $\Delta$, hence one can view Belyi maps as maps $f\colon \H/\Gamma=\Sigma \to \H/\Delta=\P^1$, see~\cite{CIW}.

In the case of a triangle group $\Delta$ the orbifold uniformization equation
reduces to a hypergeometric equation, hence the resulting $\phi(t)$ can be described in terms of hypergeometric functions.

The classical hypergeometric equation, depending on three parameters $a$, $b$, $c$
\begin{gather}\label{hypergeom}
t(t-1) F^{\prime\prime} + ((a+b-1)t -c) F^\prime+ ab F =0
\end{gather}
has the hypergeometric function
\begin{gather*}
F(a,b,c|t)=\sum \frac{(a)_n (b)_n}{(c)_n} \frac{t^n}{n!}
\end{gather*}
as a solution, with $(x)_n=x(x+1)\cdots (x+n-1)$ the Pochhammer symbol.

Two linearly independent solutions $f_1$, $f_2$, locally defined in sectors near $t=0$,
can be obtained as $\C$-linear combinations of $F(a,b,c|t)$ and $t^{1-c} F(a-c+1,b-c+1,2-c|t)$.
The Schwarz map $\phi(t)=f_1(t)/f_2(t)$ given by the ratio of two independent
solutions maps holomorphically the upper half plane $\H$ to the curvilinear triangle $T$ with
vertices $\phi(0)$, $\phi(1)$, $\phi(\infty)$ and angles $|1-c| \pi$, $|c-a-b| \pi$, and
$|a-b| \pi$. The equation~\eqref{hypergeom} is equivalent to \eqref{orbunif}, for the case of branch
points at $0$, $1$, $\infty$ and these angles. Extending the map by Schwarz reflection, one obtains in this way the developing map, for the uniformization of the orbifold datum $\big(\P^1(\C),D\big)$, with set of orbifold points $\{ 0,1,\infty \}$ and angles as above, with a tessellation of $\H$ (in the hyperbolic case) by copies of $T$ and the group generated by reflections about the edges of~$T$.

In the case where $|1-c|=m_0^{-1}$, $|c-a-b|=m_1^{-1}$, $|a-b|=m_2^{-1}$ are
inverses of positive integers, the Schwarz function $\phi(t)$ can be inverted and
the inverse function $t=t(z)$, the uniformization map, is an automorphic
function for the Fuchsian triangle group $\Delta(m_0,m_1,m_2)$, which is
related to the problem of computing the Hauptmodul for triangle groups~\cite{CIW,DGMS}.
For a Belyi map $f\colon \Sigma=\H/\Gamma \to \P^1=\H/\Delta$ the uniformization
equation can be seen as the pullback of the equation for~$\H/\Delta$.
A formulation of multivalued inverse developing maps of Belyi maps in terms of
log-Riemann surfaces is discussed in~\cite{BisPeMa}.

The properties of developing maps for Belyi maps recalled here give us the following
consequence, which we will be using in our construction of the quantum statistical
mechanical system based on Belyi-extending maps in the next subsections.

\begin{lem}\label{rhoeta}Consider an element $\eta \in \cE$, that is, a Belyi map $\eta\colon \P^1\to \P^1$, defined over~$\Q$, with $\eta(\{ 0,1,\infty \})\subseteq \{ 0,1,\infty \}$. Let $\phi_\eta$ be its developing map as above.

If $\cB_\eta$ is the subset of Belyi maps introduced in Definition~{\rm \ref{EetaBeta}}, then composition $\phi_\eta \circ f$ of any $f\in \cB_\eta$ with the developing map $\phi_\eta$ gives a~map~$f'$ satisfying $\eta \circ f'=f$.

Also, for any $\eta'\in \cE_\eta$, with $\cE_\eta$ as in Definition~{\rm \ref{EetaBeta}}, precomposition of $\eta'$ with $\phi_\eta$ gives a~map~$\eta''$ with $\eta''\circ \eta =\eta'$. We put $\rho_\eta(f):=\phi_\eta \circ f$ for $f\in \cB_\eta$.

For more general Belyi maps~$f$ we define as above the compositions $\eta'\circ \rho_\eta(f)$, for any $\eta'\in \cE_\eta$.
\end{lem}

\begin{proof} Let $\Gamma$ be the uniformizing group of the source orbifold datum
$\big(\P^1, \eta^{-1}(D)\big)$ with $D=\{0,1,\infty\}$ and let $\bar\phi_\eta(t)\colon \big(\P^1,D\big)\to \big(\P^1,\eta^{-1}(D)\big)$, and let $\pi\colon \H \to \H/\Gamma=\P^1$ be the uniformization map for $\big(\P^1, \eta^{-1}(D)\big)$, with $\phi_\pi$ the corresponding developing map.

The multivalued inverse $\phi_\eta$ for the Belyi-extending map $\eta\colon \P^1 \to \P^1$ can then be regarded as the composition $\pi\circ \tilde\phi_\eta$ of the developing map $\tilde\phi_\eta=\phi_\pi \circ \phi_\eta$ for $\eta\circ \pi$ and the uniformization $\pi\colon \H \to \H/\Gamma$ of $(\P^1, \eta^{-1}(D))$. Given $f \in \cB_\eta$, the composition $\phi_\eta \circ f$ is defined and gives
a Belyi map $f'\colon \Sigma \to \P^1$ such that $\eta\circ f'=f$. Thus, the transformation
$\rho_\eta(f)=\phi_\eta \circ f$ is well defined on $f\in \cB_\eta$.

For $\eta'\in \cE_\eta$ similarly we have $\eta'=h'\circ \eta$ for some $h'\in \cE$. Thus, the composition $\eta'\circ \phi_\eta$ is defined and gives an element in $\cE$. Given a Belyi map $f\colon \Sigma \to \P^1$, a Belyi-extending map $\eta$, and a Belyi-extending map $\eta'\in \cE_\eta$, with $\eta'=h'\circ \eta$ for some $h'\in \cE$, we have $\eta'\circ \rho_\eta(f)$ given by the composition $\eta'\circ \phi_\eta\circ f =h'\circ f$.
\end{proof}

\subsection{Belyi-extending semigroup and Belyi functions operators}
Returning to the framework of the end of Section~\ref{HopfGaloisSec}, recall that isometries $\mu_\eta$ of $\ell^2(\cE)$ act upon the basis elements $\epsilon_{\eta'}$ as $\mu_\eta \epsilon_{\eta'}=\epsilon_{\eta'\circ \eta}$. We have $\mu_\eta^* \mu_\eta = {\rm id}$.

Clearly, $e_\eta :=\mu_\eta \mu_\eta^*$ is the projector acting on $\ell^2(\cE)$ as the orthogonal projector $\Pi_\eta$ onto the subspace generated by the basis elements $\epsilon_{\eta'}$ with $\eta'\in \cE_\eta$.

Given a Belyi map $f\colon \Sigma \to \P^1$, let $D=D_f$ be the associated dessin.
We include here the case where $D$ may have multiple connected components.
This case corresponds to branched coverings where~$\Sigma$ can have multiple
components. Given a Belyi-extending map $\beta \in \cE$, we consider the dessin
$\eta(D)$ corresponding to the Belyi map $\eta \circ f$ as in~\cite{Wood}.

\begin{defn}\label{opspiBelyi}Let $\varphi\in \cG_\cD\big(\bar\Q\big)={\rm Hom}_{\rm Alg_\Q}\big(\cH_\cD,\bar\Q\big)$ be a character of the Hopf algebra of dessins. Given a Belyi map $f\colon \Sigma \to \P^1$, we define a linear operator $\pi_\varphi(f)$ by setting
\begin{gather}\label{pifD}
\pi_\varphi (f) \epsilon_\eta = \varphi( \eta(D_f) ) \epsilon_\eta,
\end{gather}
for all $\eta\in \cE$ and with $\eta(D_f)=D_{\eta\circ f}$. For simplicity of notation we shall write $\varphi(\eta\circ f):=\varphi( \eta(D_f) )$ in the above.
\end{defn}

The operator $\pi_\varphi(f)$ is bounded if
\begin{gather}\label{boundedop}
\sup_{\eta\in \cE}|\varphi(\eta(D_f))|\leq C_{\varphi,f},
\end{gather}
for some constant $C_{\varphi,f}>0$.

\subsection{Invariant characters and balanced characters}

We will now discuss simple examples of characters $\varphi$ that satisfy this boundedness
condition.

Let $\iota\colon \bar\Q \hookrightarrow \C$ be a fixed embedding.
Let $\lambda\in \bar\Q$ be an algebraic number of absolute value $|\iota(\lambda)|\leq 1$,
with the property that all its Galois conjugates are also contained in the unit disk. Note that
this condition can be easily achieved, for instance by dividing by a large integer, unlike
the more delicate conditions (such as for Pisot and Salem numbers) where one requires
one of the points in the Galois orbit to remain outside of the unit disk.

\begin{prop}\label{Ginvchar} Let $\lambda\in\bar\Q$ be chosen as above, so that $\iota(\gamma\lambda)$ is
contained in the unit disk for all $\gamma\in G={\rm Gal}\big(\bar\Q/\Q\big)$.
Consider the Galois invariant of dessins given by $\varphi(D)=\iota(\lambda)^{\# E(D)}
=\iota(\lambda)^d$ where~$d$ is the degree of the Belyi map. This determines a character
$\varphi\in \Hom\big(\cH_\cD,\bar\Q\big)$ that descends to a character of the
quotient Hopf algebra $\cH_{\cD,G}$ of Lemma~{\rm \ref{idealHopf}}.

For a Belyi map $f\colon \Sigma \to \P^1$, the operator $\pi_\varphi(f)$
associated to this character satisfies the boundedness condition.
Moreover, for any $\gamma\in G$ the character
$\gamma\circ\varphi \in \Hom(\cH_\cD,\bar\Q)$ given by $\gamma\varphi(D)
=\gamma(\lambda)^{\# E(D)}$ also gives bounded operators $\pi_{\gamma\varphi}(f)$.
\end{prop}

\begin{proof}Consider $\lambda\in\bar\Q$ as above. The map
$\varphi(D) =\iota(\lambda)^{\# E(D)}$ determines an algebra homomorphism
$\varphi\colon \cH_\cD\to \bar\Q$, since for a disjoint union $D=D_1\coprod D_2$,
we have $\# E(D)= \# E(D_1)+ \# E(D_2)$ hence $\varphi(D)=\varphi(D_1)\varphi(D_2)$.
Since the number of edges of a dessin (degree of the Belyi map) is a Galois invariant, the character
descends to the quotient Hopf algebra~$\cH_{\cD,G}$. Given a Belyi function $f\colon \Sigma \to \P^1$,
the operator $\pi_\varphi(f)$ acts on the basis $\epsilon_\eta$ of the Hilbert space~$\ell^2(\cE)$
by $\pi_\varphi(f)\epsilon_\eta = \varphi(\eta(D_f)) \epsilon_\eta$, with $\eta(D_f)=D_{\eta\circ f}$.
We know (see \cite[Proposition~3.5]{Wood}) that if $D_f$ is the dessin corresponding to
a Belyi map $f\colon \Sigma \to \P^1$ and $D_\eta$ is the dessin corresponding to $\eta\colon \P^1\to \P^1$
in $\cE$, then the edges of the dessin $D_{\eta\circ f}$ of the Belyi map $\eta\circ f\colon \Sigma \to \P^1$ are given by $E(\eta(D_f))=E(D_f)\times E(D_\eta)$. Thus, we have
\begin{gather*}
 \varphi(\eta(D_f)) = \iota(\lambda)^{\# E(\eta(D_f))} = \big(\iota(\lambda)^{\# E(D_f)}\big)^{\# E(D_\eta)}.
 \end{gather*}

Under the assumption that $|\iota(\lambda)|\leq 1$, hence $| \iota(\lambda)^{\# E(D_f)} |\leq 1$ we
have \begin{gather*}\big|\big(\iota(\lambda)^{\# E(D_f)}\big)^{\# E(D_\eta)}\big|\leq \big| \iota(\lambda)^{\# E(D_f)} \big|\leq 1\end{gather*} for all
$\eta\in \cE$, hence $\|\pi_\varphi(f)\|\leq 1$. Since we assume that all Galois conjugates
are also in the unit disk, $|\iota(\gamma(\lambda))|\leq 1$ for $\gamma\in G$, the same holds
for the operators $\pi_{\gamma\varphi}(f)$.
\end{proof}

In Proposition~\ref{Ginvchar}, we constructed a character $\varphi \in \Hom_{{\rm Alg}_\Q}\big(\cH_\cD,\bar\Q\big)$ that is invariant with respect to the $G$-action by Hopf algebra homomorphisms of $\cH_\cD$ (by automorphisms of the dual affine group scheme $\cG_\cD$) and descends to a character of the quotient Hopf algebra $\cH_{\cD,G}$. We now consider a setting that is more interesting for the purpose of obtaining the correct intertwining of symmetries and Galois action on the zero-temperature
KMS states, extending the similar properties of the Bost--Connes system. To this purpose we focus on
characters of the Hopf algebra $\cH_\cD$ that are not $G$-invariant, but that satisfy the expected
$G$-equivariance condition intertwining the $G$-action on the source $\cH_\cD$ with the $G$-action on the
target $\bar\Q$.

\begin{defn}\label{balchardef}A character $\varphi\in \Hom_{{\rm Alg}_\Q}\big(\cH_\cD,\bar\Q\big)$ is called {\em balanced} if it satisfies the identity
\begin{gather}\label{balchar}
\varphi(\gamma D) = \gamma \varphi(D),
\end{gather}
for all $\gamma\in G={\rm Gal}\big(\bar\Q/\Q\big)$ and all dessins~$D$, where $D\mapsto \gamma D$ is the $G$-action on dessins and $\varphi(D)\mapsto \gamma \varphi(D)$ is the Galois action on $\bar\Q$.
\end{defn}

The following example shows that the set of balanced characters is non-empty, although
the example constructed here is not computationally feasible, since it assumes an a priori
choice of a~set of representatives of the $G$-orbits on the set $\cD$ of dessins, which
in itself requires an explicit knowledge of these orbits. The question of constructing
an explicit map $D\mapsto \varphi(D)$ of dessins to $\bar\Q$ that intertwines the
${\rm Gal}\big(\bar\Q/\Q\big)$ actions remains a very interesting problem.

\begin{lem}\label{balcharlem}Let $\cR=\{ D \}$ be a set of representatives of the $G$-orbits on the set of dessins. Let~$\{ \lambda_D \}$ be a set of algebraic numbers with $|\iota(\gamma\lambda_D)|\leq 1$
for all $\gamma\in G$ and $\deg(\lambda_D)=\# {\rm Orb}(D)$. Then the formula
\begin{gather*}
\varphi(\gamma D)=(\gamma\lambda_D)^{\# E(D)}
\end{gather*}
determines a balanced character of $\cH_\cD$, with the property that the linear operators $\pi_\varphi(\gamma f)$ on~$\ell^2(\cE)$ are bounded for any Belyi function~$f$ and all $\gamma\in G$.
\end{lem}

\begin{proof} Let $d_D=\# {\rm Orb}(D)$ be the length of the corresponding orbit. By the orbit-stabilizer theorem this is also the cardinality of the set of cosets $d_D=\# G/{\rm Stab}(D)$.
For each $D\in \cR$, choose an algebraic number $\lambda_D\in \bar\Q$ with the property
that $\# {\rm Orb}(\lambda_D) = \deg(\lambda_D)=[\Q(\lambda_D):\Q]$ satisfies
$\deg(\lambda_D)=d_D$. We then set
\begin{gather*} \varphi(\gamma D) = \gamma(\lambda_D)^{\# E(D)}. \end{gather*}
Since $\# {\rm Orb}(D) =\# {\rm Orb}(\lambda_D)$ and we can identify the action of $G$
on ${\rm Orb}(D)$ with the left multiplication on the cosets $G/{\rm Stab}(D)$, the
map $\varphi$ bijectively maps ${\rm Orb}(D)$ to ${\rm Orb}(\lambda_D)$ in such a way
that by construction $\varphi(\gamma D)= (\gamma(\lambda_D))^{\# E(D)} =
\gamma\big(\lambda_D^{\# E(D)}\big)$, hence $\varphi$ is a balanced character.

Possibly after dividing by a sufficiently large integer, we can assume
that $\lambda_D$ and all its Galois conjugates lie inside the unit disk, so that the sequence
$\big\{ \gamma(\lambda_D)^n \big\}_{n\in \N}$ is bounded in the
$\ell^\infty$-norm, for all $\gamma\in G$.
Consider then the operators $\pi_\varphi(f)$ associated to the character $\varphi$,
for Belyi maps $f$. These act on $\ell^2(\cE)$ by
$\pi_\varphi(f)\epsilon_\eta = \varphi(\eta(D)) \epsilon_\eta$, where $D=D_f$ is the
dessin associated to the Belyi map.

Since all the Belyi-extending maps $\eta\in \cE$ are defined over $\Q$, we have $\gamma \eta(D)=\eta (\gamma D)$ for all $\gamma\in G$ and $\eta {\rm Orb}(D)={\rm Orb}(\eta(D))$. If we take
$\eta(D)$ as the representative for ${\rm Orb}(\eta(D))$, we can write
\begin{gather*} \pi_\varphi(f)\epsilon_\eta = \lambda_{\eta(D)}^{\# E(\eta(D))} \epsilon_\eta. \end{gather*}
Since the algebraic numbers $\gamma \lambda_{\eta(D)}$ lie in the unit disk,
for all $\gamma \in G$, all $D\in \cD$ and all $\eta\in \cE$, the operators $\pi_\varphi(\gamma f)$
are all bounded.
\end{proof}

\subsection{Quantum statistical mechanics of Belyi-extending maps}\label{BelyiQSMsec}

We will now proceed to the construction of a quantum statistical mechanical
system associated to dessins with the $G$-action and Belyi-extending maps.

\begin{lem}\label{sigmarhoeta}Let $f\colon \Sigma\to \P^1$ a Belyi map and $\eta\in \cE$ a Belyi-extending map. Denote by $\sigma_\eta(f)=\eta\circ f$ the action of the semigroup $\cE$ on Belyi maps by composition.

For $\rho_\eta$ defined as in Lemma~{\rm \ref{rhoeta}}, we have $\rho_\eta (\sigma_\eta(f)) =f$ for all Belyi maps $f$ and $\sigma_\eta(\rho_\eta(f))\allowbreak =f$ for $f\in \cB_\eta$.

Let $\pi_\varphi (f)$ be as in~\eqref{pifD} satisfying the boundedness condition~\eqref{boundedop}. For
any Belyi map~$f$, the operator $\pi_\varphi(\rho_\eta(f))$ acts on the range of the
projection $e_\eta$ in $\ell^2(\cE)$ as
\begin{gather*}
\pi_\varphi(\rho_\eta(f)) e_\eta \epsilon_{\eta'}=
\begin{cases}
\varphi( \eta' \circ \phi_\eta \circ f ) \epsilon_{\eta'}, & \eta'\in \cE_\eta, \\ 0, & \text{otherwise}.
\end{cases}
\end{gather*}
The operators $\pi_\varphi (f)$ of \eqref{pifD} satisfy the relations
\begin{gather*}
\mu_\eta^* \pi_\varphi (f) \mu_\eta = \pi_\varphi(\sigma_\eta(f)) \qquad \text{and} \qquad
\mu_\eta \pi_\varphi (f) \mu_\eta^* = \pi_\varphi(\rho_\eta(f)) e_\eta .
\end{gather*}
\end{lem}

\begin{defn}\label{crossdessins}Let $\cH_\cD$ be the Hopf algebra of dessins, seen as a commutative algebra over $\Q$. Denote by $\cA_{\cH_\cD, \cE}$ the non-commutative algebra over~$\Q$ generated by~$\cH_D$ as a commutative algebra (that is, by the dessins~$D$, or equivalently by the Belyi functions~$f$) and by the isometries~$\mu_\eta$, with the relations $\mu_\eta^*\mu_\eta=1$, $\mu_\eta^* f \mu_\eta = \sigma_\eta(f)$
and $\mu_\eta f \mu_\eta^*=\rho_\eta(f) e_\eta$ with $e_\eta=\mu_\eta\mu_\eta^*$.
\end{defn}

We will construct a time evolution on the algebra $\cA_{\cH_\cD, \cE}$ using invariants
of Belyi-extending maps that behave multiplicatively under composition. The following
result follows directly from the construction of the algebra $\cA_{\cH_\cD, \cE}$ in
Definition~\ref{crossdessins}.

\begin{lem}\label{Etimeev}Let $\Upsilon\colon \cE \to \N$ be a semigroup homomorphism. Then setting $\sigma_t(\mu_\eta)=\Upsilon(\eta)^{{\rm i}t} \mu_\eta$ and $\sigma_t(D)=D$ for all $\eta\in \cE$ and all $D\in \cD$ defines a time evolution
$\sigma\colon \R \to {\rm Aut}(\cA_{\cH_\cD, \cE})$. Given a character $\varphi\in
\Hom_{{\rm Alg}_\Q}\big(\cH_\cD,\bar\Q\big)$ such that the operators $\pi_\varphi(f)$ are bounded for all
Belyi maps~$f$, the time evolution $\sigma_t$ is implemented in the resulting representation of
$\cA_{\cH_\cD, \cE}$ on $\ell^2(\cE)$ by the Hamiltonian $H\epsilon_\eta = \log \Upsilon(\eta) \epsilon_\eta$, with partition function given by the (formal) Dirichlet series
\begin{gather}\label{ZetaE}
 Z(\beta)= \Tr\big({\rm e}^{-\beta H}\big)= \sum_{\eta \in \cE} \Upsilon(\eta)^{-\beta} =\sum_{n\geq 1}
\#\{ \eta\in \cE\,|\, \Upsilon(\eta)=n \} n^{-\beta}.
\end{gather}
\end{lem}

\begin{proof} The $\sigma_t$ defined as above gives a time evolution on $\cA_{\cH_\cD, \cE}$
because of the multiplicative property of~$\Upsilon$ under composition in $\cE$. A character $\varphi$ with the boundedness condition on the operators $\pi_\varphi(f)$ determines a representations $\pi_\varphi$ of the algebra $\cA_{\cH_\cD, \cE}$
by bounded operators on the Hilbert space $\ell^2(\cE)$, where the generators $D$ act as
$\pi_\varphi(f)$ with $f$ the Belyi map associated to the dessin $D$ and the isometries
$\mu_\eta$ act as shifts on the basis, $\mu_\eta \epsilon_{\eta'}=\epsilon_{\eta'\circ\eta}$.
In this representation we see directly that the time evolution is implemented by the Hamiltonian
$H$ with the partition function given by the formal Dirichlet series \eqref{ZetaE}.
\end{proof}

\begin{rem}\label{subsemi}Note that for any given choice of a subsemigroup $\cE'\subset \cE$ of the semigroup of Belyi-expanding maps, one can adapt the construction above and obtain an
algebra $\cA_{\cH_\cD,\cE'}$ generated by $\cH_\cD$ and the isometries $\mu_\eta$ with
$\eta\in \cE'$ and its representation on $\ell^2(\cE')$. A semigroup homomorphism
$\Upsilon\colon \cE'\to \N$ then determines a time evolution with partition function as in~\eqref{ZetaE} with the summation over $\eta\in \cE'$.
\end{rem}

The reason for considering such restrictions to subsemigroups $\cE'\subset \cE$ of the system $(\cA_{\cH_\cD,\cE},\sigma_t)$ is this: we will be able to obtain from the system of Definition~\ref{crossdessins} and Lemma~\ref{Etimeev} a~speciali\-zation that recovers the original Bost--Connes system.

\begin{prop}\label{BCfromE}Let $\cE_{BC}\subset \cE$ be the subsemigroup of Belyi-extending maps $\eta$
such that $\eta\colon 0 \mapsto 0$ and the ramification is maximal at $0$, that
is, $m=1$ and $n=d$. Consider the subalgebra of $\cA_{\cH_\cD,\cE}$ obtained by restricting
the semigroup to $\cE_{BC}$ and restricting the Belyi functions $f\colon \Sigma \to \P^1$
in $\cH_\cD$ to maps ramified only at $0$ and $\infty$.

Let $\Upsilon\colon \cE\to \N$ be the semigroup homomorphism given by the degree $d$ of the Belyi-expanding map, and let $\sigma_t$ be the associated time evolution. There is a choice of a balanced character $\varphi\colon \cH_\cD \to \bar\Q$ such that the restriction of the resulting time evolution to this subalgebra is the original Bost--Connes quantum statistical mechanical system.
\end{prop}

\begin{proof} Choose a balanced character $\varphi\colon \cH_\cD \to \Q^{ab}\subset \bar\Q$ constructed as in Proposition~\ref{balcharlem}, with $\lambda_D^{\#E(D)}=\zeta_d$ given by a primitive root of unity of order $d=\# E(D)$, when $D$ is the dessin of a Belyi function ramified only at $0$ and $\infty$.

These Belyi functions have $g(\Sigma)=0$ and up to a change of coordinates on $\P^1$ we
can assume they map $0$, $1$, $\infty$ to themselves. By the Riemann--Hurwitz formula, the ramification
points~$0$ and~$\infty$ have maximal ramification, $f^{-1}(0)=\{ 0 \}$ and $f^{-1}(\infty)=\{ \infty \}$, so $m=1$ while $n=\# f^{-1}(1)=d$. Similarly, for a Belyi-extending map with $m=\#\eta^{-1}(0)=1$ so that
$\eta^{-1}(1)=\deg(\eta)=d$. Thus, when restricting the system to this subset of Belyi maps and the sub\-group~$\cE_{BC}$ of Belyi-extending maps, the multiplicity of $\Upsilon(\eta)=d$ is one, since for
these maps the dessin~$D$ or~$D_\eta$ is the unique bipartite tree with
a single white colored vertex and~$d$ black colored vertices attached to the white vertex.

We can then identify the basis $\epsilon_\eta$ of $\ell^2(\cE)$ with the basis $\epsilon_d$
of~$\ell^2(\N)$ on which the operators associated to these Belyi maps act as
$\pi_\varphi(f) \epsilon_d = \zeta_{\deg(f)}^d \epsilon_d$ and the isometries $\mu_\eta$
act as the isometries $\mu_{\deg(\eta)}\epsilon_d =\epsilon_{d\deg(\eta)}$ of the Bost--Connes
system, satisfying the relations of the Bost--Connes algebra, after identifying
$\zeta_d=\alpha(e(1/d))$ for $\{ e(r) \}_{r\in \Q/\Z}$ the generators of the Bost--Connes algebra
and $\alpha\in \hat\Z^*=\Aut(\Q/\Z)$ a chosen embedding of~$\Q/\Z$ in~$\C$.
\end{proof}

\looseness=-1 In the statement of Lemma~\ref{Etimeev} above we have regarded the partition function purely as a formal Dirichlet series. In order to proceed to the consideration of Gibbs states and limiting zero-temperature states, however, we need a setting where this Dirichlet series converges for
sufficiently large $\beta >0$. As we discuss in detail in the next subsection, due to the
particular nature of semigroup homomorphisms $\Upsilon\colon \cE\to \N$ from the
semigroup of Belyi-expanding maps to the integers, the Dirichlet series~\eqref{ZetaE}
tends to be everywhere divergent, except in the case where the system is
constructed using the subsemigroup $\cE_{BC}$ of $\cE$ that recovers the
Bost--Connes system.

After discussing semigroup homomorphisms from the semigroup of Belyi-expanding maps, we return to the construction of the quantum statistical mechanical system and we show how Definition~\ref{crossdessins}
and Proposition~\ref{Etimeev} can be modified to avoid this divergence problem.

\subsection{Semigroup homomorphisms of Belyi-extending maps}\label{SemiHomSec}

We discuss here semigroup homomorphisms with source the semigroup of
Belyi-extending maps with the operation of composition. In particular, we calculate the
behaviour of the multiplicities associated to these semigroup homomorphisms,
namely the size of the level sets.

As a simple illustration of the most severe type of divergence one
expects to encounter in the partition functions, we consider first the case of the larger
semigroup $\cE_{\bar\Q}$ of Belyi maps of genus zero mapping $\{0,1,\infty\}$ into
itself, and the subsemigroup $\cE_{\bar\Q,0}$ of Belyi polynomials, which have
bipartite planar trees as dessins.

\begin{lem}\label{deghom}The map $\Upsilon\colon \cE_{\bar\Q} \to \N$ given by the degree $\Upsilon(\eta)=\deg(\eta)$
is a semigroup homomorphism. The level sets of its restriction to $\cE_{\bar\Q,0}$ are given by{\samepage
\begin{gather*}
\cE_{\Upsilon,d}:= \{ \eta\in \cE\,|\, \deg(\eta)=d \} = \{ \text{bipartite trees } T\,|\, \#E(T)=d \},
\\
 \# \cE_{\Upsilon,d} = 2 (d+1)^{d-1}. \nonumber\end{gather*}}
\end{lem}

\begin{proof} The degree of a Belyi map $\eta\in \cE$ is $\deg(\eta)=\# E(D_\eta)$,
the number of edges of its dessin. For Belyi maps $\eta\colon \P^1\to \P^1$ in $\cE_{\bar\Q,0}$ the
dessin $D_\eta$ is a bipartite tree. Since the bipartite structure on a tree is uniquely
determined by assigning the color of a single vertex, the number of elements in $\cE_{\Upsilon,d}$
can be identified with twice the number of trees on $d$ edges, which is equal to $(d+1)^{d-1}$.
\end{proof}

\begin{rem}\label{Zdeg}For $\Upsilon\colon \cE_{\bar\Q,0} \to \N$ given by the degree $\Upsilon(\eta)=\deg(\eta)$, the formal Dirichlet series~\eqref{ZetaE}
\begin{gather*}
2 \sum_{d\geq 1} (d+1)^{d-1} d^{-\beta}
\end{gather*}
is everywhere divergent.
\end{rem}

When we restrict ourselves to the subsemigroup $\cE \subset \cE_{\bar\Q}$ of Belyi-extending maps
and to the corresponding $\cE_0\subset \cE_{\bar\Q,0}$, the multiplicities decrease, since
many bipartite trees do not correspond to Belyi maps defined over $\Q$. One can see
that there are choices of subsemigroups $\cE'$ of the semigroup $\cE$ of Belyi-extending maps
for which the partition function does not have this dramatic everywhere-divergence problem.
We give an example of such a subsemigroup in Lemma~\ref{onecircuit} below. In such cases,
we can just use the quantum statistical mechanical system constructed in Section~\ref{BelyiQSMsec}
above as the correct generalization of the Bost--Connes system. However, subsemigroups
of $\cE$ with this property tend to be too small to achieve separation of Galois orbits. For this
reason, we want to allow larger semigroups of Belyi maps and we need to introduce an
appropriate method that will cure possible everywhere-divergence phenomena in the
partition function.

In the example above we have considered the degree homomorphism.
One can obtain other semigroup homomorphisms $\Upsilon\colon \cE_{\bar\Q} \to \N$ (with restrictions
$\Upsilon\colon \cE \to \N$) by composing
the degree homomorphism with an arbitrary semigroup homomorphism $\Psi\colon \N \to \N$.
Since $\N$ is the free abelian semigroup generated by primes, a homomorphism
$\Phi$ can be obtained by assigning to each prime $p$ a number $\Psi(p)\in \N$. It is easy
to check, however, that any homomorphism $\Psi\circ \Upsilon$ obtained in this way
will not compensate for the everywhere divergence problem of Remark~\ref{Zdeg}. This
leads to the natural question of whether there are other interesting semigroup
homomorphism with source $\cE$ that do not factor through the degree map.
If one restricts to the subsemigroup $\cE_{\bar\Q,0}$ as in Lemma~\ref{deghom},
for which the dessins are trees, then the length of the unique path in the tree
from~$0$ to~$1$ is an example of such a morphism. We focus here below on
a semigroup homomorphism to a non-abelian semigroup of matrices.

In particular, we will be observing more closely how other combinatorial data of dessins
behave under the composition of maps, such as the number of W/B vertices in the
bipartition (the ramification indices $m$ and $n$ of the map at the points~$0$ and~$1$).
This will identify some interesting (non-abelian) semigroup laws, and resulting
semigroup homomorphisms.

\begin{prop}\label{Mat2}
Consider the subsemigroup $\cE'_{\bar\Q}$ of Belyi maps of genus zero mapping
the points $0$, $1$, $\infty$ to themselves.
The map $\Upsilon\colon \cE'_{\bar\Q} \to M_2^+(\Z)$ given by
\begin{gather}\label{Mat2hom}
\Upsilon(\eta) = \begin{pmatrix} d & m-1 \\ 0 & 1
\end{pmatrix},
\end{gather}
where $d=\deg(\eta)$ and $m=\# \eta^{-1}(0)$, is a semigroup homomorphism.
The same holds for the map $\Upsilon\colon \cE'_{\bar\Q} \to M_2^+(\Z)$ as above, with $m$
replaced by $n=\eta^{-1}(1)$ and for the map $\Upsilon\colon \cE'_{\bar\Q} \to M_3^+(\Z)$
\begin{gather*} \Upsilon(\eta) = \begin{pmatrix} d & m-1 & n-1 \\ 0 & 1 & 0 \\
0 & 0 & 1\end{pmatrix} . \end{gather*}
The level sets of the map \eqref{Mat2hom} restricted to the subsemigroup $\cE_{\bar\Q,0}$
of tree dessins are given by
\begin{gather} \cE'_{\Upsilon, d,m} :=\left\{ \eta\in \cE_{\bar\Q,0} \,|\, \Upsilon(\eta)=\begin{pmatrix} d & m-1 \\ 0 & 1
\end{pmatrix} \right\} = \left\{ \text{bipartite trees } T \,|\, \begin{array}{l} \# E(T)=d, \\ \# V_0(T)=m \end{array} \right\},\nonumber\\
 \label{dmtrees}
\# \cE'_{\Upsilon, d,m} = m^{d+1-m} (d+1-m)^{m-1}.
\end{gather}
\end{prop}

\begin{proof} We will discuss the map \eqref{Mat2hom}: the other cases are analogous.
In restricting to the subsemigroup $\cE'_{\bar\Q}$ of $\cE_{\bar\Q}$ we are
considering the more restrictive condition on the Belyi maps, that map the
set $\{0,1,\infty\}$ to itself rather than into itself, and after a change of coordinates
on~$\P^1$, we are assuming that the points $0$, $1$, $\infty$ are mapped
to themselves by these Belyi maps.

For a composition $\eta'\circ \eta$ of two such Belyi maps,
the preimage $(\eta'\circ \eta)^{-1}(0)$ consists of the preimage $(\eta')^{-1}(0)$ of
the point $0$ in the preimage $\eta^{-1}(0)$ and of the preimages $(\eta')^{-1}(u)$
for each of the remaining $m-1$ points $u \in \eta^{-1}(0)$. The first set
$(\eta')^{-1}(0)$ consists of $m'$ points while the second set consists of
$d' (m-1)$ points. Thus, we have $\# (\eta'\circ \eta)^{-1}(0) = d' (m-1) + m'$.
It remains to check that the product in the semigroup $M_2^+(\Z)$
\begin{gather*}
\begin{pmatrix} d' & m'-1 \\ 0 & 1
\end{pmatrix} \cdot \begin{pmatrix} d & m-1 \\ 0 & 1
\end{pmatrix} = \begin{pmatrix} d' d & d'(m-1) + m'-1 \\ 0 & 1
\end{pmatrix}
\end{gather*}
is equal to $\Upsilon(\eta'\circ \eta)$. When restricting to the subsemigroup $\cE_{\bar\Q,0}$,
the counting of the size of the level sets is based on the number of bipartite trees with $m$
white and $n$ black vertices, which is equal to $m^{n-1} n^{m-1}$ and the fact that the
total number of vertices is $m+n=d+1$.
\end{proof}

The semigroup homomorphism \eqref{Mat2hom} therefore captures the
semigroup law satisfied by the ramification indices $m$ and $n$. Note
that the determinant homomorphism of semigroups $\det\colon M_2^+(\Z) \to \N$
produces the semigroup homomorphism given by the degree
\begin{gather*}
\det \circ \Upsilon(\eta)=\deg(\eta).
\end{gather*}

We can use the semigroup homomorphism constructed above to
identify choices of sufficiently small subsemigroups $\cE'$ of
$\cE$ that do not have everywhere-divergence phenomena in the
partition function. The following shows that the genus zero single cycle
normalized Belyi maps considered in~\cite{ABEGKM} provide an
example of such a semigroup.

\begin{lem}\label{onecircuit} Consider genus zero Belyi maps $\eta\colon \P^1 \to \P^1$, normalized so
that they fix the points $0$, $1$, $\infty$,
such that the corresponding dessin $D_\eta$ has a single cycle.
They form a subsemigroup $\cE' \subset \cE$ of the semigroup
of Belyi-extending maps with the property that the partition function~\eqref{ZetaE}
\begin{gather*} Z(\beta)= \sum_{\eta \in \cE'} \Upsilon(\eta)^{-\beta} \end{gather*}
of the associated quantum statistical mechanical system for the
degree homomophism $\Upsilon(\eta)=\deg(\eta)$ is convergent for $\beta>2$.
\end{lem}

\begin{proof} The normalized genus zero single cycle Belyi maps are a particular case of the
more general class of conservative (or critically fixed) rational maps~\cite{Cord}, where
one assumes that the critical points are also fixed points. The genus zero and single
cycle condition correspond to requiring that the ramification indices $m$, $n$, $r$ at $0$, $1$, $\infty$
satisfy $2d+1 =m+n+r$. One can see that these maps form a semigroup using the
semigroup law of Proposition~\ref{Mat2} for the ramification indices. Indeed, for a
composition $\eta\circ \eta'$ of two such maps we have
$2 dd'+1= d(m'-1)+m + d(n'-1)+n+d(r'-1)+r=2 dd' +d -3d +2d+1$. As observed in~\cite{ABEGKM}
it is known that all the normalized genus zero single cycle Belyi maps are defined over $\Q$,
so this is a subsemigroup of the semigroup $\cE$ of Belyi-extending maps. The number of
normalized genus zero single cycle Belyi maps of a given degree $d$ is computed in
 \cite[Corollary~2.8 and Remark~2.9]{ABEGKM} and is of the form
\begin{gather*} N(d)=\frac{1}{12} \big(d^2 +4d -c\big), \end{gather*}
where the constant $c$ takes one of the following values
\begin{gather*} c= \begin{cases} 5, & d\equiv 1 \mod 6, \\
8, & d\equiv 4 \mod 6, \\
9, & d\equiv 3,5 \mod 6 ,\\
12, & d\equiv 0,1 \mod 6.
\end{cases} \end{gather*}
Thus, the partition function \eqref{ZetaE} for the semigroup $\cE'$ for the
degree homomorphism $\Upsilon(\eta)=\deg(\eta)$
\begin{gather*} Z(\beta)= \sum_{\eta \in \cE'} \Upsilon(\eta)^{-\beta} =\sum_{d\geq 1}
\#\{ \eta\in \cE'\,|\, \deg(\eta)=d \} d^{-\beta} =\sum_{d\geq 1} N(d) d^{-\beta} \end{gather*}
is convergent for $\beta>2$.
\end{proof}

As in \cite{Wood}, given a dessin $D$ with Belyi function $f\colon \Sigma \to \P^1$
and a Belyi-extending map $\eta\in\cE$, we denote by $\eta(D)$ the dessin
of the composite function $\eta\circ f$.

\begin{cor}\label{KMSconv}Let $\cE'\subset \cE$ be a subsemigroup of Belyi-extending maps for which
the partition function $Z(\beta)= \sum_{\eta \in \cE'} \Upsilon(\eta)^{-\beta}$ is
convergent for sufficiently large $\beta$. Consider the quantum statistical
mechanical system of Section~{\rm \ref{BelyiQSMsec}} with
$\varphi\in \Hom_{{\rm Alg}_\Q}\big(\cH_\cD,\bar\Q\big)$ a~character
satisfying the boundedness condition. Then all the invariants~$\varphi(\eta(D))$,
for $\eta\in \cE'$ and dessins~$D$, occur as values of zero-temperature KMS states.
\end{cor}

\begin{proof}If $\cE'$ is a subsemigroup of Belyi-extending maps with partition function
that converges for large $\beta$, as in the case of Lemma~\ref{onecircuit}, then
the quantum statistical mechanical system of Section~\ref{BelyiQSMsec} has low temperature KMS
states of the form
\begin{gather*} \psi_{\beta,\varphi}(X) = Z(\beta)^{-1} \sum_{\eta\in \cE'} \big\langle \epsilon_\eta, X {\rm e}^{-\beta H} \epsilon_\eta \big\rangle, \end{gather*}
for elements $X$ of the crossed product algebra $\cA_{\cH_\cD,\cE'}$ and in particular
\begin{gather*} \psi_{\beta,\varphi}(D) = Z(\beta)^{-1} \sum_{\eta\in \cE'} \big\langle \epsilon_\eta, \pi_\varphi(f) {\rm e}^{-\beta H}
\epsilon_\eta \big\rangle = Z(\beta)^{-1} \sum_{\eta\in \cE'} \varphi(\eta(D)) \Upsilon(\eta)^{-\beta}, \end{gather*}
where $f$ is the Belyi map with dessin $D$ and $\varphi\in \Hom_{{\rm Alg}_\Q}\big(\cH_\cD,\bar\Q\big)$
is a character satisfying the boundedness condition. The zero-temperature KMS states are then given by
\begin{gather*} \psi_{\infty,\varphi}(X) = \lim_{\beta \to \infty} \psi_{\beta,\varphi}(X) =\varphi(X) \end{gather*}
for $X\in \cA_{\cH_\cD,\cE'}$. In particular, we consider elements of the form
$X=\mu_\eta^* \pi_\varphi(f) \mu_\eta$ for which we have
\begin{gather*} \psi_{\infty,\varphi}(\mu_\eta^* \pi_\varphi(f) \mu_\eta)=\langle \epsilon_{\rm id}, \pi_\varphi(f)
\epsilon_{\rm id} \rangle =\varphi(\eta(D)). \tag*{\qed}
\end{gather*}\renewcommand{\qed}{}
\end{proof}

In \cite{Wood} the Belyi-extending maps are used to construct new Galois invariants of dessins,
in the form of invariants of the form $\varphi(\eta(D))$, where $D$ is a given dessin, $\varphi$ is
a Galois invariant, and $\eta\in \cE$ ranges over the Belyi-extending maps. In our setting, these
invariants occur as values of zero-temperature KMS states. However,
if the subsemigroup $\cE'$ is too small (as in the case of Lemma~\ref{onecircuit}) one
does not expect that the invariants $\varphi(\eta(D))$ would have good properties with
respect to separating different Galois orbits of dessins. Thus, it is preferable to develop
a way to extend the construction of Section~\ref{BelyiQSMsec} to obtain a
quantum statistical mechanical system that can be used in cases of larger semigroups
of Belyi-extending maps, for which the partition function \eqref{ZetaE} may have the
type of everywhere-divergence problem encountered in the case of Remark~\ref{Zdeg}.

This observation was suggested to us by Lieven Le Bruyn:
examples include the case where one restricts to the trees of dynamical Belyi polynomials. This
semigroup contains several free subsemigroups, such as all Belyi polynomials of fixed degree $d > 2$ that
form a free semigroup. It is possible to take $1$ as a leaf-vertex, and consider the subsemigroup
of such dynamical Belyi polynomials having the same Julia set. (If two polynomials have different
Julia sets, then their forward orbit is dense in the plane by \cite{Stank}.)
Then this subsemigroup acts on the
inverse images of $1$ like the action of the power maps on the roots of unity. However, even this
subsemigroup is likely to be too large.

To the purpose of analyzing how to treat the everywhere divergent cases, we return to a consideration of the model
case of the semigroup $\cE_{\bar\Q,0}$ of Belyi maps with composition considered in Lemma~\ref{deghom}
and Remark~\ref{Zdeg}, for which we know that everywhere-divergence occurs. This semigroup
can be equivalently seen as a semigroup of bipartite trees with a product
operation that reflects the composition of maps, see \cite{AdrZvo}.

\subsection{Extended system, partial isometries, and partition function}\label{RegSysSec}

In this subsection we present a method for curing the type of everywhere-divergence
problems in the partition function \eqref{ZetaE} that occur in the example of the
semigroup $\cE_{\bar\Q,0}$ in Remark~\ref{Zdeg}. For the purpose of clarity,
we illustrate how the method works in the case of this semigroup.
A~similar method, mutatis mutandis, can be applied to other semigroups with
similar divergence phenomena in the partition function.

We use the semigroup homomorphism~\eqref{Mat2hom} in order to modify the
quantum statistical mechanical system of Proposition~\ref{Etimeev} in such
a way that the Dirichlet series of the resulting partition function becomes convergent
for large $\beta>0$.

In the process, we will have to slightly modify our algebra:
 the isometries $\mu_\eta$ with $\mu_\eta^* \mu_\eta=1$
and $\mu_\eta \mu_\eta^* =e_\eta$ will be replaced by partial isometries (which we will still
call $\mu_\eta$) with $\mu_\eta \mu_\eta^* =e_\eta$ and $\mu_\eta^* \mu_\eta=\tilde e_\eta$
where both $e_\eta$ and $\tilde e_\eta$ are projectors.

Actually, we will construct
the one-parameter family of such systems, depending on a parameter $\theta>0$ with
$\theta\in \R\smallsetminus \Q$. We can think of this additional parameter as
a regularization parameter for the original system $(\cA_{\cH_\cD,\cE},\sigma_t)$
that eliminates the divergence of the partition function.

\begin{defn}\label{Otheta}Given a $\theta\in \R^*_+ \smallsetminus \Q^*_+$, let $\Omega_\theta$ denote the
set $\Omega_\theta =\N \theta +\Z_+ = \{ a\theta+b \,|\, a\in \N, \, b\in \Z_+ \}$.
\end{defn}

\begin{lem}\label{pisoms}Let the semigroup be chosen as $\cE=\cE_{\bar \Q,0}$.
Consider the Hilbert space $\ell^2(\cE \times \Omega_\theta)$ with the standard
orthonormal basis $\{ \epsilon_{\eta,\omega} \}_{\eta\in \cE, \lambda\in \Omega_\theta}$.
Let $\mu_\eta$ be the operators
\begin{gather}\label{muetaL}
\mu_\eta \epsilon_{\eta',\lambda}= \begin{cases} \epsilon_{\eta'\circ \eta, \alpha_\eta^{-1}(\lambda)}, & \text{if}\
\alpha_\eta^{-1}(\lambda)\in \Omega_\theta, \\ 0, & \text{otherwise} \end{cases}
\end{gather}
where
\begin{gather*} \alpha_\eta = \begin{pmatrix} d & m-1 \\ 0 & 1 \end{pmatrix} \end{gather*}
is the image of $\eta$ under the semigroup homomorphism~\eqref{Mat2hom}, with
$d=\deg(\eta)$ and $m=\#\eta^{-1}(0)$, and the matrix acts on $\lambda$ by
fractional linear transformations.

The $\mu_\eta$ are partial isometries with $\mu_\eta \mu_\eta^* =e_\eta$ (the projector defined
as before), and $\mu_\eta^* \mu_\eta =\tilde e_\eta$ is another projector defined by
\begin{gather*}
 \tilde e_\eta \epsilon_{\eta',\lambda}= \begin{cases} \epsilon_{\eta',\lambda}, & \text{if} \
\alpha_\eta^{-1}(\lambda)\in \Omega_\theta, \\ 0, & \text{otherwise}. \end{cases}
\end{gather*}
\end{lem}

\begin{proof} The matrix $\alpha_\eta$ maps $\Omega_\theta$ to itself by the action
by fractional linear transformations, $\lambda \mapsto \alpha_\eta(\lambda)=d\lambda + m-1$.
The inverse $\alpha_\eta^{-1}$ is given by the matrix
\begin{gather*} \alpha_\eta^{-1}=\begin{pmatrix} d^{-1} & d^{-1}(1-m) \\ 0 & 1 \end{pmatrix} . \end{gather*}
The condition that for a given $\lambda=a\theta+b \in \Omega_\theta$, we have
$\alpha_\eta^{-1}(\lambda)=d^{-1} a \theta + d^{-1} (b + 1-m)$ also in~$\Omega_\theta$
is satisfied if $d|a$, $b+1-m\geq 0$, and $d|(b+1-m)$. Thus, it is on this subset
$\Omega_\theta(\eta)\subset \Omega_\theta$ that the operation $\alpha_\eta^{-1}(\theta)$
is defined with values in $\Omega_\theta$. The projector $\tilde e_\eta$ is given by the
characteristic function of this subset. The adjoint $\mu_\eta^*$ of the operator $\mu_\eta$ of
\eqref{muetaL} is given by $\mu_\eta^* \epsilon_{\eta',\lambda}= \epsilon_{\eta'\circ \phi_\eta,\alpha_\eta(\lambda)}$
if $\eta'\in \cE_\eta$ and zero otherwise. Thus, $\mu_\eta$ and $\mu^*_\eta$ satisfy the relations as stated.
\end{proof}

Now, given a Belyi function $f\colon \Sigma \to \P^1$ and a character $\varphi\in \Hom\big(\cH_\cD,\bar\Q\big)$, consider the operators $\pi_\varphi(f)$ defined as before, acting as
$\pi_\varphi(f) \epsilon_{\eta,\lambda}=\varphi(\eta(D)) \epsilon_{\eta,\lambda}$.

\begin{lem}\label{relpisomlem}Let $\sigma_\eta(f)=f\circ \eta$ as before, for all Belyi maps $f \in \cB$ and all $\eta\in \cE$. Also let $\rho_\eta(f)=f\circ \phi_\eta$ for $f \in \cB_\eta$ be the partial inverses. With the
partial isometries $\mu_\eta$ and $\mu_\eta^*$ as in Lemma~{\rm \ref{pisoms}} we have
\begin{gather}\label{relpisom}
 \mu_\eta^* \pi_\varphi(f) \mu_\eta = \pi_\varphi(\sigma_\eta(f))e_\eta \tilde e_\eta, \qquad
\mu_\eta \pi_\varphi(f) \mu_\eta^* = \pi_\varphi(\rho_\eta(f)) e_\eta \tilde e_\eta.
\end{gather}
\end{lem}

\begin{proof}
We have $\mu_\eta^* \pi_\varphi(f) \mu_\eta \epsilon_{\eta',\lambda}= \varphi(\eta' \eta (D)) e_\eta \tilde e_\eta \epsilon_{\eta',\lambda} = \pi_\varphi(\sigma_\eta(f))e_\eta \tilde e_\eta \epsilon_{\eta',\lambda}$. The other
case is similar.
\end{proof}

\begin{defn}\label{extalgAHE}Let the semigroup be chosen as $\cE=\cE_{\bar \Q,0}$.
For a fixed $\theta\in \R^*_+\smallsetminus\Q^*_+$ and a character $\varphi\in \Hom\big(\cH_\cD,\bar\Q\big)$
as above, for which the operators $\pi_\varphi(f)$ are bounded, the extended quantum statistical
mechanical system of dessins is given by the $C^*$-subalgebra
$\cA_{\cH_\cD,\cE,\theta}\subset \cB\big(\ell^2(\cE\times \Omega_\theta)\big)$ generated by the
$\pi_\varphi(f)$, with $f\colon \Sigma \to \P^1$ ranging over the Belyi functions and the partial isometries
$\mu_\eta, \mu_\eta^*$, with the relations $\mu_\eta \mu_\eta^* =e_\eta$ and $\mu_\eta^* \mu_\eta =\tilde e_\eta$ and~\eqref{relpisom}. The time evolution on $\cA_{\cH_\cD,\cE,\theta}$ is again given by
$\sigma_t(\mu_\eta)=\deg(\eta)^{{\rm i}t} \mu_\eta$.
\end{defn}

Modifying the isometries $\mu_\eta$ on $\ell^2(\cE)$ to partial
isometries on $\ell^2(\cE\times \Omega_\theta)$ by introducing the
projection $\tilde e_\eta$ will make it possible to extend the Hamiltonian
determined by the degree map to an operator~\eqref{Htheta} for
which ${\rm e}^{-\beta H}$ is trace class for large $\beta>0$. However,
this changes significantly some of the properties of the algebra of
observable as follows.

\begin{rem}\label{ground0}Unlike what happens in the original Bost--Connes system, in our case the partial
isometries $\mu_\eta$ and $\mu_\eta^*$ are {\em not} physical creation-annihilation
operators. Indeed, the ground state $\epsilon_{{\rm id},\theta}$ is in the kernel of both the range projections $e_\eta=\mu_\eta \mu_\eta^*$ and the source projections $\tilde e_\eta=\mu_\eta^* \mu_\eta$, since the identity map
does not factor through another Belyi-extending map, hence $e_\eta \epsilon_{{\rm id},\theta}=0$,
and $\theta$ does not satisfy $\alpha_\eta^{-1}(\theta)\in \Omega_\theta$, since for $a>1$ and $b\geq 0$
we have $\alpha_\eta^{-1}(\theta)=a^{-1}(\theta-b)<\theta$, hence $\tilde e_\eta \epsilon_{{\rm id},\theta}=0$ as well.
\end{rem}

\begin{defn}\label{Hextdef} Consider a family of densely defined unbounded linear operators $H$ on
$\ell^2(\cE\times \Omega_\theta)$ given by
\begin{gather}\label{Htheta}
H \epsilon_{\eta,\lambda} = \begin{cases}
 (F(\alpha_\eta(\lambda)-F(\theta)) \log(\deg(\eta)) \epsilon_{\eta,\lambda} & \text{if}\ \eta\neq {\rm id}, \\
 (F(\lambda)-F(\theta)) \epsilon_{\eta,\lambda} & \text{if} \ \eta={\rm id},
 \end{cases}
\end{gather}
where $F$ is a real valued function on the set $\Omega_\theta$.
\end{defn}

The specific form of the function $F$ and the conditions on the choice of the parameter $\theta$ will
be determined in Propositions~\ref{multH} and~\ref{propZbetatheta} below.

\begin{prop}\label{multH}Let the semigroup be chosen as $\cE=\cE_{\bar \Q,0}$.
We choose $\theta \in \R^*_+\smallsetminus \Q^*_+$ to be an algebraic number such that
$\big\{ 1, \theta, \theta^2 \big\}$ are linearly independent over~$\Q$. Let $F(\lambda)=\lambda^2$.
Then, for $d\neq 1$, the multiplicity $M(d,m,\lambda)$ of an eigenvalue
$(F(d\lambda +m-1)-F(\theta)) \log(d)$ of the operator~$H$ is equal to
\begin{gather}\label{Mdml}
M(d,m,\lambda) = 2 + \# T_{d,m},
\end{gather}
where $\# T_{d,m}$ is the number of bipartite trees with $d$ edges and $m$ white vertices, given by~\eqref{dmtrees}. For $d=1$, each eigenvalue $F(\lambda)-F(\theta)$ of $H$ has multiplicity one.
\end{prop}

\begin{proof} Suppose that $\epsilon_{\eta,\lambda}$ and $\epsilon_{\eta',\lambda'}$, with $\eta,\eta'\neq {\rm id}$, are in the same eigenspace of~$H$. We have, for $d=\deg(\eta)$, $d'=\deg(\eta')$,
$m=\#\eta^{-1}(0)$, and $m'=\#(\eta')^{-1}(0)$,
\begin{gather*} (F(d\lambda +m-1)-F(\theta)) \log(d) = (F(d' \lambda' + m'-1)-F(\theta)) \log(d'), \end{gather*}
where $\lambda=a\theta+b, \lambda'=a'\theta+b'\in \Omega_\theta$.
By our choice of $F$, it takes values $F(\lambda)\in \bar\Q$ for
all $\lambda\in \Omega_\theta$.

Recall that if we have algebraic numbers $\alpha_1$, $\alpha_2$, $\beta_1$, $\beta_2$
such that $\log(\alpha_1)$ and $\log(\alpha_2)$ are linearly independent over $\Q$, then
$\beta_1 \log(\alpha_1) + \beta_2 \log(\alpha_2)\neq 0$.
 This shows that, for $d$, $d'$ with~$\log(d)$, $\log(d')$ linearly independent over $\Q$ we have
$F(d\lambda +m-1) \log(d) - F(d' \lambda' + m'-1) \log(d')\neq 0$. So we
only need to check the dependent case.

Two logarithms of integers $\log(d)$, $\log(d')$
are linearly dependent over $\Q$, if $d^\alpha =(d')^\beta$ for some
$\alpha,\beta\in \Q^*_+$ (hence $d$ and $d'$ have the same prime factors).
Thus we can write $d=\delta^k$, $d'=\delta^\ell$ for $\delta \in \R^*_+$ and $k,\ell \in \Q^*_+$.
We are then looking at the relation
\begin{gather}\label{kellFrel}
 k (F(d\lambda +m-1)-F(\theta)) = \ell (F(d' \lambda' + m'-1)-F(\theta)).
\end{gather}
By our choice of $\theta$ and of the function $F$, a relation of the form
$k (F(u\theta+v)-F(\theta))= \ell (F(u' \theta + v')-F(\theta))$ for some $k,\ell\in \Q^*_+$, some
$u,u'\in \N$, and some $v,v'\in \Z_+$ gives
\begin{gather*}
k \big((u\theta+v)^2 -\theta^2\big) = \ell \big((u'\theta+v')^2 -\theta^2\big) ,
\end{gather*}
that is,
\begin{gather*}
k \big(u^2\theta^2 + 2 uv \theta + v^2 -\theta^2\big) = \ell \big((u')^2 \theta^2 + 2u'v'\theta + (v')^2 -\theta^2\big).
\end{gather*}

The independence of $\{ 1, \theta, \theta^2 \}$ over $\Q$ implies relations
\begin{gather*} k \big(u^2-1\big) = \ell \big((u')^2-1\big), \qquad k uv = \ell u' v', \qquad k v^2 =\ell (v')^2. \end{gather*}
The last one gives $v'=\sqrt{k/\ell} v$, which substituted in the second one gives
$u'=\sqrt{(k/\ell)} u$.

The first one then gives $k \big(u^2-1\big)=\ell \big((k/\ell) u^2-1\big)$,
hence $k=\ell$, so that $d=d'$, $u=u'$ and $v=v'$.
Thus, we obtain that the relation~\eqref{kellFrel} is satisfied for $k=\ell$ (hence $d=d'$) and
$da\theta +db+ m-1 =d a'\theta + db' + m'-1$ which gives $a=a'$ and $db+m-1=db'+m'-1$.
The latter equality gives $b-b' = (m'-m)/d$.

 We have both $1\leq m \leq d$ and $1\leq m' \leq d$ so that
$0\leq |m'-m| \leq d-1$ and $(m'-m)/(d-1) \in \Z$ for either $m=m'$ (so $b=b'$) or for either $m=1$ and $m'=d$ or $m=d$ and $m'=1$. In the first of these two cases $b'=b-1$ and in the other $b'=b+1$.
In each of the cases where $m=1$ and $m'=d$ or $m=d$ and $m'=1$ the number of bipartite
trees $T$ with fixed $d=\# E(T)$ and with either one or $d=\# V(T)-1$ white vertices is just one,
consisting of the single vertex of the different color and $d$ edges from it to the remaining
vertices all of the other color. Thus, in both of these cases the multiplicity is equal to one,
while in the remaining case with $m=m'$ and $b=b'$ the multiplicity is the number of bipartite
trees with $d$ edges and $m$ white vertices.

Thus, we obtain that the overall multiplicity of the
eigenvalue $F(d\lambda + m-1)\log(d)$ of $H$ is equal to $M_{d,m,\lambda}=2+\# T_{d,m}$
as in \eqref{Mdml}. For $d=1$, we have $\eta={\rm id}$ and the condition $F(\lambda)=F(\lambda')$
implies $\lambda=\lambda'$, so all these eigenspaces are one-dimensional.
In particular, the kernel of the operator $H$ is one-dimensional, spanned by the
vector $\epsilon_{{\rm id},\theta}$.
\end{proof}

\begin{prop}\label{covreptheta}Consider the operator $H$ of \eqref{Htheta}.
The operators ${\rm e}^{{\rm i}tH}$ for $t\in \R$ determine a covariant representation of the quantum statistical mechanical system $(\cA_{\cH_\cD,\cE,\theta}, \sigma_t)$ of Definition~{\rm \ref{extalgAHE}} on the Hilbert space $\ell^2(\cE\times \Omega_\theta)$.
\end{prop}

\begin{proof} The covariance condition prescribes that the time evolution $\sigma_t$ on $\cA_{\cH_\cD,\cE,\theta}$
is implemented by the Hamiltonian $H$ in the representation on $\ell^2(\cE\times \Omega_\theta)$,
\begin{gather*}
 {\rm e}^{{\rm i}tH} X {\rm e}^{-{\rm i}tH} \epsilon_{\eta,\lambda} = \sigma_t(X) \epsilon_{\eta,\lambda},
\end{gather*}
for all basis elements $\epsilon_{\eta,\lambda}$. We check this on the generators $\mu_\eta, \mu_\eta^*$
and $\pi_\varphi(f)$ of the algebra. For $X=\mu_\eta$ we have
\begin{gather*} \sigma_t(\mu_\eta) \epsilon_{\eta',\lambda} = \deg(\eta)^{{\rm i}t} \epsilon_{\eta'\circ \eta, \eta^{-1}(\lambda)} \end{gather*}
or zero if $\eta^{-1}(\lambda)\notin \Omega_\theta$. On the other hand, for $\eta'\neq {\rm id}$, we have
\begin{align*}
{\rm e}^{{\rm i}tH} \mu_\eta {\rm e}^{-{\rm i}tH} \epsilon_{\eta',\lambda} & = {\rm e}^{{\rm i}tH} \mu_\eta {\rm e}^{-{\rm i}t F(\alpha_{\eta'}(\lambda))}
\deg(\eta')^{-{\rm i}t} \epsilon_{\eta',\lambda} \\
& = {\rm e}^{{\rm i}tH} {\rm e}^{-{\rm i}t F(\alpha_{\eta'}(\lambda))}
\deg(\eta')^{-{\rm i}t} \epsilon_{\eta' \circ \eta, \eta^{-1}(\lambda)}
\end{align*}
or zero if $\eta^{-1}(\lambda)\notin \Omega_\theta$. In the non-zero case it is then equal to
\begin{gather*}
{\rm e}^{{\rm i}t F(\alpha_{\eta'\circ \eta} (\alpha_{\eta}^{-1}(\lambda))} \deg(\eta'\circ \eta)^{{\rm i}t}
 {\rm e}^{-{\rm i}t F(\alpha_{\eta'}(\lambda))} \deg(\eta)^{-{\rm i}t} \epsilon_{\eta' \circ \eta, \eta^{-1}(\lambda)}.
\end{gather*}
Because $\Upsilon(\eta)=\alpha_\eta$ is a semigroup homomorphism and so is the degree, we have
$F(\alpha_{\eta'\circ \eta} (\alpha_{\eta}^{-1}(\lambda)))\allowbreak =F(\alpha_{\eta'}(\lambda))$ and
the above gives
\begin{gather*} {\rm e}^{{\rm i}tH} \mu_\eta {\rm e}^{-{\rm i}tH} \epsilon_{\eta',\lambda} = \deg(\eta)^{{\rm i}t} \mu_\eta \epsilon_{\eta',\lambda}. \end{gather*}
On the vectors $\epsilon_{1,\lambda}$ we have
\begin{align*}
{\rm e}^{{\rm i}tH} \mu_\eta {\rm e}^{-{\rm i}tH} \epsilon_{1,\lambda} & = {\rm e}^{-{\rm i}t (F(\lambda)-F(\theta))} {\rm e}^{{\rm i}tH} \epsilon_{\eta,\eta^{-1}(\lambda)}
\\
& = {\rm e}^{-{\rm i}t (F(\lambda)-F(\theta))} \deg(\eta)^{{\rm i}t} {\rm e}^{{\rm i}t (F(\alpha_{\eta\circ\eta^{-1}}(\lambda)-F(\theta)}
\epsilon_{\eta,\eta^{-1}(\lambda)} = \deg(\eta)^{{\rm i}t} \mu_\eta \epsilon_{1,\lambda} ,
\end{align*}
for $\eta^{-1}(\lambda)\in \Omega_\theta$, and zero otherwise. So ${\rm e}^{{\rm i}tH}$ implements the
time evolution on the partial isometries $\mu_\eta$. In the case of their adjoints we similarly have for
${\rm e}^{{\rm i}tH} \mu_\eta^* {\rm e}^{-{\rm i}tH} \epsilon_{\eta',\lambda}$ the expression
\begin{gather*} {\rm e}^{-{\rm i}t (F(\alpha_{\eta'}(\lambda))-F(\theta))} \deg(\eta' \circ \phi_\eta)^{{\rm i}t} \deg(\eta')^{-{\rm i}t}
{\rm e}^{{\rm i}t (F(\alpha_{\eta'\circ \phi_\eta \circ \eta}(\lambda)-F(\theta))} \epsilon_{\eta'\circ \phi_\eta,\eta(\lambda)}, \end{gather*}
when $\eta'\in \cE_\eta$ with $\eta'\circ \phi_\eta=h'\in \cE$, and zero otherwise. Here we have
\begin{gather*}
\deg(\eta' \circ \phi_\eta)=\deg(h')=\frac{\deg(\eta')}{\deg(\eta)},
\end{gather*}
so we obtain
$\deg(\eta)^{-{\rm i}t} \epsilon_{\eta'\circ \phi_\eta,\eta(\lambda)} = \sigma_t(\mu_\eta^*)\epsilon_{\eta',\lambda}$.
When $\eta'={\rm id}$ we have $\eta'\notin \cE_\eta$ for $\eta\neq {\rm id}$, hence
$\mu_\eta^* \epsilon_{1,\lambda}=0$. In the case where $X=\pi_\varphi(f)$ we have
\begin{gather*} {\rm e}^{{\rm i}tH} \pi_\varphi(f) {\rm e}^{-{\rm i}tH} \epsilon_{\eta',\lambda} = \varphi(\eta'(D)) \epsilon_{\eta',\lambda} =
 \pi_\varphi(f) \epsilon_{\eta',\lambda} . \end{gather*}
So indeed the operator $H$ gives a covariant representation of $(\cA_{\cH_\cD,\cE,\theta}, \sigma_t)$ on
the Hilbert space $\ell^2(\cE\times \Omega_\theta)$.
\end{proof}

\begin{prop}\label{propZbetatheta}
Let $H$ be the operator of \eqref{Htheta}, with $F$ and $\theta$ chosen as in Proposition~{\rm \ref{multH}}.
We also assume that $\theta>1$. Then the partition function for the semigroup $\cE=\cE_{\bar \Q,0}$
\begin{gather}\label{Zbetatheta}
Z(\beta) =\Tr \big({\rm e}^{-\beta H}\big) = \sum_{\eta\in \cE, \lambda\in \Omega_\theta} M_{\deg(\eta), m(\eta), \lambda} {\rm e}^{-\beta F(\alpha_\eta(\lambda))} \deg(\eta)^{-\beta}
\end{gather}
converges for $\beta>1$.
\end{prop}

\begin{proof} By Proposition~\ref{multH}, the partition function takes the form
\begin{align*} Z(\beta) & = \sum_{d,a\geq 1,b\geq 0} \sum_{1\leq m\leq d} M(d,m,\lambda)
{\rm e}^{-\beta (F(d(a\theta+b)+ m-1)-F(\theta))} d^{-\beta} \\
& = \sum_{a\geq 1,b\geq 0} {\rm e}^{-\beta (F(a\theta+b)-F(\theta))} +
\sum_{d >1, a\geq 1,b\geq 0} \sum_{1\leq m\leq d} (2+ \# T_{d,m})
{\rm e}^{-\beta F(d(a\theta+b)+ m-1)} d^{-\beta}. \end{align*}
The first series, which corresponds to the case $d=1$ (hence $m=1$) is, up to a multiplicative
factor ${\rm e}^{\beta\theta^2}$ the series
$\sum\limits_{a\geq 1,b\geq 0} {\rm e}^{-\beta (a\theta+b)^2}$.
Since we chose $\theta>1$, we have $(a\theta+b)^2> a\theta+b$ for all $a\geq 1$ and $b\geq 0$ and
we can estimate first sum by the series $\sum\limits_{a,b} {\rm e}^{-\beta (a\theta+b)}=\sum\limits_{a\geq 1} {\rm e}^{-\beta\theta a} \sum\limits_{b\geq 0} {\rm e}^{-\beta b}$ which is convergent for all $\beta>0$.

Thus, it remains to check that the series
\begin{gather*} \sum_{d,a,b} \sum_{1\leq m\leq d} m^{d+1-m} (d+1-m)^{m-1} {\rm e}^{-\beta (d(a\theta+b)+ m-1)^2} d^{-\beta} \end{gather*}
converges for sufficiently large $\beta>0$. We can estimate the terms of this series as follows.
For $m=1$ we have $\# T_{d,1}=1$ and the sum reduces to the form
\begin{gather*} \sum_{d,a,b} {\rm e}^{-\beta (d(a\theta+b))^2} d^{-\beta} \leq \sum_{a,b} {\rm e}^{-\beta (a\theta+b)^2} \sum_d d^{-\beta}, \end{gather*}
which is convergent for $\beta>1$. Equivalently, we can estimate this sum in terms of the Jacobi theta constant and the series $\sum_d \theta_3\big(0,{\rm e}^{-\beta d^2}\big) d^{-\beta}$. Since
$\theta_3(0,q)\to 1$ for $q\to 0$, the behavior of the series then depends on
the behavior of $\sum_d d^{-\beta}$ that is convergent for $\beta>1$.
More generally the multiplicity $\# T_{d,m}=m^{d+1-m} (d+1-m)^{m-1}$ satisfies $\# T_{d,m} \leq d^{2d}$ since
$1\leq m \leq d$. Since $\theta>1$ we have $\Omega_\theta=\N\theta+\Z_+
\subset (1,\infty)$. For $m>1$ we estimate the sum by
\begin{align*} \sum_{d,a,b} \sum_{1< m\leq d} d^{2d} {\rm e}^{-\beta (d(a\theta+b)+ m-1)^2} d^{-\beta} &\leq
\sum_d d^{2d} {\rm e}^{-\beta d^2} d^{-\beta} \sum_{a,b} \sum_{m>1} {\rm e}^{-\beta (m-1)^2} {\rm e}^{-\beta (a\theta+b) (m-1)} \\
& \leq \sum_d d^{2d} {\rm e}^{-\beta d^2} d^{-\beta} \sum_{a,b} {\rm e}^{-\beta (a\theta+b)} \sum_{\ell \geq 1} {\rm e}^{-\beta \ell^2}. \end{align*}
The series $\sum\limits_{a\geq 1, b\geq 0} {\rm e}^{-\beta (a\theta+b)}$ and $\sum\limits_{a\geq 1, b\geq 0} {\rm e}^{-\beta (a\theta+b)^2}$ are convergent for all $\beta>0$, and so is the series $\sum\limits_{\ell\geq 1} {\rm e}^{-\beta \ell^2}$. The series $\sum\limits_{d\geq 1} d^{2d} {\rm e}^{-\beta d^2} d^{-\beta} = \sum\limits_{d\geq 1} {\rm e}^{2d \log(d) -\beta (d^2 +\log(d))}$ is also convergent.
\end{proof}

\subsection{Gibbs KMS states and zero-temperature states}\label{ZeroStateSec}

In this subsection we show what the KMS states look like for a modified
system of the kind introduced in the previous subsection. We find that again
the invariants~$\varphi(\eta(D))$ appear in the low temperature KMS states,
in the form of a weighted sum of the $\varphi(\eta(D))$, with $\eta$ ranging
over the chosen semigroup of Belyi maps.
Unlike the case of Corollary~\ref{KMSconv}, however, in this case the invariants
$\varphi(\eta(D))$ do not occur individually as values of zero-temperature states
on elements $\mu_\eta^* \pi_\varphi(f) \mu_\eta$ of the algebra.
As in the previous subsection, we only discuss here explicitly the case of the
semigroup $\cE=\cE_{\bar\Q,0}$ so that we can use the explicit form of the
quantum statistical mechanical system constructed above.

Consider a character $\varphi\in \Hom_{{\rm Alg}_\Q}\big(\cH_\cD,\bar\Q\big)$ that satisfies the
boundedness condition for the operators $\pi_\varphi(f)$, and let $H$ be as in
\eqref{Htheta}, with $F$ and $\theta$ as in Proposition~\ref{propZbetatheta},
for the semigroup $\cE=\cE_{\bar\Q,0}$.
In the range $\beta>1$ where the series~\eqref{Zbetatheta} is convergent,
the low temperature Gibbs KMS states of the quantum statistical
mechanical system are given by
\begin{gather*}
 \psi_{\beta,\varphi}(X) =Z(\beta)^{-1} \sum_{\eta,\lambda} \big\langle \epsilon_{\eta,\lambda},
X {\rm e}^{-\beta H} \epsilon_{\eta,\lambda} \big\rangle,
\end{gather*}
for all $X\in \cA_{\cH_\cD, \cE,\theta}$. We are interested here in the values of
these Gibbs states on the arithmetic abelian subalgebra $\cH_\cD$ of $\cA_{\cH_\cD, \cE,\theta}$.
These are given by
\begin{gather}\label{Gibbstheta}
 \psi_{\beta,\varphi}(D) =Z(\beta)^{-1} \Bigg(
 \sum_{\eta\in \cE\smallsetminus \{ {\rm id} \}, \lambda\in \Omega_\theta} \varphi(\eta(D))
 {\rm e}^{-\beta (F(\alpha_\eta(\lambda))-F(\theta))} \deg(\eta)^{-\beta} \\
\hphantom{\psi_{\beta,\varphi}(D) =Z(\beta)^{-1} \Bigg(}{} + \varphi(D) \sum_{\lambda\in \Omega_\theta}
 {\rm e}^{-\beta (F(\lambda)-F(\theta))} \Bigg) .\tag*{\qed}
\end{gather}\renewcommand{\qed}{}
\end{proof}

\begin{lem}\label{zerolim} Let $H$ be as in \eqref{Htheta}, with $F$ and $\theta$ as in Proposition~{\rm \ref{propZbetatheta}}, with the semigroup $\cE=\cE_{\bar\Q,0}$.
In the zero-temperature
limit where $\beta\to \infty$, the ground states, evaluated on the rational subalgebra $\cH_\cD$
are given by the limits
\begin{gather*}
 \psi_{\infty,\varphi} (D)=\lim_{\beta\to \infty}\psi_{\beta,\varphi}(D) = \varphi(D).
\end{gather*}
\end{lem}

\begin{proof} The ground state of the Hamiltonian $H$ of \eqref{Htheta}, with the choice of $F$ and
$\theta$ as in Proposition~\ref{propZbetatheta},
corresponds to $d=1$ and $m=1$ and to $\lambda=\theta$, and is spanned by
the vector~$\epsilon_{{\rm id},\theta}$. In the limit where $\beta\to \infty$ the
expression \eqref{Gibbstheta}, which is the normalized trace
\begin{gather*} \psi_{\beta,\varphi}(D) =\frac{\Tr\big(\pi_\varphi(f) {\rm e}^{-\beta H}\big)}{\Tr\big({\rm e}^{-\beta H}\big)}= Z(\beta)^{-1}
\sum_{\eta,\lambda} \big\langle \epsilon_{\eta,\lambda}, \pi_\varphi(f) {\rm e}^{-\beta H} \epsilon_{\eta,\lambda} \big\rangle \end{gather*}
converges to $\langle \eta_{{\rm id},\theta}, \pi_\varphi(f) \eta_{{\rm id},\theta}\rangle=\varphi(D)$.
\end{proof}

In the case of this quantum statistical mechanical system all the values $\varphi(\eta(D))$ for $\eta \in \cE_{\bar\Q,0}$ (or another semigroup for which a similar system can be constructed) are built into the Gibbs states evaluated on the elements of the rational subalgebra, as the expression~\eqref{Gibbstheta} shows. However, one cannot extract an individual term $\varphi(\eta(D))$ from the Gibbs states by taking the zero-temperature limit, because of the observation in Remark~\ref{ground0}. Indeed, the zero-temperature states evaluate trivially on elements of the form $\mu_\eta^* \pi_\varphi(f) \mu_\eta$ since we have
$\psi_{\infty,\varphi}(\mu_\eta^* \pi_\varphi(f) \mu_\eta)=\langle \epsilon_{{\rm id},\theta},
\mu_\eta^* \pi_\varphi(f) \mu_\eta \epsilon_{{\rm id},\theta} \rangle$, but
$\mu_\eta \epsilon_{{\rm id},\theta}=0$ since $\eta^{-1}(\theta)\notin\Omega_\theta$.

On the other hand, we can still obtain the intertwining of symmetries and Galois
action for zero-temperature KMS states evaluated on the arithmetic subalgebra.

\begin{prop}\label{interKMS}Suppose that $\cE$ is a semigroup of Belyi-extending maps for which the construction of the extended quantum statistical mechanical system $(\cA_{\cH_\cD, \cE},\sigma_t)$
can be applied. Let $\varphi\in \Hom_{{\rm Alg}_\Q}\big(\cH_\cD,\bar\Q\big)$ be a balanced character as in
Definition~{\rm \ref{balchar}}. Then the KMS Gibbs state~$\psi_{\infty,\varphi}$
at zero temperature evaluated on the rational subalgebra $\cH_\cD$
intertwines the action of $G={\rm Gal}\big(\bar\Q/\Q\big)$ by symmetries
of the quantum statistical mechanical system $(\cA_{\cH_\cD, \cE},\sigma_t)$
with the Galois action of~$G$ on~$\bar\Q$.
\end{prop}

\begin{proof} First note that the $G$-action on dessins gives an action of $G$ by symmetries
of the quantum statistical mechanical system $(\cA_{\cH_\cD, \cE,\theta},\sigma_t)$, namely by
automorphisms of the algebra~$\cA_{\cH_\cD, \cE,\theta}$ compatible with the time evolution:
$\gamma \circ \sigma_t = \sigma_t \circ \gamma$, for all $t\in \R$ and all $\gamma\in G$.
Here it is convenient to assume that~$\cE$ is a semigroup of Belyi-extending maps rather than a
more general subsemigroup of $\cE_{\bar\Q}$, so that the Galois group acts only on
$\cH_\cD$ and fixes the partial isometries $\mu_\eta$.

Indeed, since in the case of Belyi-extending maps $G$ acts on the abelian subalgebra $\cH_\cD$ by the action of Proposition~\ref{GHopfact} and acts trivially on the partial isometries $\mu_\eta$, $\mu_\eta^*$
while the time evolution acts on the $\mu_\eta$, $\mu_\eta^*$ and acts trivially on $\cH_\cD$, the
two actions commute. (Note, however, that if the semigroup homomorphism $\Upsilon$
generating the time evolution is itself Galois invariant, then the same argument applies
to more general subsemigroups of $\cE_{\bar\Q}$.)

Evaluating the zero-temperature KMS state $\psi_{\infty,\varphi}$
on an element $D$ of $\cH_\cD$ gives $\psi_{\infty,\varphi}(D)=\varphi(D)$. Since $\varphi$
is a balanced character, it also satisfies $\varphi(\gamma D)=\gamma \varphi(D)$ which
gives the intertwining of the $G$-actions.
\end{proof}

Below we consider some variants of quantum statistical mechanical
systems associated to dessins d'enfant.

Generally, in number theory many natural Dirichlet series appear as Euler products
whose $p$-terms encode the results of counting problems (e.g., counting points of a $\Z$-scheme $\rm{mod} p$, and subsequent twisting them by additive or multiplicative characters.

In the examples below, the attentive reader will find analogs of prime decomposition and
twisting by characters, but the central role is taken by counting/enumeration problems themselves.

As a result, we obtain again some formal partition functions/Dirichlet series
such as~\eqref{ZetaUpsilon} and~\eqref{Zetamunu}.
Typically they suffer from the same divergence problem that we
have already discussed in the main part of this section and would require a similar modification of
the underlying algebra of observables and representation. We will not discuss this further
in this paper.

\subsection{Enumeration problems for dessins d'enfant}

The enumeration problem for Grothendieck dessins d'enfant was considered in~\cite{KaZo,Zog}. In~\cite{Zog} it is shown that the generating function for the number of dessins with
assigned ramification profile at $\infty$ and given number of preimages of~$0$ and~$1$ satisfy the infinite system of PDEs given by the Kadomtsev--Petviashvili (KP)
hierarchy. In~\cite{KaZo} it is shown that this generating function satisfies the Eynard--Orantin
topological recursion. Moreover in~\cite{AmbChek} it is shown that the generating function
of the enumeration of dessins d'enfant with fixed genus, degree, and ramification profile
at $\infty$ is the partition function of a matrix model, which in the case of clean dessins
agrees with the Kontsevich--Penner model of~\cite{ChekMak}. We see here that it is also,
in an immediate and direct way, the partition function of a quantum statistical mechanical
system with algebra of observables given by the (Hopf) algebra~$\cH$ constructed above.

Another interesting aspect of the enumeration problem for dessins d'enfant is addressed in~\cite{HaKraLe}, namely a piecewise polynomiality result and a wall crossing phenomenon. More precisely, the enumeration of (not necessarily connected) dessins d'enfant corresponds to the case of ``double strictly monotone Hurwitz numbers'' considered in~\cite{HaKraLe}. In that case, the counting is given by
\begin{gather}\label{hgmunu}
 h_{g;\mu,\nu}=\sum_{\phi} \frac{1}{\# {\rm Aut}(\phi)}
\end{gather}
of all isomorphism classes of branched coverings $\phi\colon \Sigma\to \P^1$ of genus $g$, branched over
$\{0,1,\infty\}$ with assigned ramification profiles $\mu=(\mu_1,\ldots,\mu_m)$ and
$\nu=(\nu_1,\ldots,\nu_n)$ over $0$ and $\infty$, with degree $\sum_i\mu_i= d =\sum_j\nu_j$.
Let $\cH(m,n)=\Big\{ (v,w)\in \N^m\times \N^n\,|\, \sum\limits_{i=1}^m v_i =\sum\limits_{j=1}^n w_j \Big\}$ and let
$\cW(m,n)\subset \cH(m,n)$ be the hyperplane arrangement given by the equations
$\sum\limits_{i\in I} v_i =\sum\limits_{j\in J}w_j$ with $I\subset \{1,\ldots,m\}$, $J\subset \{1,\ldots, n\}$.
The branching profiles $\mu,\nu$, subject to the constraint $\sum_i \mu_i =\sum_j\nu_j$, define
a point in $\cH(m,n)$. It is shown in \cite[Theorem~4.1]{HaKraLe} that in each chamber~$C$ of the
complement of $\cW(m,n)$ there is polynomial $P_{g,C}(\mu,\nu)$ of degree $4g-3+m+n$
in $m+n$ variables $(\mu,\nu)$ such that $h_{g;\mu,\nu}=P_{g,C}(\mu,\nu)$. The behavior in
different chambers is regulated by a wall crossing formula relating the corresponding polynomials.

\subsection{Additive invariants and partition function}

Consider the Hilbert space $\ell^2(\cD)$ generated by the set $\cD$ of all
(not necessarily connected) dessins d'enfant, with the standard orthonormal
basis $\{ \epsilon_D \}_{D\in \cD}$ and with the action of the algebra $\cH_\cD$
by $D\cdot \epsilon_{D'}=\epsilon_{D \cdot D'}$. Since $\cH_\cD$ is also a Hopf algebra,
we have also an adjoint action of the Hopf algebra on itself, of the form $D\colon D'\mapsto \sum \delta' D' S(\delta'')$
where in Sweedler notation $\Delta(D)=\sum \delta'\otimes \delta''$, with $S$ the antipode,
given by the recursive formula $S(\delta'')=-\delta''+\sum S(\delta_1'')\cdot \delta_2''$ for
$\Delta(\delta'')=\sum \delta_1'' \otimes \delta_2''$. Given an element $X=\sum_i a_i D_i$ with
$a_i\in \C$ and $D_i\in \cD$, we write $\epsilon_X$ for the vector $\epsilon_X=\sum_i a_i \epsilon_{D_i}$
in $\ell^2(\cD)$. We can then set $D\cdot_\Delta \epsilon_{D'}:=\epsilon_{\sum \delta' D' S(\delta'')}$.

Let $\cN$ be an additive invariant of dessins d'enfant, that is, an invariant of isomorphism
classes of dessins with the property that it is additive on connected components,
$\cN(D\cdot D')=\cN(D)+\cN(D')$. This is the case for invariants such as the genus, the
degree, the ramification indices.

\begin{lem}\label{addtevol}
An $\R_+$-valued additive invariant $\cN$ of dessins d'enfant determines a time evolution
of the algebra $\cH_{\cD,\C}:=\cH_\cD\otimes_\Q\C$ of the form $\sigma_t(D)={\rm e}^{{\rm i}t \cN(D)} D$,
implemented on $\ell^2(\cD)$ by the Hamiltonian
$H \epsilon_D = \cN(D) \epsilon_D$. The partition function is given by
the generating function of dessins with assigned invariant $\cN(D)$.
\end{lem}

\begin{proof} We have ${\rm e}^{{\rm i}t H} D {\rm e}^{-{\rm i}tH} \epsilon_{D'} = {\rm e}^{-{\rm i}t \cN(D')} {\rm e}^{{\rm i}t \cN(D\cdot D')} \epsilon_{D\cdot D'}
= \sigma_t(D) \epsilon_{D'}$, since $\cN(D\cdot D')=\cN(D)+\cN(D')$. The partition function is given by
\begin{gather}\label{ZetaND}
 Z(\beta)=\Tr\big({\rm e}^{-\beta H}\big) = \sum_{\lambda\in \cN(\cD)} \#\{ D\in \cD\,|\, \cN(D)= \lambda\} {\rm e}^{-\beta \lambda}.
\end{gather}
With the change of variables $t={\rm e}^{-\beta}$, this can be identified with the generating function for
the counting of dessins $D\in \cD$ with assigned value of $\cN(D)$.
\end{proof}

Note that in \eqref{ZetaND}, in order to identify the series with the partition function of the
quantum statistical mechanical system, one assumes that the invariant $\cN$ is such that
the series is convergent for sufficiently large $\beta>0$, while the generating function can
be regarded more generally as a formal power series.

\begin{cor}\label{timevHopf}
If the additive invariant $\cN$ of dessins d'enfant also satisfies $\cN(D)=\cN(\delta)+\cN(D/\delta)$,
for any subdessin $\delta$ and quotient dessin $D/\delta$, then the time evolution $\sigma_t$ of
Lemma~{\rm \ref{addtevol}} is a one-parameter family of Hopf algebra homomorphisms of~$\cH_{D,\C}$.
Moreover, the representation on $\ell^2(\cD)$ induced by the adjoint action of the Hopf algebra on itself
$D\colon D'\mapsto \sum \delta' D' S(\delta'')$ is also covariant with respect to this time evolution.
\end{cor}

\begin{proof} It suffices to check that $\sigma_t(D)={\rm e}^{{\rm i}t\cN(D)} D$ is a bialgebra homomorphism, since the compatibility with the antipode $\sigma_t(S(D))=S(\sigma_t(D))$ is then automatically satisfied.
We have $\Delta(\sigma_t(D))={\rm e}^{{\rm i}t\cN(D)}\Delta(D)={\rm e}^{{\rm i}t\cN(\delta)}
{\rm e}^{{\rm i}t\cN(D/\delta)} \sum \delta\otimes D/\delta$. The compatibility with the antipode, which is
a linear antiautomorphism, gives $\sigma_t(S(D))={\rm e}^{{\rm i}t\cN(D)} S(D)$, hence under the action
$D\cdot_\Delta \epsilon_{D'}:=\epsilon_{\sum \delta' D' S(\delta'')}$ we have
${\rm e}^{{\rm i}t H} D {\rm e}^{-{\rm i}tH} \cdot_\Delta \epsilon_{D'} =\sum {\rm e}^{{\rm i}t \cN(D')} {\rm e}^{{\rm i}t \cN(\delta')} {\rm e}^{{\rm i}t \cN(\delta'')} {\rm e}^{-{\rm i}t \cN(D')}
\epsilon_{\delta' D' S(\delta'')} ={\rm e}^{{\rm i}t\cN(D)} \epsilon_{\sum \delta' D' S(\delta'')} =\sigma_t(D)\cdot_\Delta \epsilon_{D'} $, using again the property that $\cN(D)=\cN(\delta')+\cN(\delta'')$ for all the terms in the coproduct.
\end{proof}

Suppose given a finite set of integer-valued additive invariants $\cN=(\cN_1,\ldots,\cN_k)$,
with $\cN_i(D)\in \Z_+$ for all $D\in \cD$ and all $i=1,\ldots, k$. Let $\lambda=(\lambda_1,\ldots,\lambda_k)$ be a
chosen set of $\lambda_i\in \R^*_+$ that are linearly independent over $\Q$. Consider the $\R^*_+$-valued
additive invariant $\cN_\lambda(D)=\lambda_1 \cN_1(D)+\cdots + \lambda_k \cN_k(D)$, and the time
evolution $\sigma_t$ determined by $\cN_\lambda$ on $\cH_{D,\C}$ as in Lemma~\ref{addtevol}.
If each $\cN_i$ also satisfies $\cN_i(D)=\cN_i(\delta)+\cN_i(D/\delta)$ for all sub-dessins then $\sigma_t$
is also a Hopf algebra homomorphism as in Corollary~\ref{timevHopf}.

\begin{lem}\label{ZbetaNlambda}
For a given set of additive invariants
$\cN=(\cN_1,\ldots,\cN_k)$ and a choice of coefficients $\lambda=(\lambda_1,\ldots,\lambda_k)$
as above, the partition function of the time evolution $\sigma_t$ computes the generating functions of
dessins d'enfant with fixed invariants $\cN_i$ for $i=1,\ldots, k$.
\end{lem}

\begin{proof} We have
\begin{gather*} Z(\beta)=\Tr\big({\rm e}^{-\beta H}\big)=\sum_{\alpha \in \cN_\lambda(\cD)}
\#\{ D \in \cD \,|\, \cN_\lambda(D)=\alpha \} {\rm e}^{-\beta \alpha}, \end{gather*}
where $\alpha\in \cN_\lambda(\cD)$ means that $\alpha=\sum_i \lambda_i n_i$ with $n_i\in \cN_i(\cD)$.
Since the $\lambda_i$ are linearly independent over $\Q$, this determines the $n_i$ and we can write
the sum above as
\begin{gather*} \sum_{n_1,\ldots,n_k} \#\{ D\in \cD\,|\, \cN_i(D)=n_i \} {\rm e}^{-\beta \lambda_1 n_1} \cdots
{\rm e}^{-\beta \lambda_k n_k}. \end{gather*}
Upon setting $t_i={\rm e}^{-\lambda_i}$, we identify this series with the generating function for
dessins with fixed values of the invariants $\cN_i(D)$ for $i=1,\ldots, k$.
\end{proof}

In particular, we see from this simple general fact that we can reinterpret as partition functions
the generating functions of \cite{KaZo} and \cite{Zog} for the number of dessins with assigned
ramification profile at $\infty$ and given number of preimages of $0$ and $1$, as well as the
generating function of \cite{HaKraLe} of dessins of genus $g$
with assigned ramification profiles over $0$ and $\infty$.

While this system recovers the correct partition function that encodes the counting
problem for dessins, since the representation of the algebra $D\cdot \epsilon_{D'}=\epsilon_{D\cdot D'}$
on the Hilbert space is a~simple translation of the basis elements, we do not have an interesting
class of low temperature KMS states, unlike the case we discussed in the previous subsections.

\subsection{Fibered product structure on dessins}\label{FiberSec}

This subsection together with the subsequent Sections~\ref{ArrangeSec}
and~\ref{MultiinvSec} contain a short digression on a construction
of dynamics and partition functions based on different product
structures on dessins.

We describe here a possible variant of the construction presented in the previous
subsection, where instead of considering the (Hopf) algebra $\cH_\cD$ of not
necessarily connected dessins, where the multiplication is given by the disjoint
union, we consider a commutative algebra of dessins based on a different
product structure built using the fibered product of the Belyi maps. The resulting
construction of an associated quantum statistical mechanical system is similar
to the previous case, but the associated partition function will have here the
structure of a Dirichlet series rather than the usual power series generating
function for the counting of dessins.

A product operation on dessins induced by the fibered product of the Belyi maps
was discussed in~\cite{MaZo}.
Let $D_1$ and $D_2$ be two dessins with $f_i\colon X_i\to \P^1(\C)$ the associated Belyi maps.
Let $Y$ denote the desingularization of the fibered product $\tilde Y=X_1 \times_{\P^1(\C)} X_2$,
fibered along the Belyi maps~$f_i$, and let $f\colon Y\to \P^1(\C)$ denote the resulting branched covering map. Consider the graph~$D$ given by the preimage $f^{-1}(\cI)$ in~$Y$. The graph~$D$ can be combinatorially
described in terms of~$D_1$ and~$D_2$, with bipartite set of vertices
$V_0(D)=V_0(D_1)\times V_0(D_2)$ and $V_1(D)=V_1(D_1)\times V_1(D_2)$ and with set of edges
given by all pairs of edges $(e_1,e_2)\in E(D_1)\times E(D_2)$ with endpoints in $V(D)$. We denote
the fibered product of dessins by $D= D_1 \star D_2$. We correspondingly write $f=f_1\star f_2$
for the fibered product of the Belyi maps. Under this fibered product operations, the degree is multiplicative
$d=d_1 d_2$, and so are the ramifications $m=\# V_0(D)=m_1 m_2$ and $n=\# V_1(D)=n_1 n_2$
and the ramification profiles $\mu_i =\mu_{i,1} \mu_{i,2}$ and $\nu_j =\nu_{j,1} \nu_{j,2}$.

Let $\cA_\Q$ be the algebra over $\Q$ generated by the dessins $D$ with the fibered product as above.
Equivalently we think of elements of $\cA_\Q$ as functions with
finite support $a\colon \cD \to \Q$ from the set~$\cD$ of dessins, with the convolution product
$a_1 \star a_2 (D) =\sum\limits_{D = D_1 \star D_2} a_1(D_1) a_2(D_2)$.
The resulting convolution algebra $\cA_\Q$ is commutative.
We let $\cA_\C=\cA_\Q\otimes_\Q \C$ be the complex algebra obtained by change of coefficients.
The generators of the algebra $\cA_\Q$ are those dessins $D$ that admit no non-trivial
fibered product decomposition $D=D_1\star D_2$, which we refer to as
``indecomposible dessins".

\subsection{Arrangements and semigroup laws}\label{ArrangeSec}

We now consider the commutative semigroups $\cS:=\oplus_{n\geq 1} \N^n$ and $\cS\oplus \cS$,
with the product of $(\mu,\nu)\in \N^m\oplus \N^n$ and $(\mu',\nu')\in \N^{m'}\oplus \N^{n'}$ given
by $(\mu\mu',\nu\nu')$ with $(\mu\mu'=(\mu_i \mu'_{i'})_{(i,i')},\nu\nu'=(\nu_j\nu'_{j'})) \in \N^{mm'}\oplus\N^{nn'}$.
We restrict this semigroup law to the arrangements $\cH(m,n)$ considered in \cite{HaKraLe}.

\begin{lem}\label{subsemiH}
Given $m,n\in \N$ let $\cH(m,n)=\Big\{ (\mu,\nu)\in \N^m\oplus \N^n\,|\,
\sum\limits_{i=1}^m \mu_i=\sum\limits_{j=1}^n \mu_j \Big\}$. The arrangements $\cH(m,n)$
determine a subsemigroup $\cH_{deg}\subset \cS\otimes \cS$.
\end{lem}

\begin{proof} With $\cH_{deg}=\oplus_{m,n} \cH(m,n)\subset \cS\otimes \cS$ we see that
for $(\mu,\nu)\in \cH(m,n)$ and $(\mu'\nu')\in \cH(m',n')$ we have
$(\mu\mu',\nu\nu')\in \cH(mm',nn')$ with $\sum_{(i,i')} \mu_i \mu'_{i'}= d\cdot d' =
\sum_{(j,j')} \nu_j \nu'_{j"}$, where $\sum_i \mu_i =d=\sum_j\nu_j$ and
$\sum_{i'}\mu'_{i'}=d'=\sum_{j'} \nu'_{j'}$.
\end{proof}

This simple fact shows that the data of the ramification profiles at two of the three
ramification points, say at $0$ and $\infty$, of the dessins can be arranged as a
multiplicative semigroup structure. This semigroup operation is consistent with the
algebra operation in $\cA_\Q$ given by the fibered product.

\begin{lem}\label{fiberinvs} The degree, the ramification profiles and the ramification
indices over the points $\{0,1,\infty\}$ are multiplicative with respect to the fibered product operation of Belyi functions $f\colon \Sigma\to \P^1$. If $\mu=(\mu_1,\ldots,\mu_m)$
and $\nu=(\nu_1,\ldots,\nu_n)$ are the ramification profiles over $0$ and $\infty$,
then the ramification index $r$ at $1$ satisfies $r=\chi(\Sigma)-\chi(D)$, where
$\chi$ is the topological Euler characteristic.
\end{lem}

\begin{proof} We have already seen in the construction of the fibered product
dessin $D=D_1\star D_2$ in the previous subsection that the ramification
indices and the ramification profiles behave multiplicatively. By the Riemann--Hurwitz
formula the Euler characteristic $\chi(\Sigma)$ satisfies
\begin{gather*} \chi(\Sigma)=d\cdot \chi\big(\P^1\big) - \sum_{P\in\Sigma} (e_P-1), \end{gather*}
where the sum is taken over the ramification points with the corresponding
ramification index, so we have
\begin{gather*} \chi(\Sigma)=2 d + \sum_{i=1}^m (\mu_i -1) + \sum_{j=1}^n (\nu_j-1)
+ \sum_{k=1}^r (\rho_k-1) = -d +m +n +r, \end{gather*}
where $\rho=(\rho_k)_{k=1}^r$ is the ramification profile over the point $1$.
Since $m+n=\# V(D)$ and $d=\# E(D)$ we have $-d+m+n=\chi(D)$.
\end{proof}

The relation between the ramification index $r$ and the Euler characteristics
$\chi(\Sigma)$ and $\chi(D)$ shows that the enumeration of dessins with fixed
$\mu$, $\nu$, $g$ as in \cite{HaKraLe} can be reformulated as the enumeration of
dessins with fixed $\mu$, $\nu$, $r$. The advantage of this formulation is that the
assignment of the data $(\mu,\nu,r)$ is multiplicative with respect to the
fibered product operation on Belyi functions.

\subsection{Multiplicative invariants and partition function}\label{MultiinvSec}

Let $\Upsilon$ be a multiplicative invariant of dessins d'enfant, that is, an
invariant with the property that $\Upsilon(D\star D')=\Upsilon(D) \star \Upsilon(D')$,
where $D\star D'$ is the fibered product. We assume that $\Upsilon$ takes values in
a group, a semigroup, or an algebra.

As in the previous subsections, we consider the Hilbert space $\ell^2(\cD)$.
The algebra $\cA_\Q$ acts by bounded operators with
$D\cdot \epsilon_{D'}=\epsilon_{D\star D'}$.

\begin{lem}\label{evolUpsilon}
An $\N$-valued multiplicative invariant $\Upsilon$ of dessins d'enfant determines
a time evolution on the algebra $\cA_\Q\otimes_\Q \C$ of the form
$\sigma_t(D)=\Upsilon(D)^{{\rm i}t} D$, implemented by the Hamiltonian
$H \epsilon_D =\log(\Upsilon(D)) \epsilon_D$. The partition function of this
quantum statistical mechanical system is a~Dirichlet series with coefficients
enumerating dessins with assigned invariant~$\Upsilon(D)$.
\end{lem}

\begin{proof} We have ${\rm e}^{{\rm i}tH} D {\rm e}^{-{\rm i}tH} \epsilon_{D'}= \Upsilon(D')^{-{\rm i}t} \Upsilon(D\star D')^{{\rm i}t} \epsilon_{D\star D'}=
\sigma_t(D) \epsilon_{D'}$, by the multiplicative property. The partition function is given by the formal
Dirichlet series
\begin{gather}\label{ZetaUpsilon}
Z(\beta)=\Tr\big({\rm e}^{-\beta H}\big) =\sum_{n\geq 1} \# \{ D\in \cD\,|\, \Upsilon(D)=n \} n^{-\beta} . \qquad \qquad \qquad \qquad \qquad   \qed
\end{gather}\renewcommand{\qed}{}
\end{proof}

In particular, since the ramification profiles, ramification indices and degree behave
multiplicatively with respect to the fibered product of dessins, we obtain a partition
function as above, associated to the counting problem of~\cite{HaKraLe}, where the
multiplicative invariants take values in the semigroup $\cH_{\deg}$ of Lemma~\ref{subsemiH}.

\begin{prop}\label{multiZ}
The partition function \eqref{ZetaUpsilon} associated to the counting problem of~{\rm \cite{HaKraLe}} is
given by
\begin{gather}\label{Zetamunu}
\sum_{m,n,r} \sum_{(\mu,\nu)\in \cH(m,n)} h_{g;\mu,\nu} r^{-s} \mu_1^{-s_1}\cdots \mu_m^{-s_m} \nu_1^{-\sigma_1} \cdots \nu_n^{-\sigma_n}
\end{gather}
with the $h_{g;\mu,\nu}$ as in \eqref{hgmunu} and with
$\cH(m,n)=\big\{ (\mu,\nu)\in \N^m\times \N^n\,|\, \sum_i\mu_i = \sum_j\nu_j \big\}$.
\end{prop}

\begin{proof}First observe that by Lemma~\ref{fiberinvs} we can express the genus $g$ as a function
of the multiplicative invariants $d$, $r$, $m$, $n$, hence we can identify the counting
function $h_{g;\mu,\nu}$ of \eqref{hgmunu} with a function $h_{r;\mu,\nu}$. The piecewise
polynomiality result of \cite{HaKraLe} shows that the $h_{r;\mu,\nu}$ are polynomials
$P_{d,r,C}(\mu,\nu)$ in the variables $(\mu,\nu)$ of degree $2d+1-2r-m-n$,
for all $(\mu,\nu)$ in a fixed chamber $C$ of the arrangement $\cH(m,n)$.
Let $\alpha_1, \ldots, \alpha_m$ and $\kappa_1, \ldots, \kappa_n$ and $\lambda$ be coefficients
in $\R^*_+$ such that, for all integers $r,a_i,b_j >1$ with $\sum_i a_i =\sum_j b_j$,
the products
$r^\lambda a_1^{\alpha_1}\cdots a_m^{\alpha_m} b_1^{\kappa_1}\cdots b_n^{\kappa_n}\neq 1$.
Then given multiplicative invariants $(\mu,\nu)\in \cH_{deg}$ we can form
a~single multiplicative invariant
\begin{gather*} \Upsilon_{m,n}(D):=\prod_{i=1}^m \mu_i(D)^{\alpha_i} \prod_{j=1}^n \nu_j(D)^{\kappa_j}, \end{gather*}
such that $\big\{ D\,|\, \Upsilon_{m,n}(D)= r^\lambda a_1^{\alpha_1}\cdots a_m^{\alpha_m}
b_1^{\kappa_1}\cdots b_n^{\kappa_n}\big\} =\{ D\,|\, \mu_i(D)=a_i, \, \nu_j(D)=b_j \}$.
Given a~sequence $\{ \alpha_i,\kappa_j \}$ that satisfies the property above for
given $d$, $m$, $n$ we obtain a time evolution $\sigma_t(D) =\Upsilon_{m(D),n(D)}(D)^{{\rm i}t} D$
on $\cA_\C=\cA_\Q\otimes_\Q \C$ with partition function given by the formal Dirichlet series
\begin{gather*} Z(\beta) = \sum_{D\in \cD} \Upsilon_{m(D),n(D)}(D)^{-\beta}
 =\sum_{r,m,n} h_{r,m,n;\mu,\nu} r^{-\beta\lambda}
\mu_1^{-\beta \alpha_1}\cdots \mu_m^{-\beta \alpha_m} \nu_1^{-\beta \kappa_1} \cdots \nu_n^{-\beta \kappa_n}. \end{gather*}
This agrees with~\eqref{Zetamunu} for $s=\beta\lambda$, $s_i=\beta \alpha_i$ and $\sigma_j=\beta \kappa_j$.
\end{proof}

\section[Bost--Connes crossed product algebras and the Grothendieck--Teichm\"uller groupoid]{Bost--Connes crossed product algebras\\ and the Grothendieck--Teichm\"uller groupoid}\label{IharaSec}

The basic object of combinatorial topology related to the set of Belyi maps
$f\colon \Sigma \to \P^1$, is the fundamental groupoid of $\P^1\setminus \{0,1,\infty\}$
classifying (modulo homotopy) oriented paths between pairs of, say,
algebraic points of $\P^1(\C ) \setminus \{0,1,\infty\}$.

For studying its interactions with the absolute Galois group, it is convenient
to restrict the set of base points to Deligne's base points at $0$, $1$ and $\infty$ that is,
real tangent directions to these points, as explained in~\cite{Ihara1,Ihara}, and earlier,
although in a less explicit form, in~\cite{Dr}.

This groupoid can be visualized via the {\it Grothendieck--Teichm\"uller} group: product of two cyclic sugroups
generated by loops around $0$ and $1$ respectively, and connecting them involution
generated by a path from $0$ to $1$. We will sometimes refer to this involution
as {\em hidden symmetry}, or else {\em Drinfeld--Ihara} involution.

This section is dedicated to the constructions of quantum statistical mechanical
systems associated to the absolute Galois group $G={\rm Gal}(\bar\Q/\Q)$.
They transfer to the Grothendieck--Teichm\"uller environment
versions of the Bost--Connes algebras considered in \cite{Mar}.

We will start with a description of the
combinatorial version of the (profinite) Grothendieck--Teichm\"uller group ${\bf mGT}$
as the automorphism group of the genus zero modular operad, as was done in \cite{ComMan2}.

\subsection{Bost--Connes algebra with Drinfeld--Ihara involution}\label{IharaBCsec}

The group ${\bf mGT}$ was defined in \cite{ComMan2} in the following way.
One starts with the same projective limit as in the Bost--Connes endomotive (\cite[Section~3.3]{CCM},
and \cite[Section~2.1]{Mar}).
Put $X_n={\rm Spec}(\Q[\Z/n\Z])$ with projections $X_n \to X_m$ when $m|n$ ordered
by divisibility, and consider first the limit $X =\varprojlim_n X_s ={\rm Spec}(\Q[\Q/\Z])$.

We then enrich the action of the automorphism group $\hat\Z^*$ on $X$ and on the Bost--Connes
algebra $\Q[\Q/\Z]\rtimes \N$ corresponding to the
actions of the groups $\Z/n\Z^*$ of the $X_n$, with
the further symmetries $\theta_n$ of $X_n$. Concretely,
the map $\theta_n$ is the involution of $X_n=\{ 0, \ldots, n-1 \}$
given by $k\mapsto n-k+1$.

The non-abelian group ${\bf mGT}_n$ is defined as the subgroup
of the symmetric group $S_n$ generated by $\Z/n\Z^*$ and $\theta_n$. One defines the group
${\bf mGT}=\varprojlim_n {\bf mGT}_n$

\begin{lem}\label{invtheta}The involutions $\theta_n\colon X_n \to X_n$ are compatible with the
maps $\sigma_m\colon X_{nm} \to X_n$ of the projective system of the Bost--Connes endomotive, namely $\sigma_m \circ \theta_{nm}=\theta_n \circ \sigma_m$.
They determine an involution $\theta\colon X \to X$ on the projective limit.
\end{lem}

\begin{proof} If we identify $X_n$ with the set of roots of unity of
order $n$, this involution can be equivalently written as $\theta_n(\zeta)=\zeta_n \cdot \zeta^{-1}$,
where $\zeta_n=\exp(2\pi{\rm i} /n)$. Also,
identifying the set $X_n$ with roots of unity of order $n$, the maps
$\sigma_m\colon X_{nm}\to X_n$ are given by raising to the $m$-th power,
$\sigma_m(\zeta)=\zeta^m$. Thus, we have $\sigma_m( \theta_{nm}(\zeta))=
\zeta_{nm}^m \cdot \big(\zeta^{-1}\big)^m =\theta_n(\zeta^m)=\theta_n(\sigma_m(\zeta))$.
\end{proof}

We refer to the map $\theta\colon X \to X$ as the Drinfeld--Ihara involution of the
Bost--Connes endomotive mentioned above. The following reformulation of this involution in terms of its action on the Bost--Connes algebra then follows directly.

\begin{prop}\label{thetanlem}The involution $\theta\colon X \to X$ described above induces a self-map of the
arithmetic Bost--Connes algebra $\Q[\Q/\Z]\rtimes \N$,
which acts by $\theta(e(r))=e\big(\frac{1-a}{b}\big)$ on generators
$e(r)$ with $r\in \Q/\Z$, $r=a/b$ for integers with $(a,b)=1$.

The group ${\bf mGT}$ acts on the Bost--Connes quantum statistical mechanical
system compatibly with the time evolution and preserving the arithmetic subalgebra.
\end{prop}

\begin{proof} The first statement follows directly from Lemma~\ref{invtheta}, which determines
the action of $\theta$ on the generators $e(r)$ of the abelian subalgebra $\Q[\Q/\Z]$.
The action on the arithmetic Bost--Connes algebra $\Q[\Q/\Z]\rtimes \N$ is then
determined by letting $\theta$ act as the identity on the remaining generators $\mu_n$.

This action extends to an action on the Bost--Connes $C^*$-algebra $C^*(\Q/\Z)\rtimes \N$,
which preserves the arithmetic subalgebra. This action
is compatible with the time evolution since the action is trivial on the
generators $\mu_n$ so that $\theta\circ \sigma_t = \sigma_t \circ \theta$.

The group ${\bf mGT}$ is generated by elements of $\hat\Z^*$ and
the Ihara involution $\theta$. In fact, we can write arbitrary elements in ${\bf mGT}$
as sequences $\theta^{\epsilon_0} \gamma_1 \theta \gamma_2 \theta \cdots \theta \gamma_n \theta^{\epsilon_1}$
for $\epsilon_0,\epsilon_1\in \{0,1\}$ and with the $\gamma_k\in \hat\Z^*$.
The group $\hat\Z^*$ acts by automorphisms of the Bost--Connes quantum statistical
mechanical system with $\gamma \circ \sigma_t = \sigma_t\circ \gamma$, preserving
the abelian subalgebra. Since the action of~$\theta$ also has these properties, we
obtain an action of ${\bf mGT}$ as symmetries of the Bost--Connes quantum statistical mechanical system.
\end{proof}

To be more precise, the effect of enlarging the group of symmetries from $\hat\Z^*$ to ${\bf mGT}$
implies that we can consider a larger family of covariant representations of
the Bost--Connes system on the same Hilbert space $\ell^2(\N)$, parameterized
by elements of ${\bf mGT}$, by setting, for $\alpha\in {\bf mGT}$,
\begin{gather*} \pi_\alpha(e(r)) \epsilon_n = \alpha(\zeta_r)^n \epsilon_n, \end{gather*}
where $\zeta_r$ is the image of $r\in \Q/\Z$ under a fixed embedding of
the abstract roots of unity~$\Q/\Z$ in~$\C^*$. The isometries $\mu_n$ from \cite{Mar} act in the
usual way, $\mu_m\epsilon_n=\epsilon_{nm}$. This change does not affect the
generating Hamiltonian of the time evolution, nor the corresponding partition
function given by the Riemann zeta function. The low temperature Gibbs
states are still given by polylogarithm functions evaluated at roots of unity,
normalized by the Riemann zeta function. However, when evaluating
zero temperature KMS states on elements of the arithmetic subalgebra,
we now have an action of the larger group~${\bf mGT}$ on the values
in~$\Q^{ab}$. This action by a non-abelian group no longer has a Galois
interpretation as an action on roots of unity, hence it does not directly give
us a way to extend the Bost--Connes system to non-abelian Galois theory.

In the remaining part of this section, we will show that there are other possible
ways of incorporating the Drinfeld--Ihara involution in a construction generalizing the
Bost--Connes algebra that lead to more interesting changes to the
structure of the resulting quantum statistical mechanical system.

\subsection{Crossed product algebras and field extensions}\label{crossFQsec}

The variant of the Bost--Connes quantum statistical mechanical system
constructed in~\cite{Mar} was aimed at merging two different aspects of $\F_1$-geometry: the
relation described in~\cite{CCM2} of the integral Bost--Connes algebra to the
extensions~$\F_{1^m}$ of~\cite{KapSmi}, and the analytic functions over $\F_1$
constructed in~\cite{Man}. The modified Bost--Connes algebra considered in~\cite{Mar} involves an action of endomorphisms on the Habiro ring of~\cite{Hab}.

We consider here a simpler variant of the algebra of \cite{Mar}, which will enable us
 to present naturally the transition from abelian to non-abelian Galois groups
in a form that recovers the embedding of the absolute Galois group in the
Grothendieck--Teichm\"uller group as described by Ihara in \cite{Ihara}.

Let $\cF_\Q=\Q[t^r; r\in \Q^*_+]$ be the polynomial algebra in rational powers of a variable $t$.
For $s\in \Q_+^*$ let $\sigma_s$ denote the action $\sigma_s(f)(t) =f(t^s)$ for $f\in \cF_\Q$.
Thus, we can form the group crossed product algebra $\cF_\Q \rtimes \Q^*_+$.

The subsemigroup $\N$ of the group $\Q^*_+$ induces an action by endomorphisms on
the subalgebra $\cP_\Q=\Q[t]$ by $\sigma_n(f)(t)=f(t^n)$. Note that, unlike
the morphisms $\sigma_n$ of the Bost--Connes system, acting by endomorphisms of
$\Q[\Q/\Z]$, which are surjective but not injective, in this case the morphisms $\sigma_n$
acting on $\cP_\Q$ are injective.

Consider, as in the case of the original
Bost--Connes algebra, generators $\mu_n$ and $\mu_n^*$, for $n\in \N$, satisfying
the relations $\mu_n\mu_m=\mu_{nm}$, $\mu_n^*\mu_m^*=\mu^*_{nm}$, $\mu_n^*\mu_n=1$,
and in the case where $(n,m)=1$ also $\mu_n \mu_m^*=\mu_m^*\mu_n$.
Consider the algebra generated by $\cP_\Q$ and the $\mu_n$, $\mu_n^*$ with the relations
$\mu_n^* f = \sigma_n(f) \mu_n^*$ and $f \mu_n =\mu_n \sigma_n(f)$. The elements
$\pi_n=\mu_n\mu_n^*$ satisfy $\pi_n^2=\pi_n=\pi_n^*$ and $\pi_n f =f \pi_n$ for all $f\in \cP_\Q$
and all $n\in\N$.

Unlike the Bost--Connes case, the $\pi_n$ are not elements of the algebra $\cP_\Q$
on which the semigroup is acting.

Denote by $\rho_n$ the endomorphisms $\rho_n(f) :=\mu_n f \mu_n^*$.
They satisfy the relations $\sigma_n( \rho_n(f))=f$ and
$\rho_n( \sigma_n(f))= \pi_n f$. The map that sends $\mu_n f \mu_n^*$ to $f(t^{1/n})$ identifies
the direct limit of the injective maps $\sigma_n\colon \cP_\Q \to \cP_\Q$ with $\cF_\Q$ on which
the action of the $\sigma_n$ becomes an action by automorphisms. This leads to the
crossed product algebra $\cF_\Q \rtimes \Q^*_+$.

Furthermore, in \cite{Mar} there was considered the
inverse limit $\hat\cF_\Q=\varprojlim_N \cF_\Q/\cJ_N$ with respect to the ideals $\cJ_N$
generated by $(t^r)_N=(1-t^r)\cdots \big(1-t^{rN}\big)$ for $r\in \Q^*_+$, which is modelled on the Habiro
ring construction of~\cite{Hab}. Here instead we replace $\cF_\Q$ and $\cP_\Q$ with the $\Q$-algebras
$\cA_{\bar\Q}=\bar\Q\{ \{ t \}\}$, given by the algebraically closed
field of formal Puiseux series $\bar\Q\{ \{ t \}\}=\cup_{N\geq 1} \bar\Q\big(\big( t^{1/N}\big)\big)$, and
the field of rational functions $\cB_\Q=\bar\Q(t)$, respectively.

\subsection{The Drinfeld--Ihara hidden symmetry involution}

The involution $t\mapsto 1-t$ maps isomorphically $\bar\Q\{ \{ t \}\} \to \bar\Q\{ \{ 1-t \}\}$,
and as in \cite{Ihara} we denote by~$M$ and~$M'$, the maximal Galois extensions of
$\bar\Q(t)$ unramified outside of~$\{ 0,1,\infty\}$ inside $\bar\Q\{ \{ t \}\}$ and in $\bar\Q\{ \{ 1-t \}\}$, respectively, with the involution mapping $M \to M'$. The absolute Galois group $G_\Q={\rm Gal}(\bar\Q/\Q)$ acts on $\bar\Q\{ \{ t \}\}$ and $\bar\Q\{ \{ 1-t \}\}$ by acting on the Puiseux coefficients and induces actions on $M$ and $M'$. Moreover, as discussed in~\cite{Ihara}, the Galois group ${\rm Gal}\big(M/\bar\Q(t)\big)$ can be identified with the profinite fundamental group $\hat F_2=\hat \pi_1\big(\P^1\smallsetminus \{0,1,\infty\}, (0,1)\big)$,
where $F_2= \pi_1\big(\P^1\smallsetminus \{0,1,\infty\}, (0,1)\big)$ is the free group on two generators.

As observed in~\cite{Ihara}, the maximal abelian subextension of $\bar\Q(t)$ in $M$ is generated by
all the elements $t^{1/N}$ and $(1-t)^{1/N}$. Thus, we can view a slightly modified version of the
construction of the crossed product algebra $\cF_\Q \rtimes \Q^*_+$ mentioned above as a
construction of this maximal abelian extension, just like the original Bost--Connes system can be
regarded as a noncommutative geometry construction of the maximal abelian extension of $\Q$.

\subsection{Semigroup action and maximal abelian extension}

More precisely, we need to modify the above crossed product algebra construction to account
for the missing involution $t\leftrightarrow 1-t$. This can be done by replacing the polynomial algebra
$\cP_\Q=\Q[t]$ with $\cB_\Q=\bar\Q(t)$ and, in addition to the endomorphisms $\sigma_n(f)(t)=f(t^n)$
considering an extra generator $\tau(f)(t)=f(1-t)$.

\begin{defn}\label{Ssemigr}
Denote by $\cS$ the non-abelian semigroup given by the free product $\cS=\N \star \Z/2\Z$. Its elements
can be written in the form \begin{gather*}\mu_s=F^{\epsilon_0} \mu_{n_1} F \mu_{n_2} F \cdots
F \mu_{n_k} F^{\epsilon_1} , \end{gather*} with $s=s(\epsilon_0,\epsilon_1,\underline{n})$ for
$\epsilon_0,\epsilon_1\in \{ 0,1 \}$ and $\underline{n}=(n_1,\ldots, n_k)$.

Denote by $\cS_0$ and $\cS_1$ the two abelian sub-semigroups given, respectively, by $\cS_0=\N$
and $\cS_1=F \N F=\{ \tilde\mu_n:=F \mu_n F\,|\, n\in \N \}$.
\end{defn}

The following statement is obtained by directly adapting the original argument of \cite{Mar} as
discussed in the subsection \ref{crossFQsec} above.

\begin{lem}\label{S01cross} The sub-semigroups $\cS_0$ and $\cS_1$ define the action of endomorphisms $\sigma_n$ and $\tilde\sigma_n$ upon $\cB_\Q=\bar\Q(t)$, with partial inverses $\rho_n$ and $\tilde\rho_n$ respectively.

The direct limits with respect to
injective endomorphisms $\sigma_n$ and $\tilde\sigma_n$ can be identified, respectively, with the crossed
product algebras $\cB_{\Q,0}\rtimes \Q^*_+$ and $\cB_{\Q,1}\rtimes \Q^*_+$, where
$\cB_{\Q,0}=\bar\Q(t^r ; r\in \Q^*_+)$
and $\cB_{\Q,1}=\bar\Q((1-t)^r ; r\in \Q^*_+)$.
\end{lem}

\begin{proof} The elements $\tilde\mu_n= F \mu_n F$, with $\tilde\mu_n^*=F \mu_n^* F$, satisfy
$\tilde\mu_n^* f \tilde\mu_n =\tilde\sigma_n(f)$. Here $\tilde\sigma_n=\tau \circ \sigma_n \circ \tau$
is the semigroup action $\tilde\sigma_{nm}=\tilde\sigma_n \circ \tilde\sigma_m$ given by
$\tilde\sigma_n(f)(t) = f\big(1-(1-t)^n\big)$. Correspondingly we have $\tilde\mu_n f \tilde\mu_n^* =\tilde\rho_n(f)$ with $\tilde\rho_n(f)(t)=f\big(1-(1-t)^{1/n}\big)$.
\end{proof}

Together with the description of~\cite{Ihara} of the maximal abelian subextension of $\bar\Q(t)$ in $M$
as generated by all the elements $t^{1/N}$ and $(1-t)^{1/N}$ (see also~\cite{MalMat}), we then obtain
this maximal abelian subextension in the following way.

\begin{prop}Consider the algebras $\cB_{\Q,0}$ and $\cB_{\Q,1}$ of Lemma~{\rm \ref{S01cross}} as
subalgebras of $M$. They together generate the maximal abelian subextension of $M$.
\end{prop}

\subsection{Semigroup action and non-abelian extensions}

We now consider the semigroup action of the full $\cS=\N \star \Z/2\Z$ and the endomorphisms $\sigma_s$ of $\cB_\Q=\bar\Q(t)$, given by $\sigma_s(f)=\mu_s^* f \mu_s$ with partial inverses $\rho_s(f) =\mu_s f \mu_s^*$.

\begin{lem}\label{Fcross} Consider the $\bar\Q$-algebra generated by the field of rational functions $\cB_\Q=\bar\Q(t)$
and an additional generator $F$ satisfying relations $F=F^*$, $F^2={\rm Id}$, and
$F f F=\tau(f)$, for all $f\in \bar\Q(t)$.

The resulting algebra is a group crossed product $\tilde\cB_\Q =\cB_\Q \rtimes \Z/2\Z$,
where the $\Z/2\Z$ action is the Drinfeld--Ihara involution $t\mapsto 1-t$.
\end{lem}

\begin{proof} The generator $F$ is a sign operator since it satisfies $F=F^*$, $F^2={\rm Id}$,
and it implements the Drinfeld--Ihara involution by the relation $F f F=\tau(f)$, hence we obtain
the action of $\cB_\Q \rtimes \Z/2\Z$ as superalgebra.
\end{proof}

Consider now the $\bar{\Q}$-algebra generated by the field of rational functions $\cB_\Q=\bar\Q(t)$
and a set of generators $\mu_n$, $\mu_n^*$ with relations
\begin{gather*}
\mu_n\mu_m=\mu_{nm},\qquad \mu_n^*\mu_m^*=\mu^*_{nm}, \qquad \mu_n^*\mu_n=1
\end{gather*}
 for all $n\in \N$, as well as
\begin{gather*}
\mu_n \mu_m^*=\mu_m^*\mu_n
\end{gather*}
for $(n,m)=1$, and
\begin{gather*}
\mu_n^* f = \sigma_n(f) \mu_n^*,\qquad f \mu_n =\mu_n \sigma_n(f)
\end{gather*}
for all $f\in \bar\Q(t)$ and all $n\in \N$. Moreover, include an additional generator $F$ satisfying
\begin{gather*}
F=F^*,\qquad F^2={\rm Id},\qquad F f F=\tau(f)
\end{gather*}
for all $f\in \bar\Q(t)$.

This $\bar{\Q}$-algebra contains the group crossed product $\cB_\Q \rtimes \Z/2\Z$ of Lemma~\ref{Fcross} and the full semigroup $\cS=\N \star \Z/2\Z$. In fact the same argument as in Section~\ref{crossFQsec} shows that we can describe this algebra in terms of a direct limit over the endomorphisms $\sigma_s$ acting on $\cB_\Q$, where the action in the limit becomes implemented by automorphisms, so that the resulting algebra is a group crossed product by $\Upsilon=\Q^*_+\star \Z/2\Z$.

\begin{lem}\label{barQSigmas} For an element $s\in \cS$ and for $f(t)=t$ consider the function $\rho_s(t)$. This determines a curve $\Sigma_s$ which is a branched covering of $\P^1$ branched at
$\{0,1,\infty\}$, with field of functions~$\bar\Q(\Sigma_s)$ given by the smallest finite
Galois extension of $\bar\Q(t))$ unramified outside of $\{0,1,\infty\}$ that contains~$\rho_s(t)$, with Galois group $G_s:={\rm Gal}\big(\bar\Q(\Sigma_s)/\bar\Q(t)\big)$.
\end{lem}

\begin{proof}Using notation of the Definition~\ref{Ssemigr}, we can write $s$ in the form
$s=s(\epsilon_0,\epsilon_1,\underline{n})$ for some $\epsilon_0,\epsilon_1\in \{ 0,1 \}$ and some $\underline{n}=(n_1,\ldots, n_k)$. Then the element $z=\rho_s(t)$ is given by $\rho_s(t)=\tau^{\epsilon_0}\rho_{n_1} \tau\rho_{n_2}\cdots \tau\rho_{n_k}\tau^{\epsilon_1}$. It satisfies a polynomial equation of the form $\tau^{\epsilon_0}\sigma_{n_1}\tau\cdots\tau \sigma_{n_k}\tau^{\epsilon_1}(z)=t$. For example, for $s=\mu_n F \mu_m F$ we have $z=\rho_s(t)= 1-\big(1-t^{1/n}\big)^{1/m}$ satisfying $\big(1-(1-z)^m\big)^n=t$, while for $s=\mu_n F \mu_m$ we have $z=\rho_s(t)=\big(1-t^{1/n}\big)^{1/m}$ with $\big(z^m-1\big)^n=t$. Thus, the element~$\rho_s(t)$ detemines a branched covering of $\P^1$ with branch locus in the set $\{0,1,\infty\}$.

Any finite field extension of $\C(t)$ unramified outside of $\{ 0,1,\infty \}$ determines a compact complex Riemann surface $\Sigma_s$ with a branched covering of $\P^1$ unramified outside of these three points. Such $\Sigma_s$ is defined over a number field and descends to an algebraically closed fields
(see the descent technique described in \cite[Section~2.1]{MalMat}).
This gives the required result for the field of functions $\bar\Q(\Sigma_s)$.
\end{proof}

 \begin{prop}\label{barQS} The fields $\bar\Q(\Sigma_s)$ form a direct system over the semigroup $\cS$
ordered by divisibility. The direct limit $\bar\Q(\Sigma_\cS):=\varinjlim_{s\in \cS} \bar\Q(\Sigma_s)$,
with symmetries $G_\cS=\varprojlim_{s\in\cS} G_s$, contains all functions of the form $(1-t^r)^{r'}$ for $r,r'\in \Q^*_+$. The resulting action of~$\cS$
on the direct limit~$\bar\Q(\Sigma_\cS)$ becomes invertible and extends to an action of the
enveloping group $\Upsilon=\Q^*_+\star \Z/2\Z$.
\end{prop}

\begin{proof} For all $s\in \cS$, the morphisms $\sigma_s$ acting on $\cB_\Q$ are injective.
This can be seen by explicitly writing \begin{gather*}\sigma_s(f)=
\tau^{\epsilon_0}\sigma_{n_1} \tau \sigma_{n_2}\cdots \tau \sigma_{n_k} \tau^{\epsilon_1}(f),\end{gather*}
for an element $s=s(\epsilon_0,\epsilon_1,\underline{n}) \in \cS$, where both the involution $\tau$
and the endomorphisms $\sigma_{n_i}$ are injective. The map $f\mapsto \sigma_s(f)$ gives an
embedding of $\bar\Q(t)$ inside $\bar\Q(\Sigma_s)$. Thus, on the limit~$\bar\Q(\Sigma_\cS)$ the
morphisms $\sigma_s$ become invertible and the field~$\bar\Q(\Sigma_\cS)$ contains all the
$\rho_s(f)$ for all $f\in \bar\Q(t)$. In particular, it then contains all the functions of the form
$(1-t^r)^{r'}$ for $r,r'\in \Q^*_+$.
\end{proof}

\begin{rem}\label{hCrem} By viewing $\bar\Q(\Sigma_s)$ as a subfield of
$\bar\Q\{ \{ t \}\}$ (or of $\bar\Q\{ \{ 1-t \}\}$), we can identify elements in $\bar\Q(\Sigma_s)$ with
analytic functions on the region $(0,1)_\C=\C\smallsetminus ((-\infty, 0]\cup [1,\infty))$ considered in~\cite{Ihara}, with a convergent Puiseux series expansion on $(0,1)$, where all
the $(1-t^r)^{r'}$, for $r,r'\in \Q^*_+$, are assumed to be positive real for $t\in (0,1)$.
\end{rem}

\subsection[Quantum statistical mechanics of $\Q^*_+$]{Quantum statistical mechanics of $\boldsymbol{\Q^*_+}$}

The definition of a quantum statistical mechanical system associated to this construction, relies on
the general setting for a quantum statistical mechanics for the group~$\Q^*_+$ described in~\cite{MarXu}, which we recall briefly here, adapted to our environment.

Let $\cH=\ell^2(\Q^*_+)$ be the Hilbert space with canonical orthonormal basis $\{ \epsilon_r \}_{r\in \Q^*_+}$. For $r=p_1^{n_1}\cdots p_k^{n_k}$ the prime decomposition of $r\in\Q^*_+$ with $p_i$
distinct primes and $n_i\in \Z$, consider the densely defined unbounded linear operator
\begin{gather}\label{HamQ}
 H \epsilon_r = (|n_1| \log(p_1)+\cdots+ |n_k| \log(p_k)) \epsilon_r .
\end{gather}
The operator ${\rm e}^{-\beta H}$ is trace class for $\beta >1$ with
\begin{gather*}
\Tr\big( {\rm e}^{-\beta H}\big) = \frac{\zeta(\beta)^2}{\zeta(2\beta)} =\sum_{n\geq 1} \frac{2^{\omega(n)}}{n^\beta}.
\end{gather*}
Here $\zeta(\beta)$ is the Riemann zeta function, and $\omega(n)$ denotes the number of pairwise distinct
prime factors of $n$.

Consider the algebra $C^*_r(\Q^*_+)$
acting via the left regular representation on $\ell^2(\Q^*_+)$. The smallest
sub-$C^*$-algebra $\cA(\Q^*_+)$ of $\cB\big(\ell^2(\Q^*_+)\big)$ that contains $C^*_r(\Q^*_+)$ and is
preserved by the time evolution $\sigma_t(X)={\rm e}^{-{\rm i}t H} X {\rm e}^{{\rm i}t H}$, with $H$ as
above, is generated by $C^*_r(\Q^*_+)$ and the projections $\Pi_{k,\ell}$
with $\Pi_{k,\ell} \epsilon_{a/b}=\epsilon_{a/b}$ when $k|a$ and $\ell |b$, for $(a,b)=1$,
and zero otherwise. The spectral projections of $H$ belong to this $C^*$-algebra, hence
the time evolution acts by inner automorphisms, see \cite{MarXu} for a more detailed
discussion.

\subsection[Quantum statistical mechanics of $\cS$ and $\Upsilon$]{Quantum statistical mechanics of $\boldsymbol{\cS}$ and $\boldsymbol{\Upsilon}$}

In our setting, we consider first the semigroup $\cS=\N\star \Z/2\Z$. It acts upon the
Hilbert space~$\ell^2(\cS)$ via the left regular representation. Thus, for $\{ \epsilon_s \}_{s\in \cS}$ the standard orthonormal basis of~$\ell^2(\cS)$ and for a given $s'\in \cS$ we have~$\mu_{s'} \epsilon_s =\epsilon_{s's}$.

\begin{lem}\label{HSlem} Consider the semigroup $C^*$-algebra $C^*_r(\cS)$ acting on $\ell^2(\cS)$ by the left regular representation. The transformations $\sigma_t(\mu_s)=n^{{\rm i}t} \mu_s$, where
$n=n_1\cdots n_k$ for $s=s(\epsilon_0,\epsilon_1,\underline{n})$ with $\underline{n}=(n_1,\ldots, n_k)$,
define a time evolution $\sigma\colon \R \to {\rm Aut}(C^*_r(\cS))$. Consider then the densely defined unbounded linear operator on the Hilbert space $\ell^2(\cS)$ given by
\begin{gather*}
H \epsilon_s =\log(n_1\cdots n_k) \epsilon_s,
\end{gather*}
for $s=s(\epsilon_0,\epsilon_1,\underline{n})$ with $\epsilon_0,\epsilon_1\in \{ 0,1 \}$ and $\underline{n}=(n_1,\ldots, n_k)$, with $n_i\in \N$, $n_i>1$.
The operator $H$ is the Hamiltonian that generates the time evolution $\sigma_t$
on the $C^*$-algebra $C^*_r(\cS)$. The partition function of the Hamiltonian $H$ is given by{\samepage
\begin{gather*}
Z(\beta) = \Tr\big({\rm e}^{-\beta H}\big) = \frac{4}{2-\zeta(\beta)},
\end{gather*}
where $\zeta(\beta)$ is the Riemann zeta function.}
\end{lem}

\begin{proof} In order to check that $\sigma_t(\mu_s)=n^{{\rm i}t} \mu_s$ defines a time evolution as above,
we first notice that the assignment $n\colon s\mapsto n(s)=n_1\cdots n_k$ for $s=s(\epsilon_0,\epsilon_1,\underline{n})$ with $\underline{n}=(n_1,\ldots, n_k)$ is a~semigroup homomorphism $\cS \to \N$.

In order to see that the operator $H$ is the infinitesimal generator of this time evolution we check that, for all $s\in \cS$ and for all $t\in \R$, we have $\sigma_t(\mu_s)={\rm e}^{{\rm i}t H} \mu_s {\rm e}^{-{\rm i}tH}$ as operators in $\cB(\ell^2(\cS))$.

Indeed, both sides act on a basis element $\epsilon_{s'}$ as $n(ss')^{{\rm i}t} n(s')^{-{\rm i}t} \epsilon_{ss'} =n(s)^{{\rm i}t} \epsilon_{ss'}$. The multipli\-ci\-ty of an eigenvalue $\log n$ of the Hamiltonian $H$ is $4 P_n$, where the factor $4$ accounts for the
four choices of $\epsilon_0,\epsilon_1\in \{ 0,1\}$ and $P_n$ is the total number of ordered factorizations
of $n$ into nontrivial positive integer factors.

Thus, the partition function equals
\begin{align*}
Z(\beta)& =\Tr\big({\rm e}^{-\beta H}\big)=4 \sum_n P_n n^{-\beta} =4\left(1+\sum_{k=1}^\infty \sum_{n=2}^\infty n^{-\beta} \sum_{n=n_1\cdots n_k} 1\right)\\
& = 4\left(1 + \sum_{k\geq 1} \prod_{i=1}^k \sum_{n_i\geq 2} n_i^{-\beta} \right)
 =4 \sum_{k=0}^\infty (\zeta(\beta)-1)^k =\frac{4}{2-\zeta(\beta)}.\tag*{\qed}
 \end{align*}\renewcommand{\qed}{}
\end{proof}

We now consider the enveloping group $\Upsilon=\Q^*_+\star \Z/2\Z$. We proceed as in the
quantum statistical mechanics of $\Q^*_+$ described in~\cite{MarXu} and recalled above.
For a positive rational number $r\in \Q^*_+$ with prime decomposition $r=p_1^{a_1}\cdots p_\ell^{a_\ell}$ where $p_i$ a pairwise distinct primes and $a_i\in \Z$, $a_i\neq 0$, let
\begin{gather}\label{nr}
n(r)=p_1^{|a_1|}\cdots p_\ell^{|a_\ell|}.
\end{gather}

\begin{lem}\label{HUlem} Consider on $\ell^2(\Upsilon)$ the densely defined unbounded linear operator
\begin{gather}\label{HU}
H \epsilon_\upsilon = \log (n(r_1)\cdots n(r_k)) \epsilon_\upsilon,
\end{gather}
for an element $\upsilon=\upsilon(\epsilon_0,\epsilon_1,\underline{r})$ with
$\epsilon_0,\epsilon_1\in \{0,1\}$ and $\underline{r}=(r_1,\ldots, r_k)$ with $r_i\in \Q^*_+$,
$r_i\neq 1$. The operator ${\rm e}^{-\beta H}$ satisfies
\begin{gather*}
Z(\beta) =\Tr\big({\rm e}^{-\beta H}\big) =\frac{4 \zeta(2\beta)}{2 \zeta(2\beta) - \zeta(\beta)^2}.
\end{gather*}
\end{lem}

\begin{proof} The eigenvalue $\log n$ of $H$ has multiplicity $4 \sum\limits_{n=n_1\cdots n_k} 2^{\omega(n_1)+\cdots +\omega(n_k)}$ where the factor of $4$ accounts for the choices of $\epsilon_0,\epsilon_1\in \{0,1\}$ and
each factor $2^{\omega(n_i)}$, with $\omega(n_i)$ the number of distinct prime factors of $n_i$, accounts
for all the $r_i\in \Q^*_+$ with $n(r_i)=n_i$.

Thus, we obtain
\begin{align*}
 \Tr\big({\rm e}^{-\beta H}\big) & = 4 \left(1+\sum_{k\geq 1}
\sum_{n=n_1\cdots n_k} 2^{\omega(n_1)+\cdots +\omega(n_k)} n^{-\beta} \right)\\
& = 4 \left(1+ \sum_{k\geq 1} \prod_{i=1}^k \sum_{n_i\geq 2} \frac{2^{\omega(n_i)}}{n_i^\beta} \right)=
4 \left(1+ \sum_k \frac{\zeta(\beta)^2}{\zeta(2\beta)} -1\right)^k = \frac{4}{2-\frac{\zeta(\beta)^2}{\zeta(2\beta)}}.\tag*{\qed}
\end{align*}\renewcommand{\qed}{}
\end{proof}

As in the case of the quantum statistical mechanics of $\Q^*_+$, the time evolution on $\cB\big(\ell^2(\Upsilon)\big)$ ge\-ne\-rated by the Hamiltonian~$H$ of \eqref{HU}, no longer preserves the reduced group $C^*$-algebra~$C^*_r(\Upsilon)\!$ acting on $\ell^2(\Upsilon)$ through the left regular representation, since we have
$n(r r')\!=\!n(r) n(r') (b,u) (a,v)\!$ for $r=u/v$ with $(u,v)=1$ and $r'=a/b$ with $(a,b)=1$. This implies the following behavior of the time evolution on the generators of $C^*_r(\Upsilon)$.

\begin{lem}\label{innertime} For $\underline{m}=(m_1,\ldots, m_k)$ and
$\underline{\ell}=(\ell_1,\ldots,\ell_k)$, let $\Pi_{\underline{m},\underline{\ell}}$ denote the projection on~$\ell^2(\Upsilon)$ given by $\Pi_{\underline{m},\underline{\ell}} \mu_\upsilon =\mu_\upsilon$ if
$\upsilon=\upsilon(\epsilon_0,\epsilon_1,\underline{r})$ with $\underline{r}=(r_1,\ldots,r_k)$ where
$r_i=a_i/b_i$, and $(a_i,b_i)$ are coprime integers such that $k_i | b_i$ and $\ell_i | a_i$.

The sub-$C^*$-algebra $\cA_\Upsilon$ of $\cB(\ell^2(\Upsilon))$ generated by $C^*_r(\Upsilon)$
and by the projections $\Pi_{\underline{m},\underline{\ell}}$ is stable under the time evolution
$\sigma_t\colon X \mapsto {\rm e}^{{\rm i}tH} X {\rm e}^{-{\rm i}tH}$ on $\cB\big(\ell^2(\Upsilon)\big)$. This time evolution acts on $\cA_\Upsilon$ by inner automorphisms.
\end{lem}

\begin{proof} Consider an element $s=s(\epsilon_0,\epsilon_1,\underline{r})$ with $\underline{r}=(r_1,\ldots, r_k)$ and $r_i=a_i/b_i$ where $u_i,v_i\in \N$ with $(u_i,v_i)=1$. We have
\begin{gather*} {\rm e}^{{\rm i}tH} \mu_s {\rm e}^{-{\rm i}tH} \epsilon_{s'} = n(ss')^{{\rm i}t} n(s')^{-{\rm i}t} \mu_s\epsilon_{s'} =
\sum_{m_i | u_i, \ell_i | v_i} n(s)^{{\rm i}t} (m_1\cdots m_k)^{{\rm i}t} (\ell_1\cdots \ell_k)^{{\rm i}t} \mu_s \Pi_{\underline{m},\underline{\ell}}. \end{gather*}
This shows that the sub-$C^*$-algebra of $\cB(\ell^2(\Upsilon))$ generated by $C^*_r(\Upsilon)$
and by the projections~$\Pi_{\underline{m},\underline{\ell}}$ is stable under the time evolution $\sigma_t$.
As in the case of $\Q^*_+$ discussed in~\cite{MarXu} the spectral projections of~$H$ are in the algebra
$\cA_\Upsilon$ hence the time evolution is acting by inner automorphisms on~$\cA_\Upsilon$.
\end{proof}

\subsection{Quantum statistical mechanics and the Drinfeld--Ihara embedding}

We now consider the crossed product algebra $\bar\Q(\Sigma_\cS)\rtimes \Upsilon$ introduced in
Proposition~\ref{barQS}, and its quantum statistical properties.

By Remark~\ref{hCrem}, elements $h\in \bar\Q(\Sigma_\cS)$ determine analytic functions
$h_\C$ in the region $(0,1)_\C$ with convergent Puiseux series.

\begin{defn}\label{repQS}Given a function $h\in \bar\Q(\Sigma_\cS)$ and the choice of a point $\tau\in (0,1)$, we define a~linear operator $\pi_\tau(h)$ on $\ell^2(\Upsilon)$ by setting
\begin{gather}\label{pitauh}
\pi_\tau(h) \epsilon_\upsilon = \sigma_\upsilon(h)_\C (\tau) \epsilon_\upsilon ,
\end{gather}
where for $\upsilon=\upsilon(\epsilon_0,\epsilon_1,\underline{r})$, with
$\underline{r}=(r_1,\ldots, r_k)$, we put
\begin{gather*}
\sigma_\upsilon(h)=\tau^{\epsilon_0}\circ \sigma_{r_1}\circ \tau \circ \cdots \circ \tau\circ \sigma_{r_k}\circ \tau^{\epsilon_1} (h) .
\end{gather*}
\end{defn}

\begin{rem}\label{boundh}If the analytic function $h_\C$ is bounded on the unit interval,
then for any choice of $\tau\in (0,1)$ the operator of~\eqref{pitauh} is bounded, $\pi_\tau(h)\in \cB\big(\ell^2(\Upsilon)\big)$.
\end{rem}

\begin{lem}\label{crossYact} For any choice of a point $\tau\in (0,1)$, the operators $\pi_\tau(h)$ from \eqref{pitauh} determine an action of the crossed product algebra $\bar\Q(\Sigma_\cS)\rtimes \Upsilon$ on the Hilbert space $\ell^2(\Upsilon)$.
\end{lem}

\begin{proof} We let the generators $\mu_\upsilon$ with $\upsilon\in \Upsilon$ of the algebra act on
the Hilbert space through the left regular representation
$\mu_\upsilon \epsilon_{\upsilon'}=\epsilon_{\upsilon\upsilon'}$. It is then sufficient to check that
the operators $\pi_\tau(h)$ of~\eqref{pitauh} satisfy the relations
\begin{gather*}\mu_\upsilon^* \pi_\tau(h) \mu_\upsilon =\pi_\tau(\sigma_\upsilon(h)) \qquad \text{and} \qquad \mu_\upsilon \pi_\tau(h) \mu_\upsilon^* = \pi_\tau(\rho_\upsilon(h)) . \end{gather*}
We have $\epsilon_{\upsilon\upsilon'}=\mu_\upsilon \epsilon_{\upsilon'}$ and
\begin{gather*}\pi_\tau(h) \mu_\upsilon \epsilon_{\upsilon'} =h(\sigma_{\upsilon \upsilon'}(\tau)) \epsilon_{\upsilon\upsilon'} =
 \mu_\upsilon \sigma_{\upsilon \upsilon'}(h)(\tau) \epsilon_{\upsilon'} = \mu_\upsilon \pi_\tau(\sigma_\upsilon(h)) \epsilon_{\upsilon'}.\end{gather*}
 The other relation is checked similarly.
\end{proof}

The same argument as in Lemma~\ref{innertime} then gives the following.

\begin{lem}\label{timeevSigmaSY}
The time evolution $\sigma_t(X)={\rm e}^{{\rm i}tH} X {\rm e}^{-{\rm i}tH}$ on $\cB(\ell^2(\Upsilon))$ generated by
the Hamiltonian $H$ of \eqref{HU} induces a time evolution on the algebra generated by
$\bar\Q(\Sigma_\cS)\rtimes \Upsilon$ and the projections $\Pi_{\underline{k},\underline{\ell}}$ of
Lemma~{\rm \ref{innertime}}. This time evolution acts as the identity on $\bar\Q(\Sigma_\cS)$.
\end{lem}

\begin{lem}\label{GactSigmaSY} Let $\cA_{\Q,\Sigma_\cS,\Upsilon}$ denote the algebra generated by
$\bar\Q(\Sigma_\cS)\rtimes \Upsilon$ and the projections $\Pi_{\underline{k},\underline{\ell}}$,
with the time evolution $\sigma_t$ generated by the Hamiltonian of~\eqref{HU}. The absolute
Galois group $G={\rm Gal}(\bar\Q/\Q)$ acts by symmetries of the dynamical system
$(\cA_{\Q,\Sigma_\cS,\Upsilon},\sigma_t)$.
\end{lem}

\begin{proof} The absolute Galois group $G={\rm Gal}\big(\bar\Q/\Q\big)$ acts on $\bar\Q(\Sigma_\cS)$
through the morphism $G \to G_\cS={\rm Gal}\big(\bar\Q(\Sigma_\cS)/\Q\big)$. Extending this
action by the trivial action on the generators $\mu_\upsilon$ and $\Pi_{\underline{k},\underline{\ell}}$,
we get an action on the algebra $\cA_{\Q,\Sigma_\cS,\Upsilon}$ with the property that
$\gamma \circ \sigma_t = \sigma_t \circ \gamma$, for all $\gamma\in G$ and all $t\in \R$.
\end{proof}

For $h\in \bar\Q(\Sigma_\cS)$ we consider as in Remark~\ref{hCrem} the associated analytic
function $h_\C$ on $(0,1)_\C$. Then, proceeding as in \cite{Ihara}, one can consider the
action of $G$ on $\Q\{\{ t \}\}$ and on $\bar\Q\{ \{ 1-t \}\}$ via the action on the Puiseux
coefficients. Given an element $h\in \bar\Q(\Sigma_\cS)$, it can be seen as an element in
$\Q\{\{ t \}\}$ or as an element in $\bar\Q\{ \{ 1-t \}\}$, since the function $h_\C$ can be
expanded in Puiseux series in $t$ or in $1-t$ with coefficients in $\bar\Q$. Given an
element $\gamma\in G$ one acts on~$h_\C$ with~$\gamma^{-1}$, through the action
on the Puiseux coefficients at $t=0$, then takes the expansion in~$1-t$ of the resulting element
and acts by~$\gamma$ on the Puiseux coefficients at $t=1$ of this function. The function
obtained in this way is then again expanded in~$t$. The transformation constructed in this
way is the element $f_\gamma$ in $\hat\pi_1\big(\P^1\smallsetminus \{0,1,\infty\},(0,1)\big)$
considered by Drinfeld and Ihara.

\subsection{Gibbs states}

We discuss here some functions associated to the evaluation of low temperature KMS
states (Gibbs states) of the time evolutions defined earlier in this section. We start with
a simpler case based on the quantum statistical mechanics of $\Q^*_+$ of \cite{MarXu},
and then we move to the dynamics considered in Lemma~\ref{timeevSigmaSY}.

\begin{lem}\label{GibbsN} Consider the Hilbert space $\ell^2(\N)$ with the Hamiltonian $H\epsilon_n =\log(n) \epsilon_n$. Let $h_\C$ be an analytic function in the region $(0,1)_\C$, which is bounded on the unit interval and with convergent Puiseux series $h_\C(\tau)=\sum_k a_k \tau^{k/N}$ for some
$N$, at all $\tau\in (0,1)$. Let $\pi_\tau(h)$ be the linear operator on $\ell^2(\N)$
defined by $\pi_\tau(h)\epsilon_n=\sigma_n(h_\C)(\tau)\epsilon_n = h_\C(\tau^n)\epsilon_n$. The
corresponding Gibbs state is given by
\begin{gather*}
\varphi_{\tau,\beta}(h):=\frac{\Tr\big( \pi_\tau(h) {\rm e}^{-\beta H} \big)}{\Tr\big({\rm e}^{-\beta H} \big)} =\frac{1}{\zeta(\beta)} \sum_k a_k {\rm Li}_\beta\big(\tau^{k/N}\big).
\end{gather*}
\end{lem}

\begin{proof} We have
\begin{gather*}
\varphi_{\tau,\beta}(h):=\frac{\Tr\big( \pi_\tau(h) {\rm e}^{-\beta H} \big)}{\Tr\big({\rm e}^{-\beta H} \big)} =\frac{1}{\zeta(\beta)}
\sum_{n\geq 1} h_\C(\tau^n) n^{-\beta}.
\end{gather*}
Under the assumption on $h_\C$, the series $\sum\limits_{k\in \Z,\, k\geq -M} a_k \tau^{kn/N}$ is absolutely
convergent at $\tau$ and we can write the above as
\begin{gather*}
\frac{1}{\zeta(\beta)} \sum_k a_k \sum_n \tau^{kn/N} n^{-\beta} = \frac{1}{\zeta(\beta)} \sum_k a_k {\rm Li}_\beta\big(\tau^{k/N}\big).
\end{gather*}
We are considering here the polylogarithm function ${\rm Li}_s(z)$ in the range where $z\in (0,1)$ and $s>1$, hence we have its integral description of the form
\begin{gather*}
{\rm Li}_s(z) = \frac{1}{2} z + z \int_0^\infty \frac{\sin(s \arctan t - t \log z)}{\big(1+t^2\big)^{s/2} \big({\rm e}^{2\pi t} -1\big)} \, {\rm d}t.
\end{gather*}
We can estimate the integral by
\begin{gather*}
\int_0^1 \frac{\sin(s \arctan t - t \log z)}{\big(1+t^2\big)^{s/2} \big({\rm e}^{2\pi t} -1\big)} \, {\rm d}t \leq \int_0^1 \frac{(s-\log z)}{\big(1+t^2\big)^{s/2} 2\pi} \, {\rm d}t ,
\\
\int_1^\infty \frac{\sin(s \arctan t - t \log z)}{\big(1+t^2\big)^{s/2} \big({\rm e}^{2\pi t} -1\big)} \, {\rm d}t \leq
 \int_1^\infty \frac{{\rm d}t}{\big(1+t^2\big)^{s/2} \big({\rm e}^{2\pi t} -1\big)}.
 \end{gather*}
 Thus, for some positive $A_\beta,B_\beta >0$, we obtain, in our range $\beta>1$ and $z\in (0,1)$, the estimate
 ${\rm Li}_\beta \big(\tau^{k/N}\big) \leq A_\beta \tau^{k/N} - B_\beta \tau^{k/N} \log \tau^{k/N}$. For $\tau\in (0,1)$,
 the entropy function $-\tau\log\tau$ is maximal at $\tau=1/e$ with value $1/e$. Thus, we further estimate
 \begin{gather*}
 - B_\beta \tau^{k/N} \log \tau^{k/N} = \frac{-B_\beta}{(1-\eta)} \tau^{k\eta/N} \tau^{k(1-\eta)/N} \log \tau^{k(1-\eta)/N}
 \leq \frac{B_\beta}{e (1-\eta)} \tau^{k\eta/N}.
 \end{gather*}
 If $0<\eta<1$ is chosen close to $1$ and so that $\tau^{\eta/N}$ is still within the domain of
 convergence of $\sum_k\! a_k \tau^{\eta k/N}$, then we obtain from this estimate the convergence of
 the series $\sum_k\! a_k {\rm Li}_\beta\big(\tau^{k/N}\big)$.
\end{proof}

We now consider the effect of extending the previous setting from $\ell^2(\N)$ to $\ell^2(\Q^*_+)$,
with the Hamiltonian of~\eqref{HamQ}.

\begin{lem}\label{GibbsQ}Consider the Hilbert space $\ell^2(\Q^*_+)$ with the Hamiltonian $H$ of~\eqref{HamQ}. Let $h_\C$ be as in Lemma~{\rm \ref{GibbsN}} and let $\pi_\tau(h)$ denote the linear operator on $\ell^2(\Q^*_+)$ defined by $\pi_\tau(h)\epsilon_r= h_\C(\tau^r) \epsilon_r$. The corresponding
Gibbs state is given by
\begin{gather*} \varphi_{\tau,\beta}(h):=\frac{\Tr\big( \pi_\tau(h) {\rm e}^{-\beta H} \big)}{\Tr\big({\rm e}^{-\beta H} \big)} =
\frac{\zeta(2\beta)}{\zeta(\beta)^2} \sum_k a_k {\rm Li}_\beta\big(\tau^{k/N}\big). \end{gather*}
\end{lem}

\begin{proof} We have
\begin{gather*}
\varphi_{\tau,\beta}(h):=\frac{\Tr\big( \pi_\tau(h) {\rm e}^{-\beta H} \big)}{\Tr\big({\rm e}^{-\beta H} \big)}
=\frac{\zeta(2\beta)}{\zeta(\beta)^2} \sum_n \sum_{j=1}^{2^\omega(n)} h_\C(\tau^{r_j}) n^{-\beta},
\end{gather*}
where $\omega(n)$ is the number of pairwise distinct prime factors of $n=p_1^{k_1} \cdots p_{\omega(n)}^{k_{\omega(n)}}$ and the $r_j$ run over all the $2^{\omega(n)}$ possible choices of $\pm$ signs in $p_1^{\pm k_1} \cdots p_{\omega(n)}^{\pm k_{\omega(n)}}$.
By the assumptions on $h_\C$, the Puiseux series $\sum_k a_k \tau^{r k/N}$ is absolutely convergent at~$\tau$,
hence we can write the summation in the expression above as
\begin{gather*}
 \sum_k a_k \sum_n \sum_{j=1}^{2^\omega(n)} \tau^{r_j k/N} n^{-\beta}.
\end{gather*}
We have $\min \{ r_j\, |\, j=1,\ldots, \omega(n) \}=1/n$, and $\tau^{r_j k/N}\leq \tau^{\frac{k}{nN}}$ for
 all $j=1,\ldots, \omega(n)$, hence we can give an upper bound for this
 \begin{gather*}
 \sum_k a_k \sum_n \frac{2^{\omega(n)} \tau^{\frac{k}{nN}}}{n^\beta} .
 \end{gather*}
Where the absolute convergence holds we have
\begin{gather*} \sum_n \frac{2^{\omega(n)} \tau^{k/(nN)}}{n^\beta} = \sum_{n\geq 1} 2^{\omega(n)} \sum_{\ell \geq 0}
 \frac{ \log^\ell \big(\tau^{k/N}\big) }{\ell ! n^{\beta +\ell}}
 = \sum_{\ell \geq 0} \frac{ \log^\ell \big(\tau^{k/N}\big) }{\ell !} \frac{\zeta(\beta+\ell)^2}{\zeta(2(\beta +\ell))}. \end{gather*}
Since $\zeta(s)^2/\zeta(2s)$ is decreasing for a real variable $s>1$, the convergence of the series above
is controlled by the convergence of the Puiseux series $\sum_k a_k \tau^{k/N}$.
\end{proof}

The more general case is less explicit, but the case above serves as a model for the expected
behavior of the Gibbs states.
On the Hilbert space $\ell^2(\Upsilon)$ with the Hamiltonian $H$ of \eqref{HU} and
the operators $\pi_\tau(h)$ of \eqref{pitauh}, the Gibbs states have the form
\begin{gather*} \varphi_{\tau,\beta}(h)=\frac{\Tr\big(\pi_\tau(h) {\rm e}^{-\beta H}\big)}{\Tr\big({\rm e}^{-\beta H}\big)} = \sum_{\upsilon\in \Upsilon} \sigma_\upsilon(h)(\tau) n(\upsilon)^{-\beta}, \end{gather*}
where for $\upsilon=\upsilon(\epsilon_0,\epsilon_1,\underline(r))$ with
$\underline{r}=(r_1,\ldots, r_k)$ we write $n(\upsilon):=n(r_1)\cdots n(r_k)$ with
$n(r_j)$ as in~\eqref{nr}. As before, for a fixed positive integer $n$, we have
$\min\{ r \,|\, n(r)=n \}=1/n$ and $\max\{ r \,|\, n(r)=n \}=n$. We have
$\sigma_\upsilon(h)=\tau^{\epsilon_0}\sigma_{r_1}\tau\cdots \tau\sigma_{r_k}\tau^{\epsilon_1}(h)$,
where $\sigma_{r_j}(t)=t^{r_j}$. Using estimates of terms of the form $(1-\tau^r)^{r'}$ by terms
$\big(1-\tau^{n(r)}\big)^{1/n(r')}$ we can estimate the convergence of the series above
in terms of a series
\begin{gather*} h(\tau)+ \sum_k \sum_{n=n_1\cdots n_k} 2^{\omega(n_1)+\cdots + \omega(n_k)} h(P_{n_1,\ldots,n_k}(\tau)),
n^{-\beta}, \end{gather*}
where $P_{n_1,\ldots,n_k}(\tau)$ is the expression obtained as above, majorizing
$\sigma_\upsilon(h)(\tau)$ with the $\sigma_{r_j}$ replaced by either $\sigma_{n_j}$ or
$\sigma_{1/n_j}$. One can then, in principle, obtain convergence estimates
in terms of Puiseux series in~$t$ and~$1-t$.

\subsection{The group ${\bf mGT}$ and the symmetries of genus zero modular operad}\label{section3.10}

In this subsection we show very briefly, following \cite{ComMan} and~\cite{ComMan2}, how the
Grothendieck--Teichm\"uller group ${\bf mGT}$ acts upon the family of modular spaces of genus zero curves with marked points, compatibly with its operadic structure.

Let $S$ be a finite set of cardinality $|S| \ge 3$. The smooth manifold $\overline{M}_{0,S}$
is the moduli space of stable curves of genus zero with~$|S|$ pairwise distinct non-singular points
on it bijectively marked by $S$. The group of permutations of~$S$ acts by smooth
automorphisms (re-markings) upon $\overline{M}_{0,S}$ and in fact, is its complete
automorphism group.

The group ${\bf mGT}$ appears naturally if we restrict ourselves (as we will do)
by finite subsets of roots unity in~$\C^*$. The operadic morphisms in which~$S$
can be varied in a controlled way and are only subsets of roots of unity,
are also compatible with re-markings defining the family of
all componentwise automorphisms groups.

For further details, see \cite{ComMan2}.

In \cite{LieManMar} and \cite{ManMar1} it was shown that the endomorphisms $\sigma_n$
and $\rho_n$ of the Bost--Connes algebra lift to various equivariant Grothendieck rings
and further to the level of assemblers and homotopy theoretic spectra. The symmetries of the Bost--Connes
system, given by $\hat\Z^*$ in the original version, or by ${\bf mGT}$ according to
Proposition~\ref{thetanlem} above, also can be lifted to these categorical and homotopy theoretic
levels in a similar manner. The relation between the group ${\bf mGT}$ and the symmetries of
the modular operad suggests that the same type of Bost--Connes structure, consisting of
endmorphisms $\sigma_n$ with partial inverse $\rho_n$ and symmetries given by ${\bf mGT}$
may also have a possible lift at the level of the modular operad. This question remains to
be investigated.

\section{Quasi-triangular quasi-Hopf algebras}\label{section4}

In this section we consider a different point of view, based on quasi-triangular quasi-Hopf algebras.
The results of Drinfeld \cite{Dr} showed that the Grothendieck--Teichm\"uller group $GT$ acts
by tranformations of the structure (associator and $R$-matrix) of quasi-triangular quasi-Hopf algebras,
hence the absolute Galois group also acts via its embedding into $GT$. We show here that there
are systems of quasi-triangular quasi-Hopf algebras naturally associated to the constructions we
discussed in the previous section. In particular, we show that the Bost--Connes endomorphisms
$\sigma_n$, $\rho_n$ on the group algebra $\C[\Q/\Z]$ determine a system of quasi-triangular
quasi-Hopf algebras. We also show that there is a system of quasi-triangular
quasi-Hopf algebras associated to our previous construction of the Hopf algebra of dessins d'enfant.

\subsection{Twisted quantum double and quasi-triangular quasi-Hopf algebras}

Recall the following notions from~\cite{Dri3}. A quasi-Hopf algebra $\cH$ over a field $\K$ is a unital associative algebra over $\K$ with multiplication $m\colon \cH\otimes \cH \to \cH$, endowed with a counit $\epsilon\colon \cH \to \K$, a~comultiplication $\Delta\colon \cH \to \cH\otimes \cH$, an invertible element $\Phi\in \cH \otimes \cH \otimes \cH$ (the associator), as well as an antihomomorphism
$S\colon \cH \to \cH$ and two distinguished elements $\alpha,\beta\in \cH$ such that
\begin{enumerate}\itemsep=0pt
\item $\Phi$ satisfies the pentagon identity
\begin{gather*} (\Delta\otimes 1 \otimes 1) \Phi (1\otimes 1 \otimes \Delta)\Phi =
(\Phi\otimes 1) (1\otimes \Delta\otimes 1)\Phi (1\otimes \Phi). \end{gather*}
\item $\Delta$ is quasi-associative
\begin{gather*} (1\otimes \Delta) \Delta(x)=\Phi^{-1} (\Delta\otimes 1) \Delta(x) \Phi , \ \ \forall x\in \cH. \end{gather*}
\item $\Phi$ and the counit $\epsilon$ satisfy
\begin{gather*} (1\otimes \epsilon \otimes 1) \Phi =1. \end{gather*}
\item The antipode $S$ satisfies the relations
\begin{itemize}\itemsep=0pt
\item $m (S\otimes \alpha) \Delta(x)=\epsilon(x) \alpha$ and $m (1\otimes \beta S) \Delta(x)=\epsilon(x) \beta$, for all $c\in \cH$,
\item $m(m\otimes 1) (S\otimes \alpha\otimes \beta S)\Phi=1$ and $m(m\otimes 1) (1\otimes \beta S \otimes \alpha) \Phi^{-1}=1$.
\end{itemize}
\end{enumerate}

The structure of quasi-triangular quasi-Hopf algebra upon $\cH$ is given by an invertible element $\cR \in \cH \otimes \cH$ such that $\Delta^t(x)=\cR \Delta(x) \cR^{-1}$, where $\Delta^t$ is the coproduct $\Delta$ with swapped order of the two factors in $\cH\otimes \cH$.

The Drinfeld quantum double construction of \cite{Dri2} assigns to a finite dimensional Hopf algebra~$\cH$
a quasi-triangular Hopf algebra $\cD(\cH)=\cH\otimes \cH^*$, where $\cH^*$ is the linear dual of~$\cH$
with the Hopf algebra structure with coproduct corresponding to the product of~$\cH$ and viceversa. The
quasi-triangular structure is determined by the element $\cR$ obtained by considering the image of the
identity in $\Hom(\cH,\cH)$ under the identification of the latter with $\cH \otimes \cH^*$,
\begin{gather*} \cR =\sum_a e^a \otimes 1 \otimes 1 \otimes e_a, \end{gather*}
where $\{e_a \}$ is a basis of $\cH$ and $\{ e^a \}$ is the dual basis.

The construction of the Drinfeld quantum double can be extended from Hopf algebras
to quasi-Hopf algebras, in such a way that it assigns to a quasi-Hopf algebra a
quasi-triangular quasi-Hopf algebra. This can be done in a categorical way, as shown
in \cite{Majid}. Namely, the Drinfeld quantum double construction has a categorical
formulation (see \cite{Dri2}) by considering the rigid monoidal category $\cM_\cH$ of
modules over a~Hopf algebra $\cH$ and associating to it a~braided tensor category
$\cD(\cM_\cH)$ which can be identified with the category of modules $\cM_{\cD(\cH)}$
over the Drinfeld quantum double $\cD(\cH)$ quasi-triangular Hopf algebra.
Since by the Tannaka reconstruction theorem for Hopf algebras, there is a bijection
between Hopf algebras and rigid monoidal categories with a forgetful functor to
finite dimensional vector spaces, the categorical formulation of the Drinfeld quantum
double can be seen as a construction that assigns to a rigid monoidal category $\cC$
a braided tensor category $\cD(\cC)$ called the Drinfeld double or the Drinfeld center.
Using Tannaka reconstruction, $\cC$ determines a Hopf algebra and $\cD(\cH)$
determines its double quasi-triangular Hopf algebras. It is shown in \cite{Majid} that the
same kind of categorical argument can be applied starting with a quasi-Hopf algebra $\cH$,
by considering its category of modules and then obtaining from the Drinfeld center
construction $\cD(\cM_{\cH})$ a quasi-triangular quasi-Hopf algebra, $\cD(\cH)$, the twisted Drinfeld double, by applying the Tannaka reconstruction theorem for quasi-Hopf algebras proven in~\cite{Majid2}.

The categorical formulation of Drinfeld quantum doubles can also be used to show functoriality
properties. Indeed, the functoriality of the Drinfeld center was proved in \cite{KoZhe}, as a~functor
with source a category whose objects are tensor categories and whose
morphisms are bi\-modules and target a category whose objects are braid tensor
categories with morphisms given by monoidal categories obtained from a suitable tensoring
operation of the source and target braid tensor categories, see~\cite{KoZhe} for more details.

In particular, we are interested here in the twisted Drinfeld double~$\cD^\omega(G)$
of a finite group $G$. It is a quasi-triangular
quasi-Hopf algebra, obtained by applying the Drinfeld quantum double construction of
\cite{Dri2} to the quasi-Hopf algebra obtained by twisting the Hopf algebra
associated to $G$ by a $3$-cocycle $\omega\in H^3(G,\bG_m)$.
More precisely, the group algebra $\C[G]$ of $G$ has a Hopf algebra
structure with the coproduct of group-like elements given by
$\Delta(g)=g\otimes g$, and product given by the convolution product of the group algebra,
while $\C^G$ has coproduct given by the convolution product $\Delta(e^g)= \sum\limits_{g=ab} e^a\otimes e^b$
and product given by the pointwise product of functions, $e^g e^h =\delta_{g,h} e^g$.
One considers then the product~$\C^G\otimes \C[G]$ with basis~$e_g$ of~$\C[G]$ and~$e^g$ of~$\C^G$.
A $3$-cocycle $\omega\in Z^3(G,U(1)$ for a finite group $G$ satisfies the cocycle identity
\begin{gather*} \omega(y,s,t)\omega(x,ys,t)\omega(x,y,s)=\omega(s,y,st)\omega(xy,s,t) \end{gather*}
and $\omega(x,e,y)=1$ for $e$ the identity element in $G$. The choice of a
$3$-cocycle $\omega$ has the effect of twisting the Hopf algebra $\C^G$ into
a quasi-Hopf algebra $\C^G_\omega$, with associator $\Phi$ given by
$\Phi=\sum\limits_{a,b,c \in G} \omega(a,b,c)^{-1} e^a \otimes e^b \otimes e^c$.
The twisted quantum double $\cD^\omega(G)=\C^G_\omega \otimes \C[G]$ is
the extension $\C^G_\omega \to \cD^\omega(G) \to \C[G]$ with the first map
given by $e^g \mapsto e^g\otimes 1$ and the second by $e^g\otimes e_{g'} \mapsto \delta_{g,1} e_{g'}$.
The quasi-Hopf algebra structure of $\cD^\omega(G)$ is determined by the product and coproduct
\begin{gather*} (e^g\otimes e_h) ({\rm e}^{g'}\otimes e_{h'})= \theta_g(h,h') \delta_{h^{-1} g h, g'} e^g \otimes hh', \\
\Delta(e^g \otimes e_h) =\sum_{g=ab} \gamma_h(a,b) e^a\otimes e_h \otimes e^b \otimes e_h \end{gather*}
where $\theta_g(h,h')$ and $\gamma_h(a,b)$ are given by
\begin{gather*} \theta_g(h,h') = \frac{\omega(g,h,h') \omega\big(h,h', (hh')^{-1} g hh'\big)}{\omega\big(h, h^{-1} g h, h'\big)}, \\
 \gamma_h(a,b)= \frac{\omega(a,b,h) \omega\big(h,h^{-1}ah,h^{-1}bh\big)}{\omega\big(a,h,h^{-1}bh\big)}.
 \end{gather*}
The quasi-triangular structure on $\cD^\omega(G)$ is given by
\begin{gather*} \cR = \sum_{a\in G} (e^a \otimes 1) \otimes (1\otimes e_a). \end{gather*}

This construction of a quasi-triangular quasi-Hopf algebra associated to a finite group
was introduced in~\cite{PDR} in the context of RCFT orbifold models and
$3$-dimensional topological field theory. The compatibility between the direct
construction described in~\cite{PDR} and the general categorical formulation of the
twisted Drinfeld double mentioned above is discussed in detail in~\cite{Majid}.

\subsection{A Bost--Connes system of quasi-triangular quasi-Hopf algebras}

We show here that there is a direct system of quasi-triangular quasi-Hopf algebras
naturally associated to the Bost--Connes datum, by which we mean the datum
of the algebra $\C[\Q/\Z]$ together with the endomorphisms $\sigma_n(e(r))=e(nr)$
and $\rho_n(e(r))=n^{-1} \sum\limits_{s\colon ns=r} e(s)$. We also write $\tilde\rho_n(e(r))=\sum\limits_{s\colon ns=r} e(s)$
for the correspondences considered in the integral Bost--Connes system of~\cite{CCM2}.

Consider the tensor product $\C[\Q/\Z]\otimes \C^{\Q/\Z}$.
Using the Bost--Connes notation we write $e(r)$ for the basis of $\C[\Q/\Z]$ for $r\in \Q/\Z$
and we write $e_r$ for the dual basis of $\C^{\Q/\Z}$. In particular, for a fixed level $n\in \N$,
we consider the tensor product $\C[\Z/n\Z] \otimes \C^{\Z/n\Z}$ and we still use the same
notation~$e(r)$ and~$e_r$ for the dual bases.

The choice of a $3$-cocycle $\omega$ for $\Z/n\Z$ is given by the choice of a representative in
\[ H^3(\Z/n\Z,U(1))\simeq \Z/n\Z.\] An explicit description of the representatives is given
in \cite[Proposition~2.3]{HLY} as
\begin{gather*} \omega_a(g_{i_1}, g_{i_2}, g_{i_3}) = \zeta_m^{a i_1 [(i_2+i_3)/3]}, \qquad 0\leq a \leq m, \end{gather*}
for $\zeta_m$ a primitive root of unity of order~$m$. The twisted Drinfeld double $\cD^\omega(\Z/n\Z)$ of a finite cyclic group, with $\omega$ chosen as one of the representatives of $H^3(\Z/n\Z,U(1))\simeq \Z/n\Z$ above, is then obtained as recalled in the previous subsection.

\begin{prop}\label{BCqtqH}The Bost--Connes endomorphisms $\sigma_n$ and $\rho_n$ determine a direct system of quasi-triangular quasi-Hopf algebras $\cD^{\omega_n}(\Z/n\Z)$ indexed by the positive integers $n\in \N$ ordered by divisibility, with associators related by $\omega_{nm}=\sigma_n(\omega_m)$
and $R$-matrices related by $\cR_{nm}=\tilde\rho_n(\cR_m)$.
\end{prop}

\begin{proof} The Bost--Connes maps $\sigma_n\colon \Z/nm\Z \twoheadrightarrow \Z/m\Z$ map roots of unity of order $nm$ to their $n$-th power. They organize the roots of unity into a projective system with $\varprojlim_n \Z/n\Z=\hat \Z$. By identifying $X_n=\Spec(\C[\Z/n\Z])$ with $\Z/n\Z$, the maps above induce morphisms $\sigma_n\colon \C[\Z/m\Z] \to \C[\Z/nm\Z]$ that determine the limit algebra $\C[\Q/\Z]$ as a direct limit of the $\C[\Z/n\Z]$. On~$\C^{\Z/n\Z}$ the morphisms $\sigma_n\colon \Z/nm\Z \twoheadrightarrow \Z/m\Z$ act by precomposition, $\sigma_n\colon \C^{\Z/m\Z}\to \C^{\Z/nm\Z}$ with $f\mapsto \sigma_n(f)=f\circ \sigma_n$. The group cohomology is a contravariant functor with respect to group homomorphisms, with the restriction map $\sigma_n\colon H^3(\Z/m\Z,U(1))\to H^3(\Z/nm\Z,U(1))$.
Consider then the $R$-matrix of the quasi-triangular structure, given by the element $\cR_m=\sum\limits_{a\in \Z/m\Z} (e_a \otimes 1) \otimes (1\otimes e(a))$. The image of $\cR_m$ under the map $\tilde\rho_n$ is given by
\begin{gather*}
 \tilde\rho_n(\cR_m)=\sum_{a\in \Z/m\Z} \sum_{b\colon nb=a} (e_b \otimes 1) \otimes (1 \otimes e(b)) =
\sum_{b\in \Z/nm\Z} (e_b \otimes 1) \otimes (1 \otimes e(b)) =\cR_{nm}.\tag*{\qed}
\end{gather*}\renewcommand{\qed}{}
\end{proof}

\subsection{Dessins d'enfant and quasi-triangular quasi-Hopf algebras}

We now consider a similar construction of a system of quasi-triangular quasi-Hopf algebras associated to the Hopf algebra of dessins d'enfant of Proposition~\ref{HopfDessins}.

\begin{prop}\label{HDqtqH} The Hopf algebra of dessins d'enfant of Proposition~{\rm \ref{HopfDessins}} determines an associated system of quasi-triangular quasi-Hopf algebras $\cD^{\omega_d}(\cG_d)$.
\end{prop}

\begin{proof} The construction of the quasi-triangular quasi-Hopf algebras $\cD^{\omega_d}(\cG_d)$
is similar to the one we have seen in Proposition~\ref{BCqtqH} associated to the Bost--Connes
datum, and relies in a~simi\-lar way on the twisted Drinfeld quantum double. By construction, the Hopf algebra $\cH_\cD$ of dessins d'enfant of Proposition~\ref{HopfDessins} is a connected commutative Hopf algebra, with grading as in Lemma~\ref{HDgrading} related to the degree of the Belyi maps, $\cH_\cD=\oplus_{d\geq 0} \cH_{\cD,d}$ with $\cH_{\cD,0}=\Q$. Thus, the dual affine group scheme $\cG$ is a pro-unipotent affine group scheme, $\cG =\varprojlim_d \cG_d$. We work here with complex coefficients, with $\cH_{\cD,\C}=\cH_\cD \otimes_\Q \C$. For a fixed $d\in \N$
consider the product $\cH_{\cD,\C,d} \otimes \C[\cG_d]$. We can use this to construct a~twisted Drinfeld quantum double $\cD^{\omega_d}(\cG_d)$, given the choice of a $3$-cocycle
$\omega\in H^3(\cG_d,U(1))$. The functoriality of the Drinfeld quantum double via the functoriality of the Drinfeld center \cite{KoZhe}, together with the functoriality of the group cohomology and the categorical construction of the twisted Drinfeld quantum double of~\cite{Majid} then show that the projective system of group homomorphisms between the $\cG_d$ and corresponding dual direct system of Hopf algebras $\cH_{\cD,\C,d}$ induce a system of quasi-triangular quasi-Hopf algebras $\cD^{\omega_d}(\cG_d)$.
\end{proof}

We then have two actions of the absolute Galois group $G={\rm Gal}\big(\bar\Q/\Q\big)$ on the
quasi-triangular quasi-Hopf algebras $\cD^{\omega_d}(\cG_d)$. On the one hand,
the action of $G$ by Hopf algebra automorphisms of $\cH_\cD$ restricts to an action
on the $\cH_{\cD,d}$ since the degree is a Galois invariant, hence~$G$ acts by
automorphisms of $\cG_d$. On the other hand we also have the embedding of~$G$
into the Grothendieck--Teichm\"uller group $GT$, which acts on the quasi-triangulated
quasi-Hopf structure of the $\cD^{\omega_d}(\cG_d)$, by transforming the pair $(\Phi,\cR)$
of the associator and the $R$-matrix as in~\cite{Dr}.

\subsection*{Acknowledgments} The second named author is partially supported by
NSF grant DMS-1707882, and by NSERC Discovery Grant RGPIN-2018-04937
and Accelerator Supplement grant RGPAS-2018-522593. We thank the anonymous
referees for several very useful comments that significantly improve the paper,
and Lieven Le Bruyn for his suggestions in a series of mails that helped us to avoid ambiguities.

\pdfbookmark[1]{References}{ref}
\LastPageEnding

\end{document}